\documentclass{amsart}

\usepackage{gensymb,epic,eepic,float,fullpage}
\usepackage{longtable}

\usepackage{amsmath,amsfonts,amssymb,amsthm}
\usepackage{mathtools}
\usepackage[colorinlistoftodos]{todonotes}
\usepackage{cancel}
\usepackage{stmaryrd}
\usepackage{url}
\usepackage{mathrsfs}
\usepackage{bbm}
\usepackage{bm}
\usepackage{tikz}
\usepackage[outline]{contour}
\usetikzlibrary{graphs,patterns,decorations.markings,arrows, arrows.meta,matrix}
\usetikzlibrary{calc,decorations.pathmorphing,decorations.pathreplacing,shapes}
\usepackage{graphicx}

\usepackage{hyperref}

\numberwithin{equation}{section}

\newcommand\restr[2]{{
  \left.\kern-\nulldelimiterspace 
  #1 
  \vphantom{\big|} 
  \right|_{#2} 
 }}

\newtheorem{prop}{Proposition}[section]
\newtheorem{theo}[prop]{Theorem}
\newtheorem{lem}[prop]{Lemma}
\newtheorem{cor}[prop]{Corollary}
\theoremstyle{definition}
\newtheorem{rem}[prop]{Remark}

 \colorlet{lgray}{white!85!black}
 \colorlet{dgray}{white!45!black}
\colorlet{lred}{white!85!red}
\colorlet{dred}{white!35!red}
\colorlet{lgreen}{white!60!green}
\colorlet{dgreen}{black!30!green}
\colorlet{lpurple}{white!60!purple}
\colorlet{lblue}{white!60!blue}
\definecolor{green}{rgb}{0.1,0.8,0.1}
\definecolor{yellow}{rgb}{1.0,0.85,0.25}
\definecolor{purple}{rgb}{1.0, 0, 1.0}
\definecolor{blue}{rgb}{0, 0, 1.0}

\tikzstyle{unfused}=[line width=1.5pt, draw=lgray, arrows={->[line width=1.5pt, length = 2mm, width=2mm,dgray]}]
\tikzstyle{unfused*}=[lgray, line width=1.5pt]
\tikzstyle{fused}=[lgray, line width=4pt, arrows={->[line width=4pt, length = 4mm, width=4mm,dgray]}]
\tikzstyle{fused*}=[lgray, line width=4pt]
\tikzstyle{cont}=[lred, line width=4pt, arrows={->[line width=4pt, length = 4mm, width=4mm,dred]}]
\tikzstyle{dual}=[black, line width=1pt, dashed]
\tikzstyle{lightdual}=[black, line width=0.5pt, dashed]
\tikzstyle{cut}=[black, line width=1.0pt]

 \renewcommand{\tikz}[2]{
\begin{tikzpicture}[scale=#1,baseline=(current bounding box.center),>=Stealth]
#2
\end{tikzpicture}}

\newcommand{\tikzbase}[3]{
\begin{tikzpicture}[scale=#1,baseline={([yshift=#2]current bounding box.center)},>=Stealth]
#3
\end{tikzpicture}}

\DeclareMathOperator{\End}{End}
\DeclareMathOperator{\Span}{Span}

\def \be{\begin{equation*}}
\def \ee{\end{equation*}}
\def\({\left(}
\def\){\right)}
\def\[{\left[}
\def\]{\right]}

\def \1{\mathbbm 1}

\def \i{\mathbf{i}}

\def \Z{\mathbb Z}
\def \W{\mathcal W}
\def \R{\mathcal R}
\def \V{\mathbb V}
\def \S {{\sf S}}
\def \s {{\sf s}}
\def \FF {{\sf \widetilde{F}}}
\def \ZW{{ZW}}
\def \Zw{{Zw}}
\def \T{\mathbb T}
\def \C{\mathbb C}
\def \cc{{\sf c}}
\def \B{\mathbb B}
\def \tB{\widetilde{B}}
\def \F{\mathbb F}

\begin{document}
	
\title{Inhomogeneous spin $q$-Whittaker polynomials}

\begin{abstract}
We introduce and study an inhomogeneous generalization of the spin $q$-Whittaker polynomials 
from \cite{BW17}. These are symmetric polynomials, and we prove a branching rule, skew dual
and non-dual Cauchy identities, and an integral representation for them. Our main tool is a 
novel family of deformed Yang-Baxter equations.  
\end{abstract}

\author{Alexei Borodin and Sergei Korotkikh}

\maketitle

\tableofcontents

\section{Introduction}

\subsection*{Background} Connections between integrable lattice models and the theory of symmetric functions have been long known to be important to both domains; one might, for example, consult the following works (the list is certainly incomplete): Fomin-Kirillov \cite{FK93}, Tsilevich \cite{Tsi05}, Brubaker-Bump-Friedberg \cite{BBS09}, Zinn-Justin \cite{ZJ09}, Korff \cite{Kor11}, 
Wheeler-Zinn-Justin \cite{WZJ15, WZJ16}, Motegi-Sakai \cite{MS13}, Motegi \cite{Mot16}, Cantini-de Gier-Wheeler \cite{CdGW15}, Garbali-de Gier-Wheeler \cite{GdGW16}, Borodin \cite{Bor14, Bor17}, Borodin-Petrov \cite{BP16a, BP16b}, Borodin-Wheeler \cite{BW17, BW19}, Garbali-Wheeler \cite{GW18}, Buciumas-Scrimshaw \cite{BS20}, Mucciconi-Petrov \cite{MP20}, Petrov \cite{Pet20}.

In particular, in \cite{Bor14, BP16a, BP16b, BW17} families of \emph{spin} Hall-Littlewood and $q$-Whittaker symmetric functions were introduced. They generalized the classical families of the Hall-Littewood and $q$-Whittaker symmetric polynomials (both of which are special cases of the Macdonald symmetric functions, see Macdonald \cite{Mac95}), and the additional \emph{spin parameter} was related to the spin in the underlying integrable vertex models. It was shown, in particular, that symmetric functions from those families satisfy explicit branching rules as well as (skew) Cauchy and dual Cauchy identities.
A different version of the spin $q$-Whittaker polynomials was also recently suggested in \cite{MP20}; it appears to enjoy a few valuable properties that the original one lacked. 

The introduction of the spin Hall-Littlewood functions was prompted by their usefulness in asymptotic analysis of a variety of probabilistic systems that include certain integrable models of directed polymers in random media and exclusions processes in (1+1)-dimensions, see Borodin-Corwin-Petrov-Sasamoto \cite{BCPS13, BCPS14}, where they arose as an orthogonal basis of eigenfunctions for Hamiltonians of certain associated (`dual') quantum integrable systems. Probabilistic applications of the spin q-Whittaker polynomials were demonstrated by Bufetov-Mucciconi-Petrov \cite{BMP19} and Mucciconi-Petrov \cite{MP20}. 

In \cite{BP16b} an inhomogeneous version of the spin Hall-Littlewood functions was introduced. The inhomogeneities appear naturally from the point of view of the vertex models involved, and they are also crucial for some of the probabilistic applications, such as the analysis of the Asymmetric Simple Exclusion Process (ASEP) and the stochastic six vertex model with varying initial conditions, see Aggarwal-Borodin \cite{AB16} and Aggarwal \cite{Agg16}, and large time asymptotics of the exponential jump model of Borodin-Petrov \cite{BP17}. Curiously, while the inhomogeneities significantly change the symmetric functions, they do not affect the product factors of the (skew) Cauchy identities. 
 
The main goal of the present work is to define and study inhomogeneous spin $q$-Whittaker polynomials. 
So far it has been unclear how to do that. Even though the approach of \cite{BW17}, that allowed to obtain the homogeneous spin $q$-Whittaker polynomials from the Hall-Littlewood ones via the procedure of \emph{fusion}, is also applicable in the inhomogeneous case, the resulting functions do not appear to satisfy the dual Cauchy identity involving both $q$-Whittaker and Hall-Littlewood functions, which is 
the most natural requirement. Indeed, we show that the `correct' inhomogeneous spin $q$-Whittaker polynomials (those that do satisfy the dual Cauchy identity with the inhomogeneous spin Hall-Littlewood functions) are different from the functions naively obtained via fusion. These two families (the `naive' and the `correct' one) can be united in a (non-dual) Cauchy identity, however. 

With equal inhomogeneities our inhomogeneous polynomials degenerate to those of \cite{BW17}. It remains to be seen if the modified definitions of \cite{MP20} can be applied in the inhomogeneous setting, and if this would again lead to more desirable properties. 

In the next subsection we list the properties of these new inhomogeneous symmetric polynomials that we were able to prove. 
The key new ingredient of our approach is a \emph{deformed Yang-Baxter equation} for the vertex weights of the underlying integrable vertex model; it is also discussed below.

\subsection*{Main results.} For a pair of sequences of complex parameters $\S=(s_0, s_1, s_2, \dots)$ and $\Xi=(\xi_0, \xi_1,\xi_2,\dots)$ we define functions $\F_{\lambda/\mu}(\kappa_1, \dots, \kappa_n\mid\Xi,\S)$, which depend on $n$ variables $\kappa_1, \dots, \kappa_n$ and are labeled by pairs of partitions $\lambda,\mu$. Here are the main facts we prove about them. 
\begin{itemize}
\item The functions $\F_{\lambda/\mu}(\kappa_1, \dots, \kappa_n\mid\Xi,\S)$ are symmetric polynomials in $\kappa_1,\dots, \kappa_n$.

\item For a single variable, $\F_{\lambda/\mu}(\kappa\mid\Xi,\S)$ has a fully factorized expression. To write it down, we first define an auxiliary function $ZW_{\lambda/\mu}(\kappa\mid\Xi,\S)$ as  a partition function of a one-row vertex model of the form

\medskip
	
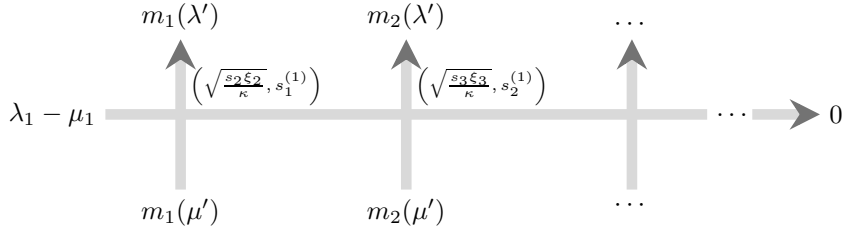
\tikz{1}{
		\foreach\y in {2}{
			\draw[fused*] (1,\y) -- (9,\y);
			\node[right] at (9,\y) {$\dots$};
			\draw[fused] (9.6,\y) -- (10.5,\y);
		}
		\foreach\x in {1,...,3}{
			\draw[fused] (3*\x-1,1) -- (3*\x-1,3);
		}
		\node[above right] at (2,2) {\tiny $\(\sqrt{\frac{s_2\xi_2}{\kappa}}, s^{(1)}_1\)$};
		\node[above right] at (5,2) {\tiny $\(\sqrt{\frac{s_3\xi_3}{\kappa}}, s^{(1)}_2\)$};
		
		\node[above] at (8,3) {$\cdots$};
		\node[above] at (5,3) {$m_2(\lambda')$};
		\node[above] at (2,3) {$m_1(\lambda')$};
		\node[left] at (1,2) {$\lambda_1-\mu_1$};
		\node[below] at (8,1) {$\cdots$};
		\node[below] at (5,1) {$m_2(\mu')$};
		\node[below] at (2,1) {$m_1(\mu')$};
		\node[right] at (10.5,2) {$0$};
	}
\medskip\\
This partition function is the product of weights of all the vertices in this row, where each vertex has four nonnegative integral labels on the edges incident to it; the labels on the vertical edges are the multiplicities in $\lambda'$ and $\mu'$: for $j\ge 1$, $m_j(\mu')=\mu_j-\mu_{j+1}$, $m_j(\lambda')=\lambda_j-\lambda_{j+1}$;
for any $k\ge 1$ we defined $s^{(1)}_k:=\sqrt{s_k s_{k+1}\xi_{k+1}/\xi_k}$, and the weights of the vertices are given by (note that all the square roots will disappear upon substitution into these weights) 
\begin{equation*}
	\tikz{1}{
		\draw[fused] (-1,0) -- (1,0);
		\draw[fused] (0,-1) -- (0,1);
		\node[left] at (-1,0) {\tiny $j$};\node[right] at (1,0) {\tiny $l$};
		\node[below] at (0,-1) {\tiny $i$};\node[above] at (0,1) {\tiny $k$};
		\node[above right] at (0,0.1) {\tiny $(t,s)$};
	}
	=\1_{i+j=k+l}\ \1_{i\geq l}\ s^{2l}t^{-2l}\frac{(s^2/t^2;q)_{i-l}(t^2;q)_l}{(s^2;q)_i}\frac{(q;q)_i}{(q;q)_{i-l}(q;q)_l}.
\end{equation*}
Even though the row is semi-infinite, the number of vertices of weight $\ne 1$ is readily seen to be finite. 

Then the polynomial $\F_{\lambda/\mu}(\kappa\mid\Xi,\S)$ is defined by the following expression:
\begin{equation*}
	\F_{\lambda/\mu}(\kappa\mid\Xi,\S):=\frac{(-\tau_\Xi\S)^{\mu}}{(-\S)^{\lambda}}\frac{\cc_{\tau_{\Xi}\S}(\mu)}{\cc_\S(\lambda)}\,\({\kappa s_0/\xi_0}\)^{\lambda_1-\mu_1}\frac{(\kappa^{-1} s_1 \xi_1;q)_{\lambda_1-\mu_1}}{(q;q)_{\lambda_1-\mu_1}}\,ZW_{\lambda/\mu}(\kappa\mid\Xi,\S),
\end{equation*}
where $\tau_\Xi\S=(s_1^{(1)}, s_2^{(1)},\dots)$, and
\begin{equation*}
\cc_\S(\lambda):=\prod_{i\geq1}\frac{(s_i^2;q)_{\lambda_i-\lambda_{i+1}}}{(q;q)_{\lambda_i-\lambda_{i+1}}},\qquad (-\S)^{\lambda}:=\prod_{i\geq 1}(-s_{i-1})^{\lambda_i}.
\end{equation*}

\item The following branching rule holds:
\be
\F_{\lambda/\nu}(\kappa_1,\dots, \kappa_n\mid\Xi,\S)=\sum_{\mu}\F_{\lambda/\mu}(\kappa_1,\dots, \kappa_m\mid\Xi,\S)\F_{\mu/\nu}(\kappa_{m+1},\dots, \kappa_n\mid\tau^m_{\S}\Xi,\tau^m_{\Xi}\S),
\ee
where $\tau^k_{\S}\Xi=(\xi_0^{(k)}, \xi_1^{(k)}, \xi_2^{(k)},\dots)$, $\tau^k_{\Xi}\S=(s_0^{(k)}, s_1^{(k)}, s_2^{(k)},\dots)$ with
\begin{equation*}
\xi_i^{(k)}:=\sqrt{\xi_{i+k}s_{i+k}\xi_i/s_i},\qquad s_i^{(k)}:=\sqrt{s_{i+k}\xi_{i+k}s_i/\xi_i},
\end{equation*}

Together with the partition function expression for the one-variable polynomials, this allows to write the general $\F_{\lambda/\nu}(\kappa_1,\dots, \kappa_n\mid\Xi,\S)$ in terms of the partition function for the 
vertex model in Figure \ref{partitionFunction} below (see Section \ref{FconstSect}).

\item The functions $\F_{\lambda}(\kappa_1, \dots, \kappa_n\mid\Xi,\S):=\F_{\lambda/\varnothing}(\kappa_1, \dots, \kappa_n\mid\Xi,\S)$ satisfy the following stability relation:
\be
\F_{\lambda}(\kappa_1, \dots, \kappa_{n-1}, s_n\xi_n\mid \Xi,\S)=\F_{\lambda}(\kappa_1, \dots, \kappa_{n-1}\mid \Xi,\S).
\ee

\item We have the following dual skew Cauchy identity:
\begin{multline*}
\sum_\lambda \FF_{\lambda'/\nu'}^*(u_1,\dots, u_m\mid\Xi,\S)\F_{\lambda/\mu}(\kappa_1,\dots, \kappa_n\mid\Xi,\S)\\
=\prod_{i=1}^n\prod_{j=1}^m\frac{1-u_j\kappa_i}{1-u_j\xi_is_i}\sum_{\lambda}\F_{\nu/\lambda}(\kappa_1,\dots, \kappa_n\mid\Xi,\S)\FF^*_{\mu'/\lambda'}(u_1,\dots, u_m\mid\tau^n_\S\Xi,\tau^n_\Xi\S),
\end{multline*}
where $\FF_{\lambda'/\nu'}^*$ denotes a certain stable version of the inhomogeneous spin Hall-Littlewood functions described in Section \ref{HLsect} below. In the case $\mu=\nu=\varnothing$ the identity reduces to
\be
\sum_\lambda \FF_{\lambda'}^*(u_1,\dots, u_m\mid\Xi,\S)\F_{\lambda}(\kappa_1,\dots, \kappa_n\mid\Xi,\S)=\prod_{i=1}^n\prod_{j=1}^m\frac{1-u_j\kappa_i}{1-u_j\xi_is_i}.
\ee
In these dual Cauchy identities the number of nonzero terms is always finite, thus no convergence conditions are needed. 

\item The following (skew) Cauchy-type summation identity holds:
\begin{multline*}
\sum_\lambda\F^*_{\lambda/\nu}(\chi_1,\dots, \chi_m\mid\bar\S,\S)\F_{\lambda/\mu}(\kappa_1,\dots, \kappa_n\mid\S,\S)\\
=\prod_{i=1}^n\prod_{j=1}^m\frac{(\kappa_i;q)_\infty(s_{i}^2\chi_j;q)_\infty}{(s_{i}^2;q)_\infty(\kappa_i\chi_j;q)_\infty}\sum_\rho\F_{\nu/\rho}(\kappa_1,\dots,\kappa_n\mid\S,\S)\F^*_{\mu/\rho}(\chi_1,\dots,\chi_m\mid\tau^n\bar\S,\tau^n\S).
\end{multline*}
Here $\F^*_{\lambda/\mu}$ denotes a certain rescaling of the functions $\F_{\lambda/\mu}$, and we set  $\overline{\S}=(s_0^{-1}, s_1^{-1}, s_2^{-1},\dots)$. In the case $\mu=\nu=\varnothing$ one obtains the following Cauchy identity:
\begin{equation*}
\sum_\lambda\F^*_{\lambda}(\chi_1,\dots, \chi_m\mid\bar\S,\S)\F_{\lambda}(\kappa_1,\dots, \kappa_n\mid\S,\S)=\prod_{i=1}^n\prod_{j=1}^m\frac{(\kappa_i;q)_\infty(s_{i}^2\chi_j;q)_\infty}{(s_{i}^2;q)_\infty(\kappa_i\chi_j;q)_\infty}.
\end{equation*}
These identities could be either viewed as holding in suitable rings of power series or numerically under certain restrictions on the parameters, see Theorem \ref{Cauchy} below.

\item The polynomials $\F_\mu$ admit an integral representation of the form

\begin{multline*}
	\F_{\mu}(\kappa_1,\dots, \kappa_n\mid\Xi,\S)\\
	=\oint_{\mathcal C}\frac{dz_1}{2\pi\i z_1}\dots\oint_{\mathcal C}\frac{dz_k}{2\pi\i z_k}\prod_{\alpha<\beta}\frac{z_\alpha-z_\beta}{z_\alpha-qz_\beta}\prod_{\alpha=1}^k\(\frac{\xi_0^{-1}}{z_\alpha-\xi^{-1}_{\mu'_\alpha}s_{\mu'_\alpha}}\prod_{j=1}^{\mu'_\alpha-1}\frac{1-s_j\xi_jz_\alpha}{z_\alpha\xi_j-s_j}\prod_{i=1}^n\frac{1-z_\alpha\kappa_i}{1-z_\alpha\xi_is_i}\),
\end{multline*}
where $k=\mu_1$, and the integration contour $\mathcal C$ is depicted in Figure \ref{Cpict} and described immediately before Theorem \ref{integral}.
\end{itemize}

Comparing with \cite{BW17}, for $s_0=s_1=\dots=s$ and $\xi_0=\xi_1=\dots=\xi$ our functions $\FF_{\lambda/\mu}$ and $\F_{\lambda/\mu}$ reduce to the spin $q$-Whittaker and stable spin Hall-Littlewood functions from \cite{BW17}; in particular, our functions degenerate to the usual Hall-Littlewood and $q$-Whittaker polynomials when $s=0$. 

\subsection*{Deformed Yang-Baxter equations}
The key tool that allowed us to prove most of the results above is a family of identities that we call \emph{deformed Yang-Baxter equations}. Their detailed exposition can be found in Section \ref{deformationSection} below. While the conventional Yang-Baxter equations for vertex models
have the general form
\be
\R^{(12)}_{\bf{x}}\R^{(13)}_{\bf{y}}\R^{(23)}_{\bf{z}}=\R^{(23)}_{\bf{z}}\R^{(13)}_{\bf{y}}\R^{(12)}_{\bf{x}},
\ee
where $\bf{x},\bf{y},\bf{z}$ are parameters subject to certain relations (in our setting the parameters $\bf{x},\bf{y},\bf{z}$ are pairs of complex numbers), and $\R^{(12)}_{\bf{x}},\R^{(13)}_{\bf{y}},\R^{(23)}_{\bf{z}}$ are operators acting in two of the three tensor factors in a tensor product $V_1\otimes V_2\otimes V_3$ with suitable linear spaces $V_i$. It turns out that in certain situations one can add an additional degree of freedom resulting in equations of the form 
\be
\R^{(12)}_{{\bf x}(\tau)}\R^{(13)}_{{\bf y}(\tau)}\R^{(23)}_{{\bf z}(\tau)}=\R^{(23)}_{{\bf z}'(\tau)}\R^{(13)}_{{\bf y}'(\tau)}\R^{(12)}_{{\bf x}'(\tau)}.
\ee
where the subscripts ${\bf x}(\tau), {\bf y}(\tau), {\bf z}(\tau)$ and ${\bf x}'(\tau), {\bf y}'(\tau), {\bf z}'(\tau)$ now depend on a new parameter $\tau$. Note that the deformed equation has different operators in the left-hand and right-hand sides.

In this work we provide two examples of the deformed Yang-Baxter equation; one of them implies that the functions $\F_{\lambda/\mu}(\kappa_1,\dots,\kappa_n\mid\Xi,\S)$ are symmetric, while the other one is crucial for proving the Cauchy identities. Additionally, we describe a method we use to deform the previously known Yang-Baxter equations. Such deformations might be useful in other contexts. For instance, in the end of Section \ref{deformationSection} we briefly mention a different application: One can use a deformation of a higher rank Yang-Baxter equation to extend the results of \cite{BW20} on integral representations for certain observables of a colored stochastic higher spin six-vertex model. Unfortunately, so far our understanding of the deformed Yang-Baxter equation is based solely on  computational tricks we use to derive them;  we are not aware of either a representation theoretic meaning of such equations or a complete list of situations where one might expect to see them.

\subsection*{Applications to probability.} The inhomogeneous spin $q$-Whittaker functions $\F_{\lambda/\mu}(\kappa_1, \dots, \kappa_n\mid\Xi,\S)$ can be defined in terms of a new stochastic integrable vertex model with parameters attached to diagonals, see Section \ref{sqWhitSect} below. Earlier results \cite{BP16a}, \cite{BP16b}, \cite{BW20} suggest that this model could be studied via Cauchy-type identities, which would lead to explicit expressions for averages of certain natural observables. Indeed, such an approach is also applicable to our new model.
However, very recently another method of analyzing stochastic vertex models was introduced by Bufetov-Korotkikh \cite{BK20}. It is less constructive but yields more general results; its key feature is a certain local relation for $q$-moments of the height functions for the model. The adaptation of this method to the setting of the present paper leads to novel probabilistic results that include, in particular, large scale asymptotics for partition functions of a new integrable extension of the Beta polymer model of Barraquand-Corwin \cite{BC15}. A detailed exposition of these results will appear in the upcoming work \cite{Kor21} by one of the authors.

\subsection*{Layout of the paper.} In Section \ref{prelim}, we recall the definition and properties of the (stochastic) higher spin six vertex model and its extensions obtainable by fusion. In particular, we list several instances of the Yang-Baxter equation used throughout the work. In Section \ref{deformationSection}, we introduce the deformed Yang-Baxter equations; our construction there is based on certain rationality properties of the vertex weights. Section \ref{rowOperatorsSection} is devoted to \emph{row operators}. On one hand, these operators encode partition functions of single rows of a vertex model subject to certain boundary conditions, while on the other hand, the row operators capture the branching structure of our symmetric functions. In the same section we also prove exchange relations between the row operators using the (deformed) Yang-Baxter equations and a `zipper' argument. Having developed the framework of row operators, in Section \ref{sqWhitSect} we define the inhomogeneous spin $q$-Whittaker functions $\F_{\lambda/\mu}(\kappa_1, \dots, \kappa_n\mid \Xi,\S)$ and provide proofs of the above-listed  properties. Finally, in Section \ref{cauchySection} we state and prove the (skew) Cauchy identity between the functions $\F_{\lambda/\mu}(\kappa_1, \dots, \kappa_n\mid \S,\S)$ and $\F_{\lambda/\mu}(\chi_1, \dots, \chi_m\mid \overline{\S},\S)$.

\subsection*{Notation.} In this work we follow the standard notations regarding partitions: A \emph{partition} $\lambda$ is a monotone infinite sequence $(\lambda_1,\lambda_2,\lambda_3,\dots)$ of nonnegative integers satisfying
\be
\lambda_1\geq\lambda_2\geq\lambda_3\geq\dots\geq 0, \qquad |\lambda|:=\lambda_1+\lambda_2+\lambda_3+\dots<\infty.
\ee
The coordinates $\lambda_i$ are called parts of the partition. For a partition $\lambda$, its \emph{length} $l(\lambda)$ is equal to the number of nonzero parts $\lambda_i$, while $m_k(\lambda)$ is used to denote the number of parts $\lambda_i$ equal to $k$. The \emph{conjugate} $\lambda'$ of a partition $\lambda$ is defined by $\lambda'_i=\#\{j\geq 1: \lambda_j\geq 1\}$.

For a pair of partitions $\lambda,\mu$, we write $\mu\subset\lambda$ iff for any $i\geq 1$ we have $\mu_i\leq\lambda_i$. Further, we say the partition $\lambda$ \emph{interlaces} the partition $\mu$, and write $\lambda\succ\mu$, if $\lambda_i\geq\mu_i\geq\lambda_{i+1}$ for all $i\geq 1$. Note that in this case we also have $0\leq l(\lambda)-l(\mu)\leq 1$. 

Throughout the text we use the following notation for $q$-Pochhammer symbol:
\be
(x;q)_n:=\prod_{i=1}^{n}(1-xq^{i-1}),
\ee
with the convention $(x;q)_0:=1$. We also often assume that $|q|<1$, so the notation $(x;q)_{\infty}$ makes sense both numerically and formally in the space of power series in $q$.

\subsection*{Acknowledgments.} A.~B. was partially supported by the NSF grant DMS-1853981 and the Simons  Investigator program. A.~B. and S.~K. were also partially supported by the NSF FRG grant DMS-1664619.

\section{Preliminaries on vertex models and the Yang-Baxter equation}\label{prelim}
Vertex models considered in this work consist of collections of vertices with oriented edges between them. A configuration of a model is an assignment of integer labels to all edges in a way such that the \emph{(arrow) conservation law} holds: for each vertex the sum of labels of the incoming edges equals the sum of labels of the outgoing ones. To each configuration we assign a weight, which is equal to the product of local weights corresponding to each vertex and depending on the labels around the vertex. In this section we describe various local weights, as well as recall several known properties of them.

\subsection{Higher spin vertex weights.} We start with the (stochastic) higher spin six vertex model. Our exposition in this section follows \cite{BP16b} and \cite{BW17}, up to a change of notation discussed in Remark \ref{renorm} below. 

Consider vertices on a square grid with vertical edges directed upwards and having integer labels from $\Z_{\geq 0}$, and horizontal edges directed to the right and having labels from $\{0,1\}$. To each configuration of labels around a vertex we assign a \emph{vertex weight}, which is denoted by 
\vspace{1ex}\begin{equation}
\label{hsVertexDef}
\tikz{1}{
	\draw[unfused] (-1,0) -- (1,0);
	\draw[fused] (0,-1) -- (0,1);
	\node[left] at (-1,0) {\tiny $j$};\node[right] at (1,0) {\tiny $l$};
	\node[below] at (0,-1) {\tiny $i$};\node[above] at (0,1) {\tiny $k$};
	\node[above right] at (0,0) {\tiny $w^{\s}_{u;s}$};
}=w^{\s}_{u;s}(i,j;k,l)
\end{equation}
and given explicitly by the following table
\begin{align}
\label{hsValues}
\begin{array}{cccc}
\tikz{0.9}{
	\draw[unfused] (-1,0) -- (1,0);
	\draw[fused] (0,-1) -- (0,1);
	\node[left] at (-1,0) {\tiny $0$};\node[right] at (1,0) {\tiny $0$};
	\node[below] at (0,-1) {\tiny $g$};\node[above] at (0,1) {\tiny $g$};
	\node[above right] at (0,0) {\tiny $w^\s_{u;s}$};
}
\qquad
&
\tikz{0.9}{
	\draw[unfused] (-1,0) -- (1,0);
	\draw[fused] (0,-1) -- (0,1);
	\node[left] at (-1,0) {\tiny $0$};\node[right] at (1,0) {\tiny $1$};
	\node[below] at (0,-1) {\tiny $g$};\node[above] at (0,1) {\tiny $g-1$};
	\node[above right] at (0,0) {\tiny $w^\s_{u;s}$};
}
\qquad
&
\tikz{0.9}{
	\draw[unfused] (-1,0) -- (1,0);
	\draw[fused] (0,-1) -- (0,1);
	\node[left] at (-1,0) {\tiny $1$};\node[right] at (1,0) {\tiny $0$};
	\node[below] at (0,-1) {\tiny $g$};\node[above] at (0,1) {\tiny $g+1$};
	\node[above right] at (0,0) {\tiny $w^\s_{u;s}$};
}
\qquad
&
\tikz{0.9}{
	\draw[unfused] (-1,0) -- (1,0);
	\draw[fused] (0,-1) -- (0,1);
	\node[left] at (-1,0) {\tiny $1$};\node[right] at (1,0) {\tiny $1$};
	\node[below] at (0,-1) {\tiny $g$};\node[above] at (0,1) {\tiny $g$};
	\node[above right] at (0,0) {\tiny $w^\s_{u;s}$};
}
\\
\\
\dfrac{1-suq^g}{1-su}
\qquad
&
\dfrac{(q^g-1)su}{1-su}
\qquad
&
\dfrac{1-s^2q^g}{1-su}
\qquad
&
\dfrac{s^2q^g-su}{1-su}
\vspace{1ex}
\end{array}
\end{align}
Here $q$ is a fixed global ``quantization" parameter of the model, the parameters $u,s$ are local ``spectral" and ``spin" parameters, respectively, that vary depending on the vertex, and $g\in\mathbb Z_{\geq 0}$ denotes any nonnegaive integer. We assume that $q$, $u$, $s$ are generic complex numbers and often additionally assume that $|q|<1$. All unlisted weights are set to $0$, enforcing the \emph{conservation law}:
\be
w^\s_{u;s}(i,j;k,l)=0, \quad\text{unless}\ i+j=k+l.
\ee

The weights $w^\s_{u;s}$ are closely related to matrix coefficients of the higher spin $R$-matrix of  $U_q(\widehat{\frak{sl}_2})$, see \cite[Proposition 2.4]{Bor14} for details. As a consequence of this, they satisfy a version of the \emph{Yang-Baxter equation} (also sometimes referred to as the \emph{LLR equation}), see, \emph{e.g}, \cite[Proposition 2.2]{BW17}: for any $a_1, b_1\in\Z_{\geq 0}$ and $a_2,a_3,b_2,b_3\in\{0,1\}$ we have
\begin{multline}
\label{higherSpinYangBaxterFormula}
\sum_{l_1,l_2,l_3}w^\s_{x;s}(a_1,a_2;l_1,l_2)w^\s_{y;s}(l_1,a_3;b_1,l_3)\R_{x/y}(l_2,l_3;b_2,b_3)\\
=\sum_{l_1,l_2,l_3}\R_{x/y}(a_2,a_3;l_2,l_3)w^\s_{y;s}(a_1,l_3;l_1,b_3)w^\s_{x;s}(l_1,l_2;b_1,b_2).
\end{multline}
Here $\R_z$ denotes the \emph{$\R$-matrix} defined by
\be
\R_z(0,0;0,0)=\R_z(1,1;1,1)=1-qz,
\ee
\begin{align*}
&\R_z(0,1;0,1)={1-z}, &\R_z(0,1;1,0)={z(1-q)},&\\
&\R_z(1,0;0,1)={1-q}, &\R_z(1,0;1,0)={q(1-z)},&
\end{align*}
and for all values not listed above $\R_z(i,j;k,l)=0$.

Throughout this work, instead of writing equations like \eqref{higherSpinYangBaxterFormula}, we extensively use graphical notation for partition functions of vertex models. In general, we consider a collection of vertices and oriented edges (possibly of different kinds) connected to them, with some edges connecting a pair of vertices, and some edges being connected to only one vertex with the other end being free. The edges of the former type are called \emph{internal edges}, while the edges of the latter type are referred to as \emph{boundary edges}.  Additionally, to each vertex we assign a family of vertex weights, with each weight corresponding to a particular configuration of labels of the edges connected with the vertex, for example \eqref{hsVertexDef}, and to each boundary edge we assign a label which is called the \emph{boundary condition} corresponding to the edge. Given such data, the corresponding \emph{partition function} is defined as the sum of the products of the weights of all vertices, taken over all possible assignments of labels to the internal edges.

For example, the Yang-Baxter equation \eqref{higherSpinYangBaxterFormula} for the weights $w_{u;s}^\s$ can be graphically interpreted as follows:
\begin{equation}
\label{higherSpinYangBaxter}
\tikzbase{1.2}{-3}{
	\draw[unfused] 
	(-1,1) node[left] {\tiny \color{black} $a_3$} -- (1,1)  -- (2,0) node[below] {\tiny \color{black} $b_3$};
	\draw[unfused] 
	(-1,0) node[left] {\tiny \color{black} $a_2$} -- (1,0)  -- (2,1) node[above] {\tiny \color{black} $b_2$};
	\draw[fused] 
	(0,-1) node[below] {\tiny \color{black} $a_1$} -- (0,0.5)  -- (0,2) node[above] {\tiny \color{black} $b_1$};
	\node[above right] at (0,0) {\tiny{ $w^\s_{x;s}$}};
	\node[above right] at (0,1) {\tiny{ $w^\s_{y;s}$}};
	\node[right] at (1.5,0.5) {\tiny{$\ \R_{x/y}$}};
}\qquad
=
\qquad
\tikzbase{1.2}{-3}{
	\draw[unfused]
	(-2,0.5) node[above] {\tiny \color{black} $a_3$} -- (-1,-0.5)  -- (1,-0.5) node[right] {\tiny \color{black} $b_3$};
	\draw[unfused] 
	(-2,-0.5) node[below] {\tiny \color{black} $a_2$} -- (-1,0.5)  -- (1,0.5) node[right] {\tiny \color{black} $b_2$};
	\draw[fused] 
	(0,-1.5) node[below] {\tiny \color{black} $a_1$} -- (0,0)  -- (0,1.5) node[above] {\tiny \color{black} $b_1$};
	\node[above right] at (0,-0.5) {\tiny{ $w^\s_{y;s}$}};
	\node[above right] at (0,0.5) {\tiny{ $w^\s_{x;s}$}};
	\node[right] at (-1.5,0) {\tiny{$\ \R_{x/y}$}};
}
\end{equation}
The vertices with horizontal and vertical edges have the weights $w^\s_{u;s}$ given by \eqref{hsVertexDef}, while the vertices with diagonal edges, which we call \emph{tilted} vertices, correspond to the $\R$-matrix:
\be
\tikz{1}{
	\draw[unfused] (-0.5,0.5) -- (0.5,-0.5);
	\draw[unfused] (-0.5,-0.5) -- (0.5,0.5);
	\node[below] at (-0.5,-0.5) {\tiny $i$};\node[above] at (0.5,0.5) {\tiny $k$};
	\node[below] at (0.5,-0.5) {\tiny $l$};\node[above] at (-0.5,0.5) {\tiny $j$};
	\node[right] at (0,0) {\tiny $\R_z$};
}=\R_{z}(i,j;k,l).
\ee
Note that the diagrams we use have two kinds of edges, which visually differ by their thickness. Throughout this work we follow the convention that thin edges carry labels from $\{0,1\}$, while thick edges are labeled by any non-negative integer.

\begin{rem}
\label{renorm}
The letter $\s$ in the superscript of the notation $w^\s_{u;s}$ stands for ``stochastic", and should not be confused with the spin parameter $s$ of a vertex. We use such a notation to distinguish our weights from the weights $w_{u;s}$ often used in this context, \emph{cf.} \cite{Bor14}, \cite{BP16b} and \cite{BW17}. Our weights $w^\s_{u;s}$ coincide with the weights ${\sf L}_{u,s}$ used in \cite{BP16b}, and they are related to the weights $w_{u;s}$ by \cite[(2.2) and (2.3)]{BP16b}:
\be
w^\s_{u;s}(i,j;k,l)=(-s)^l\frac{(s^2;q)_k}{(q;q)_k}\frac{(q;q)_i}{(s^2;q)_i}w_{u;s}(i,j;k,l).
\ee

The word ``stochastic" refers to the fact that the sum of these weights over different labels of the outgoing edges (and fixed values of the incoming ones) is identically equal to $1$, \emph{cf.} Remark \ref{stochRem} below.
\end{rem}

\subsection{Fusion}\label{fusionSection} The weights $w^{\s}_{u;s}$ can be used to reach more general vertex models by a procedure called \emph{fusion} that we review below. See \cite[Section 4]{BW17} and \cite[Section 5]{BP16b} for a more detailed exposition of the fusion procedure, as well as further references.   

Fix $J>0$ and consider a column consisting of $J$ vertices having weights $w^\s_{u;s}, w^\s_{qu;s}, \dots, w^\s_{q^{J-1}u;s}$ starting from the bottom. This configuration is called a \emph{$J$-column}, and using the Yang-Baxter equation one can show that, under specific boundary left and right conditions, this column of vertices behaves like a single vertex of a new type, with left and right edges having labels from $\{0, \dots, J\}$. 

More precisely, for $j\geq 0$ set
\be
Z_j(J):=\sum_{a_1+\dots+a_J=j}q^{\sum_{r=1}^J(r-1)a_r}=\frac{q^{\frac{j(j-1)}{2}}(q;q)_{J}}{(q;q)_{j}(q;q)_{J-j}},
\ee
where the sum is taken over all sequences $(a_1,\dots, a_J)\in\{0,1\}^J$ satisfying $\sum_{r=1}^J a_r=j$, and the equality is a particular case of the $q$-binomial identity. For $i,j,k,l\in\Z_{\geq 0}$ such that $j,l\leq J$ define vertex weights $\W^{(J)}_{u;s}$ by
\begin{equation}
\label{fusion}
\W^{(J)}_{u;s}(i,j;k,l):=\frac{Z_l(J)}{Z_j(J)}q^{-\sum_{r=1}^J(r-1)b_r}\cdot\sum_{a_1+\dots+a_J=j}q^{\sum_{r=1}^J(r-1)a_r}\(
\tikz{0.85}{
	\foreach\y in {2,...,3}{
		\draw[unfused] (1,\y) -- (3.8,\y);
	}
	\draw[unfused] (1,4.5) -- (3.8,4.5);
	\draw[fused*] (2,1) -- (2, 3.5) node[above, scale = 0.6] {\color{black} $\vdots$};
	\draw[fused] (2,4) -- (2,5.5);
	\node[above] at (2,5.5) {$k$};
	\node[left] at (1,2) {$a_1$};
	\node[left] at (1,3) {$a_2$};
	\node[left] at (1,4.5) {$a_{J}$};
	\node[below] at (2,1) {$i$};
	\node[right] at (3.8,2) {$b_1$};
	\node[right] at (3.8,3) {$b_2$};
	\node[right] at (3.8,4.5) {$b_{J}$};
	\node[above right, scale = 0.8] at (2,2) {$w^\s_{u;s}$};
	\node[above right, scale = 0.8] at (2,3) {$w^\s_{qu;s}$};	
	\node[above right, scale = 0.8] at (2,4.5) {$w^\s_{q^{J-1}u;s}$};	
}\).
\end{equation}
Here $(b_1,\dots,b_J)\in\{0,1\}^J$ is any vector satisfying $\sum_{r=1}^Jb_r=l$. It turns out that the right-hand side of \eqref{fusion} does not depend on the particular choice of $(b_1,\dots, b_J)$, so the values $\W^{(J)}_{u;s}(i,j;k,l)$ are well defined. This property is called \emph{$q$-exchangeability}, and it allows to stack \eqref{fusion} horizontally, identifying a row of $J$-columns with a single row of ``fused" vertices:
\begin{multline}
\label{multiFusion}
\sum_{\substack{j_0,j_1,\dots,j_L\in\{0,\dots,J\}\\j_0=j,j_L=l}}\W^{(J)}_{u_1;s_1}(i_1,j_0;k_1,j_1)\W^{(J)}_{u_2;s_2}(i_2,j_1;k_2,j_2)\dots\W^{(J)}_{u_L;s_L}(i_L,j_{L-1};k_L,j_L)\\
=\frac{Z_l(J)}{Z_j(J)}q^{-\sum_{r=1}^J(r-1)b_r} \sum_{a_1+\dots+ a_r=j}q^{\sum_{r=1}^J(r-1)a_r}\(
\tikz{0.85}{
	\foreach\y in {2,...,3}{
		\draw[unfused*] (1,\y) -- (3,\y) node[right, scale =0.6] {\color{black} $\dots$};
		\draw[unfused] (3.5, \y) -- (6,\y);
	}
	\draw[unfused*] (1,4.5) -- (3,4.5) node[right, scale =0.6] {\color{black} $\dots$};;
	\draw[unfused] (3.5, 4.5) -- (6,4.5);
	\draw[fused*] (2,1) -- (2, 3.5) node[above, scale = 0.6] {\color{black} $\vdots$};
	\draw[fused] (2,4) -- (2,5.5);
	\draw[fused*] (4,1) -- (4, 3.5) node[above, scale = 0.6] {\color{black} $\vdots$};
	\draw[fused] (4,4) -- (4,5.5);
	\node[above] at (2,5.5) {$k_1$};
	\node[above] at (4,5.5) {$k_L$};
	\node[left] at (1,2) {$a_1$};
	\node[left] at (1,3) {$a_2$};
	\node[left] at (1,4.5) {$a_{J}$};
	\node[below] at (2,1) {$i_1$};
	\node[below] at (4,1) {$i_L$};
	\node[right] at (6,2) {$b_1$};
	\node[right] at (6,3) {$b_2$};
	\node[right] at (6,4.5) {$b_{J}$};
	\node[above right, scale = 0.8] at (2,2) {$w^\s_{u_1;s_1}$};
	\node[above right, scale = 0.8] at (2,3) {$w^\s_{qu_1;s_1}$};	
	\node[above right, scale = 0.8] at (2,4.5) {$w^\s_{q^{J-1}u_1;s_1}$};	
	\node[above right, scale = 0.8] at (4,2) {$w^\s_{u_L;s_L}$};
	\node[above right, scale = 0.8] at (4,3) {$w^\s_{qu_L;s_L}$};	
	\node[above right, scale = 0.8] at (4,4.5) {$w^\s_{q^{J-1}u_L;s_L}$};	
}\).
\end{multline}
Note that the sum on the left-hand side of \eqref{multiFusion} actually has at most one nonzero term: due to the conservation law we must have $j_r=j+\sum_{m=1}^r i_m-\sum_{m=1}^r k_m$.

The resulting weights $\W^{(J)}_{u;s}$ can be given by an explicit expression obtained in \cite{Man14}, here we use a slightly different form of this expression given in \cite[Theorem C.1.1]{BW18} \footnote{Our weights $\W^{(J)}_{u;s}(i,j;k,l)$ coincide with the analytic continuation $\restr{W_{J,M}(u/s;q;i \bm{e}_1, j \bm{e}_1, k \bm{e}_1, l\bm{e}_1)}{q^{M}=s^{-2}}$ of the weights $W_{L,M}(x/y;q;\bm{A},\bm{B},\bm{C},\bm{D})$ from \cite[Appendix C]{BW18}.} and originated from \cite{BM16}:
\begin{equation}
\label{fullyFusedDef}
\W^{(J)}_{u;s}(i,j;k,l)=\1_{i+j=k+l}\ u^{l-j}q^{iJ}s^{j+l}\sum_{p=0}^{min(j,k)}\phi(k-p,k+l-p;suq^J,su)\phi(p,j;sq^{-J}/u,q^{-J}),
\end{equation}
where
\be
\phi(a,b;x,y):=(y/x)^a\frac{(x;q)_a(y/x;q)_{b-a}}{(y;q)_b}\frac{(q;q)_b}{(q;q)_a(q;q)_{b-a}}.
\ee

Actually, one can construct the weights $\W^{(J)}_{u;s}$ by fusing vertices with the weights $\R_z$, expressing $\W^{(J)}_{u;s}$ via a $J\times M$ rectangular grid of the vertices with weights $\R_{z_{i,j}}$ for a certain choice of parameters $z_{i,j}$ and specific boundary conditions. We are not using this construction for the current work; more details on this approach can be found in \cite[Appendix]{BGW19}.

\subsection{Simplification of $\W$-weights}\label{simplify} The weights $\W^{(J)}_{u;s}$ are rather complicated, but in this work we only use a particular ``$q$-Hahn" specialization of them: When $u=s$ the weights $\W^{(J)}_{u;s}$ simplify, with all but $p=j$ term vanishing in \eqref{fullyFusedDef}, providing the following expression
\begin{multline}
\label{Wdef}
\tikz{1}{
	\draw[fused] (-1,0) -- (1,0);
	\draw[fused] (0,-1) -- (0,1);
	\node[left] at (-1,0) {\tiny $j$};\node[right] at (1,0) {\tiny $l$};
	\node[below] at (0,-1) {\tiny $i$};\node[above] at (0,1) {\tiny $k$};
	\node[above right] at (0,0.1) {\tiny $W^\s_{t,s}$};
}=W^\s_{t,s}(i,j;k,l):=\restr{\W^{(J)}_{s;s}(i,j;k,l)}{q^J=t^{-2}}\\
=\1_{i+j=k+l}\ \1_{i\geq l}\ s^{2l}t^{-2l}\frac{(s^2/t^2;q)_{i-l}(t^2;q)_l}{(s^2;q)_i}\frac{(q;q)_i}{(q;q)_{i-l}(q;q)_l}.
\end{multline}
Here we have used the fact that for fixed $i,j,k,l$ the weight $\W^{(J)}_{u;s}(i,j;k,l)$ depends rationally on $q^J$ for sufficiently large $J$. Thus, we can perform an \emph{analytic continuation} replacing $q^J$ by $t^2$, where $t$ is a new parameter.

Similarly to Remark \ref{renorm}, the weights $W^\s_{t,s}$ defined above differ from the weights $W_{x}$ from \cite[(35)]{BW17}, and these two families of weights are related by
\begin{equation}
\label{Wmatching}
W_{s,t}^\s(i,j;k,l)=(-s)^l\frac{(s^2;q)_k}{(q;q)_k} \frac{(q;q)_i}{(s^2;q)_i} W_{-s/t^2}(i,j;k,l),
\end{equation}
where the global spin parameter from \cite{BW17} is assumed to be equal to $s$.

The vertices with the weights $W^\s_{t,s}$ satisfy a version of the Yang-Baxter equation: for every $a_1,a_2,a_3,b_1,b_2,b_3\in\Z_{\geq 0}$ we have
\begin{equation}
\label{WYB}
\tikzbase{1.2}{-3}{
	\draw[fused] 
	(-1,1) node[left,scale=0.6] {\color{black} $a_1$} -- (1,1) -- (2,0) node[below,scale=0.6] {\color{black} $b_1$};
	\draw[fused] 
	(-1,0) node[left,scale=0.6] {\color{black} $a_2$} -- (1,0) -- (2,1) node[above,scale=0.6] {\color{black} $b_2$};
	\draw[fused] 
	(0,-1) node[below,scale=0.6] {\color{black} $a_3$} -- (0,0.5) -- (0,2) node[above,scale=0.6] {\color{black} $b_3$};
	\node[above right] at (0,0) {\tiny{ $W^\s_{t_2,t_3}$}};
	\node[above right] at (0,1) {\tiny{ $W^\s_{t_1,t_3}$}};
	\node[right] at (1.5,0.5) {\ \tiny{$W^\s_{t_1,t_2}$}};
}\qquad
=
\qquad
\tikzbase{1.2}{-3}{
	\draw[fused]
	(-2,0.5) node[above,scale=0.6] {\color{black} $a_1$} -- (-1,-0.5) -- (1,-0.5) node[right,scale=0.6] {\color{black} $b_1$};
	\draw[fused] 
	(-2,-0.5) node[below,scale=0.6] {\color{black} $a_2$} -- (-1,0.5)  -- (1,0.5) node[right,scale=0.6] {\color{black} $b_2$};
	\draw[fused] 
	(0,-1.5) node[below,scale=0.6] {\color{black} $a_3$} -- (0,0) -- (0,1.5) node[above,scale=0.6] {\color{black} $b_3$};
	\node[above right] at (0,-0.4) {\tiny{ $W^\s_{t_1,t_3}$}};
	\node[above right] at (0,0.6) {\tiny{ $W^\s_{t_2,t_3}$}};
	\node[right] at (-1.5,0) {\ \tiny{$W^\s_{t_1,t_2}$}};
}
\end{equation}
where the tilted vertex is just a rotation by 45$^\circ$ of the vertex with the weights $W^\s_{t_1,t_2}$. This Yang-Baxter equation is given in \cite[Corollary 4.3]{BW17}
: it can be proved either by a direct computation, or using the Yang-Baxter equation \eqref{higherSpinYangBaxter} together with fusion.

\begin{rem}
\label{stochRem}
Apart from the Yang-Baxter equation, the weights $w^\s_{u;s}$, $W^\s_{t,s}$ and $\W^{(J)}_{u;s}$ satisfy another property, namely, they are \emph{stochastic}: we have
\be
\sum_{k,l} w^\s_{u;s}(i,j;k,l)=1,\qquad \sum_{k,l} W^\s_{t,s}(i,j;k,l)=1,\qquad \sum_{k,l} \W^{(J)}_{u;s}(i,j;k,l)=1,
\ee
with fixed $i,j$ in the appropriate label sets.

We use the weights $w^\s_{u;s}$ and $W^\s_{t,s}$ instead of $w_{u;s}$ and $W_{x}$ because the deformations of the Yang-Baxter equation in Section \ref{deformationSection} below look more natural for the stochastic weights. 
\end{rem}

\subsection{Dual weights.} Apart from vertices with edges oriented along the up-right direction we also use vertices with edges oriented along the up-left direction. We call them \emph{dual vertices}, and their weights are closely related to the weights of the regular vertices. Such vertices play an important role for symmetric functions arising from vertex models, especially for the proofs of Cauchy type summation identities for such functions. Dual vertex weights were used in  \cite{BW17}, and they are equivalent to the conjugated weights from \cite{BP16b}. Since our weights $w^\s$ and $W^\s$ differ from the weights used in those works, our dual vertex weights also differ, though we do not state the exact matching between these weights; it can be readily recovered from the definitions below, Remark \ref{renorm} and \eqref{Wmatching}.

Define dual vertex weights $w^{\s*}_{u;s}$ by the following relation:
\vspace{1ex}\begin{multline}
\label{dualWeights}
\tikz{1}{
	\draw[unfused] (1,0) -- (-1,0);
	\draw[fused] (0,-1) -- (0,1);
	\node[left] at (-1,0) {\tiny $j$};\node[right] at (1,0) {\tiny $l$};
	\node[below] at (0,-1) {\tiny $i$};\node[above] at (0,1) {\tiny $k$};
	\node[above right] at (0,0) {\tiny $w^{\s*}_{u;s}$};
}=w^{\s*}_{u;s}(i,l;k,j):=s^{-2l}\frac{(q;q)_i}{(s^2;q)_i}\frac{(s^2;q)_k}{(q;q)_k}w^\s_{u;s}(k,j;i,l)\\
=s^{-2l}\frac{(q;q)_i}{(s^2;q)_i}\frac{(s^2;q)_k}{(q;q)_k}
\tikz{1}{
	\draw[unfused] (-1,0) -- (1,0);
	\draw[fused] (0,-1) -- (0,1);
	\node[left] at (-1,0) {\tiny $j$};\node[right] at (1,0) {\tiny $l$};
	\node[below] at (0,-1) {\tiny $k$};\node[above] at (0,1) {\tiny $i$};
	\node[above right] at (0,0) {\tiny $w^{\s}_{u;s}$};
}.
\end{multline}
The explicit values of weights $w^{\s*}_{u;s}$  can be computed from \eqref{hsValues} and they are summarized below:
\begin{align}
\label{dualWeightsTable}
\begin{array}{cccc}
\tikz{0.8}{
	\draw[unfused] (1,0) -- (-1,0);
	\draw[fused] (0,-1) -- (0,1);
	\node[left] at (-1,0) {\tiny $0$};\node[right] at (1,0) {\tiny $0$};
	\node[below] at (0,-1) {\tiny $g$};\node[above] at (0,1) {\tiny $g$};
	\node[above right] at (0,0) {\tiny $w^{\s*}_{u;s}$};
}
\qquad
&
\tikz{0.8}{
	\draw[unfused] (1,0) -- (-1,0);
	\draw[fused] (0,-1) -- (0,1);
	\node[left] at (-1,0) {\tiny $0$};\node[right] at (1,0) {\tiny $1$};
	\node[below] at (0,-1) {\tiny $g$};\node[above] at (0,1) {\tiny $g+1$};
	\node[above right] at (0,0) {\tiny $w^{\s*}_{u;s}$};
}
\qquad
&
\tikz{0.8}{
	\draw[unfused] (1,0) -- (-1,0);
	\draw[fused] (0,-1) -- (0,1);
	\node[left] at (-1,0) {\tiny $1$};\node[right] at (1,0) {\tiny $0$};
	\node[below] at (0,-1) {\tiny $g$};\node[above] at (0,1) {\tiny $g-1$};
	\node[above right] at (0,0) {\tiny $w^{\s*}_{u;s}$};
}
\qquad
&
\tikz{0.8}{
	\draw[unfused] (1,0) -- (-1,0);
	\draw[fused] (0,-1) -- (0,1);
	\node[left] at (-1,0) {\tiny $1$};\node[right] at (1,0) {\tiny $1$};
	\node[below] at (0,-1) {\tiny $g$};\node[above] at (0,1) {\tiny $g$};
	\node[above right] at (0,0) {\tiny $w^{\s*}_{u;s}$};
}
\\
\\
\dfrac{1-suq^g}{1-su}
\qquad
&
\dfrac{-s^{-1}u(1-s^2q^g)}{1-su}
\qquad
&
\dfrac{1-q^g}{1-su}
\qquad
&
\dfrac{q^g-s^{-1}u}{1-su}
\vspace{1ex}
\end{array}
\end{align}
As before, all unlisted configurations have weight $0$, with the conservation law taking the form
\be
w^{\s*}_{u;s}(i,l;k,j)=0, \quad\text{unless}\ i+l=k+j.
\ee

Looking at the explicit values, one can establish another relation between the weights $w^{\s}_{u,s}$ and $w^{\s*}_{u,s}$:
\begin{equation}
\label{directionChange}
\tikz{1}{
	\draw[unfused] (1,0) -- (-1,0);
	\draw[fused] (0,-1) -- (0,1);
	\node[left] at (-1,0) {\tiny $j$};\node[right] at (1,0) {\tiny $l$};
	\node[below] at (0,-1) {\tiny $i$};\node[above] at (0,1) {\tiny $k$};
	\node[above right] at (0,0) {\tiny $w^{\s*}_{u;s}$};
}=
\frac{s-u}{s(1-us)}\
\tikz{1}{
	\draw[unfused] (-1,0) -- (1,0);
	\draw[fused] (0,-1) -- (0,1);
	\node[left] at (-1,0) {\tiny $(1-j)$};\node[right] at (1,0) {\tiny $(1-l)$};
	\node[below] at (0,-1) {\tiny $i$};\node[above] at (0,1) {\tiny $k$};
	\node[above right] at (0,0) {\tiny $w^{\s}_{u^{-1};s}$};
}.
\end{equation}
This is essentially the relation between the initial and conjugated weights in \cite{BP16b} that was used to prove the Cauchy identity there.

Using the definition \eqref{dualWeights}, the dual weights $w^{\s*}_{u;s}$ can be interpreted as rescaled $w^{\s}_{u,s}$ weights reflected across the horizontal axis and with the direction of all edges reversed. Thus, one can rewrite the Yang-Baxter equation \eqref{higherSpinYangBaxter} in terms of the dual weights by reversing the orientation of all the edges, reflecting the whole vertex diagrams on both sides of \eqref{higherSpinYangBaxter} across the horizontal axis and then renormalizing the weights. It turns out that the factors arising form this rescaling actually cancel out\footnote{Cancellation of the terms $(-s)^{-2l}$ uses the conservation law of the tilted vertex.}, and we arrive at the following Yang-Baxter equation:
\begin{equation}
\label{higherSpinDualYangBaxter}
\tikzbase{1.2}{-3}{
	\draw[unfused] 
	(2,0) node[below] {\tiny \color{black} $b_3$} -- (1,1)  -- (-1,1) node[left] {\tiny \color{black} $a_3$};
	\draw[unfused] 
	(2,1) node[above] {\tiny \color{black} $b_2$} -- (1,0)  -- (-1,0) node[left] {\tiny \color{black} $a_2$};
	\draw[fused] 
	(0,-1) node[below] {\tiny \color{black} $a_1$} -- (0,0.5)  -- (0,2) node[above] {\tiny \color{black} $b_1$};
	\node[above right] at (0,0) {\tiny{ $w^{\s*}_{x;s}$}};
	\node[above right] at (0,1) {\tiny{ $w^{\s*}_{y;s}$}};
	\node[right] at (1.5,0.5) {\tiny{$\R^*_{x/y}$}};
}\qquad
=
\qquad
\tikzbase{1.2}{-3}{
	\draw[unfused]
	(1,-0.5) node[right] {\tiny \color{black} $b_3$} -- (-1,-0.5)  -- (-2,0.5) node[above] {\tiny \color{black} $a_3$};
	\draw[unfused] 
	(1,0.5) node[right] {\tiny \color{black} $b_2$} -- (-1,0.5)  -- (-2,-0.5) node[below] {\tiny \color{black} $a_2$};
	\draw[fused] 
	(0,-1.5) node[below] {\tiny \color{black} $a_1$} -- (0,0)  -- (0,1.5) node[above] {\tiny \color{black} $b_1$};
	\node[above right] at (0,-0.5) {\tiny{ $w^{\s*}_{y;s}$}};
	\node[above right] at (0,0.5) {\tiny{ $w^{\s*}_{x;s}$}};
	\node[right] at (-1.5,0) {\tiny{$\R^*_{x/y}$}};
}
\end{equation}
where
\be
\tikz{1}{
	\draw[unfused] (0.5,0.5) -- (-0.5,-0.5);
	\draw[unfused] (0.5,-0.5) -- (-0.5,0.5);
	\node[below] at (-0.5,-0.5) {\tiny $j$};\node[above] at (0.5,0.5) {\tiny $l$};
	\node[below] at (0.5,-0.5) {\tiny $k$};\node[above] at (-0.5,0.5) {\tiny $i$};
	\node[right] at (0,0) {\tiny $\R^*_z$};
}=\R_{z^{-1}}(i,j;k,l).
\ee
Alternatively, one can recover the same Yang-Baxter equation using \eqref{directionChange}.

There is also another Yang-Baxter equation that involves both $W^\s_{t,s}$ and $w^{\s*}_{u;s}$.
\begin{prop} \label{CauchyYBProp}Assume that $x,y,s,t$ are complex numbers satisfying $xs=yt$. Then for any $a_1, b_1\in\{0,1\}$ and $a_2,b_2,a_3,b_3\in\Z_{\geq 0}$ we have
\begin{equation}
\label{CauchyYB}
\tikzbase{1.2}{-3}{
	\draw[unfused] 
	(2,0) node[below,scale=0.6] {\color{black} $b_1$} -- (1,1)  -- (-1,1) node[left,scale=0.6] {\color{black} $a_1$};
	\draw[fused] 
	(-1,0) node[left,scale=0.6] {\color{black} $a_2$} -- (1,0) -- (2,1) node[above,scale=0.6] {\color{black} $b_2$};
	\draw[fused] 
	(0,-1) node[below,scale=0.6] {\color{black} $a_3$} -- (0,0.5) -- (0,2) node[above,scale=0.6] {\color{black} $b_3$};
	\node[above right] at (0,0) {\tiny{ $W^\s_{t,s}$}};
	\node[above right] at (0,1) {\tiny{ $w^{\s*}_{x;s}$}};
	\node[right] at (1.5,0.5) {\tiny\ {$w^{\s*}_{y;t}$}};
}\qquad
=
\qquad
\tikzbase{1.2}{-3}{
	\draw[unfused]
	(1,-0.5) node[right,scale=0.6] {\color{black} $b_1$} -- (-1,-0.5) -- (-2,0.5) node[above,scale=0.6] {\color{black} $a_1$};
	\draw[fused] 
	(-2,-0.5) node[below,scale=0.6] {\color{black} $a_2$} -- (-1,0.5) -- (1,0.5) node[right,scale=0.6] {\color{black} $b_2$};
	\draw[fused] 
	(0,-1.5) node[below,scale=0.6] {\color{black} $a_3$} -- (0,0) -- (0,1.5) node[above,scale=0.6] {\color{black} $b_3$};
	\node[above right] at (0,-0.5) {\tiny{ $w^{\s*}_{x; s}$}};
	\node[above right] at (0,0.5) {\tiny{ $W^{\s}_{t, s}$}};
	\node[right] at (-1.5,0) {\tiny\ {$w^{\s*}_{y;t}$}};
}
\end{equation}
where the tilted vertices are rotated by $45^\circ$ vertices with the weights $w^{\s*}_{y;t}$. 
\end{prop}
\begin{proof}
There are three ways to prove Proposition \ref{CauchyYBProp}. First, it can be verified by a direct computation, which reduces to four equalities between partition functions corresponding to the values of $a_1,b_1$, with each partition function consisting of at most two summands. Alternatively, using \eqref{directionChange}, we can reduce the Yang-Baxter equation \eqref{CauchyYB} to the equation 
\begin{equation}
\label{CauchyYBChanged}
\tikzbase{1.2}{-3}{
	\draw[unfused] 
	(-1,1) node[left,scale=0.6] {\color{black} $1-a_1$} -- (1,1)  -- (2,0) node[below,scale=0.6] {\color{black} $1-b_1$};
	\draw[fused] 
	(-1,0) node[left,scale=0.6] {\color{black} $a_2$} -- (1,0) -- (2,1) node[above,scale=0.6] {\color{black} $b_2$};
	\draw[fused] 
	(0,-1) node[below,scale=0.6] {\color{black} $a_3$} -- (0,0.5) -- (0,2) node[above,scale=0.6] {\color{black} $b_3$};
	\node[above right] at (0,0) {\tiny{ $W^\s_{t,s}$}};
	\node[above right] at (0,1) {\tiny{ $w^{\s}_{x^{-1};s}$}};
	\node[right] at (1.5,0.5) {\tiny\ {$w^{\s}_{y^{-1};t}$}};
}\qquad
=
\qquad
\tikzbase{1.2}{-3}{
	\draw[unfused]
	(-2,0.5) node[above,scale=0.6] {\color{black} $1-a_1$} -- (-1,-0.5) -- (1,-0.5) node[right,scale=0.6] {\color{black} $1-b_1$};
	\draw[fused] 
	(-2,-0.5) node[below,scale=0.6] {\color{black} $a_2$} -- (-1,0.5) -- (1,0.5) node[right,scale=0.6] {\color{black} $b_2$};
	\draw[fused] 
	(0,-1.5) node[below,scale=0.6] {\color{black} $a_3$} -- (0,0) -- (0,1.5) node[above,scale=0.6] {\color{black} $b_3$};
	\node[above right] at (0,-0.5) {\tiny{ $w^{\s}_{x^{-1}; s}$}};
	\node[above right] at (0,0.5) {\tiny{ $W^{\s}_{t, s}$}};
	\node[right] at (-1.5,0) {\tiny\ {$w^{\s}_{y^{-1};t}$}};
}
\end{equation}
which follows from \eqref{higherSpinYangBaxter} by fusion. Finally, one can see \eqref{CauchyYB} or, more precisely, \eqref{CauchyYBChanged} as a particular case of the Yang-Baxter equation for the weights $\W^{(J)}_{u;s}$:
\begin{multline*}
\sum_{l_1,l_2,l_3}\W^{(J)}_{y;s}(a_3,a_2;l_3,l_2)\W^{(I)}_{x;s}(l_3,a_1;b_3,l_1)\W^{(I)}_{xy^{-1}q^{-J/2};q^{-J/2}}(l_2,l_1;b_2,b_1)\\
=\sum_{l_1,l_2,l_3}\W^{(I)}_{xy^{-1}q^{-J/2};q^{-J/2}}(a_2,a_1;l_2,l_1)\W^{(I)}_{x;s}(a_3,l_1;l_3,b_1)\W^{(J)}_{y;s}(l_3,l_2;b_3,b_2),
\end{multline*}
which in turn can be seen as an analytic continuation of the $n=1$ case of the master Yang-Baxter equation from \cite[C.1.2]{BW18}.
\end{proof}

In this work we also use a dual analogue of the weights $W_{t,s}^\s$. Define
\vspace{1ex}\begin{multline}
\label{dualWWeights}
\tikz{1}{
	\draw[fused] (1,0) -- (-1,0);
	\draw[fused] (0,-1) -- (0,1);
	\node[left] at (-1,0) {\tiny $j$};\node[right] at (1,0) {\tiny $l$};
	\node[below] at (0,-1) {\tiny $i$};\node[above] at (0,1) {\tiny $k$};
	\node[above right] at (0,0) {\tiny $W^{\s*}_{t,s}$};
}=W^{\s*}_{t,s}(i,l;k,j):=s^{-2l}\frac{(q;q)_i}{(s^2;q)_i}\frac{(s^2;q)_k}{(q;q)_k}\cdot t^{2k}\frac{(t^2;q)_j}{(q;q)_j}\frac{(q;q)_l}{(t^2;q)_l} W^\s_{t,s}(k,j;i,l)\\
=\1_{i+l=j+k}\ \1_{i\geq j}\  t^{2i-2j}\frac{(s^2/t^2;q)_{i-j}(t^2;q)_j}{(s^2;q)_i}\frac{(q;q)_i}{(q;q)_{i-j}(q;q)_j},
\end{multline}
where for the last computation we used \eqref{Wdef} and the relation $i-j=k-l$ implied by $\1_{i+l=j+k}$.

\section{Deformations of Yang-Baxter equations} \label{deformationSection}
The Yang-Baxter equations listed in the previous section have the same structure: they change positions of vertices while leaving the vertices themselves and their weights unchanged. In this section we introduce deformed Yang-Baxter equations, which change the weights of vertices in a nontrivial way.

We start with a straightforward deformation of \eqref{WYB}. 

\begin{prop}
\label{defWYBprop}
We have the following equality of rational functions in $t_1,t_2,t_3,\eta$:
\begin{multline*}
\sum_{l_1, l_2, l_3} W^{\s}_{\eta t_2, \eta t_3}(a_3, a_2; l_3, l_2)W^{\s}_{t_1, t_3}(l_3, a_1; b_3, l_1)W^{\s}_{\eta t_1, \eta t_2}(l_2, l_1; b_2, b_1)\\
=\sum_{l_1, l_2, l_3} W^{\s}_{t_1, t_2}(a_2, a_1; l_2, l_1)W^{\s}_{\eta t_1, \eta  t_3}(a_3, l_1; l_3, b_1)W^{\s}_{t_2, t_3}(l_3, l_2; b_3, b_2),
\end{multline*}
or, equivalently,
\begin{equation}
\label{defWYB}
\tikzbase{1.2}{-3}{
	\draw[fused] 
	(-1,1) node[left,scale=0.6] {\color{black} $a_1$} -- (1.2,1) -- (2.2,0) node[below,scale=0.6] {\color{black} $b_1$};
	\draw[fused] 
	(-1,0) node[left,scale=0.6] {\color{black} $a_2$} -- (1.2,0) -- (2.2,1) node[above,scale=0.6] {\color{black} $b_2$};
	\draw[fused] 
	(0,-1) node[below,scale=0.6] {\color{black} $a_3$} -- (0,0.5) -- (0,2) node[above,scale=0.6] {\color{black} $b_3$};
	\node[above right] at (0,0) {\tiny{ $W^\s_{\eta t_2,\eta t_3}$}};
	\node[above right] at (0,1) {\tiny{ $W^\s_{t_1,t_3}$}};
	\node[right] at (1.7,0.5) {\ \tiny{$W^\s_{\eta t_1,\eta t_2}$}};
}\quad
=
\quad
\tikzbase{1.2}{-3}{
	\draw[fused]
	(-2,0.5) node[above,scale=0.6] {\color{black} $a_1$} -- (-1,-0.5) -- (1.5,-0.5) node[right,scale=0.6] {\color{black} $b_1$};
	\draw[fused] 
	(-2,-0.5) node[below,scale=0.6] {\color{black} $a_2$} -- (-1,0.5)  -- (1.5,0.5) node[right,scale=0.6] {\color{black} $b_2$};
	\draw[fused] 
	(0,-1.5) node[below,scale=0.6] {\color{black} $a_3$} -- (0,0) -- (0,1.5) node[above,scale=0.6] {\color{black} $b_3$};
	\node[above right] at (0,-0.5) {\tiny{ $W^\s_{\eta t_1,\eta t_3}$}};
	\node[above right] at (0,0.5) {\tiny{ $W^\s_{t_2,t_3}$}};
	\node[right] at (-1.5,0) {\ \tiny{$W^\s_{t_1,t_2}$}};
}
\end{equation}
where the tilted vertex is a vertex with weights $W^\s$, rotated by $45^\circ$.
\end{prop}
\begin{proof}
Note that 
\be
\tikz{1}{
	\draw[fused] (-1,0) -- (1,0);
	\draw[fused] (0,-1) -- (0,1);
	\node[left] at (-1,0) {\tiny $j$};\node[right] at (1,0) {\tiny $l$};
	\node[below] at (0,-1) {\tiny $i$};\node[above] at (0,1) {\tiny $k$};
	\node[above right] at (0,0.1) {\tiny $W^\s_{\eta t,\eta s}$};
}=\frac{(\eta^2t^2;q)_l}{(t^2;q)_l}\frac{(s^2;q)_i}{(\eta^2 s^2;q)_i}
\tikz{1}{
	\draw[fused] (-1,0) -- (1,0);
	\draw[fused] (0,-1) -- (0,1);
	\node[left] at (-1,0) {\tiny $j$};\node[right] at (1,0) {\tiny $l$};
	\node[below] at (0,-1) {\tiny $i$};\node[above] at (0,1) {\tiny $k$};
	\node[above right] at (0,0.1) {\tiny $W^\s_{t,s}$};
}.
\ee
Applying this equality to \eqref{defWYB} and cancelling common factors corresponding to internal edges, we reduce \eqref{defWYB} to \eqref{WYB} with both sides multiplied by
\be
\frac{(\eta^2t_1^2;q)_{b_1}}{(t_1^2;q)_{b_1}}\frac{(t_3^2;q)_{a_3}}{(\eta^2 t_3^2;q)_{a_3}}.
\ee
\end{proof}

There is also another possible deformation, which is crucial for the main results of this work. This time we deform the Yang-Baxter equation \eqref{CauchyYB}:

\begin{prop}
\label{defCauchyYBprop}
Let $\eta, x,y, s,t$ be parameters satisfying the relation $xs=yt$. The following equality of rational functions holds:
\begin{multline*}
\sum_{l_1,l_2,l_3} W^\s_{t,s}(a_3,a_2;l_3,l_2)w^{\s*}_{\eta x;\eta s}(l_3, l_1; b_3, a_1)w^{\s*}_{y;t}(l_2,b_1;b_2,l_1)\\
=\sum_{l_1,l_2,l_3} w^{\s*}_{\eta y;\eta t}(a_2,l_1;l_2,a_1)w^{\s*}_{x;s}(a_3, b_1; l_3, l_1)W^\s_{t,s}(l_3,l_2;b_3,b_2),
\end{multline*}
or, equivalently,
\begin{equation}
\label{defCauchyYB}
\tikzbase{1.2}{-3}{
	\draw[unfused] 
	(2,0) node[below,scale=0.6] {\color{black} $b_1$} -- (1,1)  -- (-1,1) node[left,scale=0.6] {\color{black} $a_1$};
	\draw[fused] 
	(-1,0) node[left,scale=0.6] {\color{black} $a_2$} -- (1,0) -- (2,1) node[above,scale=0.6] {\color{black} $b_2$};
	\draw[fused] 
	(0,-1) node[below,scale=0.6] {\color{black} $a_3$} -- (0,0.5) -- (0,2) node[above,scale=0.6] {\color{black} $b_3$};
	\node[above right] at (0,0) {\tiny{ $W^\s_{t,s}$}};
	\node[above right] at (0,1) {\tiny{ $w^{\s*}_{\eta x;\eta s}$}};
	\node[right] at (1.5,0.5) {\tiny\ {$w^{\s*}_{y; t}$}};
}\quad
=
\quad
\tikzbase{1.2}{-3}{
	\draw[unfused]
	(1,-0.5) node[right,scale=0.6] {\color{black} $b_1$} -- (-1,-0.5) -- (-2,0.5) node[above,scale=0.6] {\color{black} $a_1$};
	\draw[fused] 
	(-2,-0.5) node[below,scale=0.6] {\color{black} $a_2$} -- (-1,0.5) -- (1,0.5) node[right,scale=0.6] {\color{black} $b_2$};
	\draw[fused] 
	(0,-1.5) node[below,scale=0.6] {\color{black} $a_3$} -- (0,0) -- (0,1.5) node[above,scale=0.6] {\color{black} $b_3$};
	\node[above right] at (0,-0.5) {\tiny{ $w^{\s*}_{x; s}$}};
	\node[above right] at (0,0.5) {\tiny{ $W^\s_{t, s}$}};
	\node[right] at (-1.5,0) {\tiny\ {$w^{\s*}_{\eta y;\eta t}$}};
}
\end{equation}
where the tilted vertex is a vertex with weights $w^{\s*}$, rotated by $45^\circ$.
\end{prop}

The proof is based on the following observation: note that the weights $w^\s_{u;s}(g,j;g+j-l,l)$ and $w^{\s*}_{u;s}(g,l;g+l-j,j)$ depend on $g$ as rational functions in $\alpha=q^g$. This suggests that we can analytically continue the weights to get a model with certain edges having complex labels rather than integer ones, and then specialize the analytically continued complex labels in a specific way to get a deformation of the initial model.

In more detail, consider the vertex weights $w^{\s*}_{u;s}$. Using the explicit values \eqref{dualWeightsTable} of these weights, we can define their analytic continuation denoted by
\vspace{1ex}\be
\tikz{1}{
	\draw[unfused] (1,0) -- (-1,0);
	\draw[cont] (0,-1) -- (0,1);
	\node[left] at (-1,0) {\tiny $j$};\node[right] at (1,0) {\tiny $l$};
	\node[below] at (0,-1) {\tiny $\alpha$};\node[above] at (0,1) {\tiny $\beta$};
	\node[above right] at (0,0) {\tiny $\widetilde{w}^{\s*}_{u;s}$};
}=\widetilde{w}^{\s*}_{u;s}(\alpha,l;\beta,j)
\ee
with $\alpha,\beta$ being complex parameters assigned to edges with analytically continued labels. The values of $\widetilde{w}^{\s*}_{u;s}$ are given by
\begin{align}
\label{contWeights}
\begin{array}{cccc}
\tikz{0.85}{
	\draw[unfused] (1,0) -- (-1,0);
	\draw[cont] (0,-1) -- (0,1);
	\node[left] at (-1,0) {\tiny $0$};\node[right] at (1,0) {\tiny $0$};
	\node[below] at (0,-1) {\tiny $\alpha$};\node[above] at (0,1) {\tiny $\alpha$};
	\node[above right] at (0,0) {\tiny $\widetilde{w}^{\s*}_{u;s}$};
}
\qquad
&
\tikz{0.85}{
	\draw[unfused] (1,0) -- (-1,0);
	\draw[cont] (0,-1) -- (0,1);
	\node[left] at (-1,0) {\tiny $0$};\node[right] at (1,0) {\tiny $1$};
	\node[below] at (0,-1) {\tiny $\alpha$};\node[above] at (0,1) {\tiny $q\alpha$};
	\node[above right] at (0,0) {\tiny $\widetilde{w}^{\s*}_{u;s}$};
}
\qquad
&
\tikz{0.85}{
	\draw[unfused] (1,0) -- (-1,0);
	\draw[cont] (0,-1) -- (0,1);
	\node[left] at (-1,0) {\tiny $1$};\node[right] at (1,0) {\tiny $0$};
	\node[below] at (0,-1) {\tiny $\alpha$};\node[above] at (0,1) {\tiny $q^{-1}\alpha$};
	\node[above right] at (0,0) {\tiny $\widetilde{w}^{\s*}_{u;s}$};
}
\qquad
&
\tikz{0.85}{
	\draw[unfused] (1,0) -- (-1,0);
	\draw[cont] (0,-1) -- (0,1);
	\node[left] at (-1,0) {\tiny $1$};\node[right] at (1,0) {\tiny $1$};
	\node[below] at (0,-1) {\tiny $\alpha$};\node[above] at (0,1) {\tiny $\alpha$};
	\node[above right] at (0,0) {\tiny $\widetilde{w}^{\s*}_{u;s}$};
}
\\
\\
\dfrac{1-su\alpha}{1-su}
\qquad
&
\dfrac{-s^{-1}u(1-s^2\alpha)}{1-su}
\qquad
&
\dfrac{1-\alpha}{1-su}
\qquad
&
\dfrac{\alpha-s^{-1}u}{1-su}
\vspace{1ex}
\end{array}
\end{align}
and all other values are $0$. One can immediately observe that for any $i,k\in\Z_{\geq 0}$ the weight $\widetilde{w}^{\s*}_{u;s}(q^i,l;q^k,j)$ is equal to $w^{\s*}_{u;s}(i,l;k,j)$. Moreover, comparing \eqref{contWeights} with \eqref{dualWeightsTable} we can find a more general specialization, producing $w^{\s*}$-weights with deformed parameters: For any complex $\eta$ we have 
\vspace{1ex}\begin{equation}
\label{contSpec}
\tikz{1}{
	\draw[unfused] (1,0) -- (-1,0);
	\draw[cont] (0,-1) -- (0,1);
	\node[left] at (-1,0) {\tiny $j$};\node[right] at (1,0) {\tiny $l$};
	\node[below] at (0,-1) {\tiny $\eta^{2(1-j)}q^i$};\node[above] at (0,1) {\tiny $\eta^{2(1-j)}q^k$};
	\node[above right] at (0,0) {\tiny $\widetilde{w}^{\s*}_{u;s}$};
}=\frac{1-\eta^2su}{1-su}
\tikz{1}{
	\draw[unfused] (1,0) -- (-1,0);
	\draw[fused] (0,-1) -- (0,1);
	\node[left] at (-1,0) {\tiny $j$};\node[right] at (1,0) {\tiny $l$};
	\node[below] at (0,-1) {\tiny $i$};\node[above] at (0,1) {\tiny $k$};
	\node[above right] at (0,0) {\tiny $w^{\s*}_{\eta u; \eta s}$};
}.
\end{equation}

Similarly, we can perform such a continuation for the weights $W^\s_{t,s}$. Here we have four edges which might participate in this continuation, and for the application below we choose a pair of them, while leaving the other pair untouched. Define weights $\widetilde{W}^\s_{t,s}(i,\alpha;\beta,l)$ by setting
\be
\tikz{1}{
	\draw[fused] (0,-1) -- (0,0) -- (1,0);
	\draw[cont] (-1,0) -- (0,0) -- (0,1);
	\node[left] at (-1,0) {\tiny $\alpha$};\node[right] at (1,0) {\tiny $l$};
	\node[below] at (0,-1) {\tiny $i$};\node[above] at (0,1) {\tiny $q^\Delta \alpha$};
	\node[above right] at (0,0.1) {\tiny $\widetilde{W}^\s_{t,s}$};
}=\widetilde{W}^\s_{t,s}(i,\alpha;q^\Delta\alpha,l):=\1_{\Delta=i-l}\1_{\Delta\geq 0}\ s^{2l}t^{-2l}\frac{(s^2/t^2;q)_{i-l}(t^2;q)_l}{(s^2;q)_i}\frac{(q;q)_i}{(q;q)_{i-l}(q;q)_l}
\ee
for any $\Delta\in\Z$, while $\widetilde{W}^\s_{t,s}(i,\alpha;\beta,l)=0$ for all other choices of $\beta$. Note that these weight do not actually depend on $\alpha$, so we can change $\alpha$ without affecting the weight.

Now we can apply the idea explained above to prove Proposition \ref{defCauchyYBprop}.

\begin{proof}[Proof of Proposition \ref{defCauchyYBprop}]
Fix $a_1,a_3,b_1,b_2\in \mathbb Z_{\geq 0}$ and let $\Delta\in \mathbb Z$. Rewrite the Yang-Baxter equation \eqref{CauchyYB} as
\begin{equation}
\label{contCauchyYB}
\tikzbase{1.2}{-3}{
	\draw[unfused] 
	(2,0) node[below] {\tiny \color{black} $b_1$} -- (1,1)  -- (-1,1) node[left] {\tiny\color{black} $a_1$};
	\draw[fused] 
	(0,-1) node[below] {\tiny \color{black} $a_3$}  -- (0,0) -- (1,0) -- (2,1) node[above] {\tiny\color{black} $b_2$};
	\draw[cont] 
	(-1,0) node[left] {\tiny \color{black} $\alpha$} -- (0,0)  -- (0,2) node[above] {\tiny\color{black} $q^\Delta\alpha$};
	\node[above right] at (0,0) {\tiny{ $\widetilde{W}^\s_{t,s}$}};
	\node[above right] at (0,1) {\tiny{ $\tilde w^{\s*}_{x; s}$}};
	\node[right] at (1.5,0.5) {\tiny\ {$w^{\s*}_{y; t}$}};
}\qquad
=
\qquad
\tikzbase{1.2}{-3}{
	\draw[unfused]
	(1,-0.5) node[right] {\tiny \color{black} $b_1$} -- (-1,-0.5) -- (-2,0.5) node[above] {\tiny \color{black} $a_1$};
	\draw[fused] 
	(0,-1.5) node[below] {\tiny \color{black} $a_3$}  -- (0, 0.5) -- (1,0.5) node[right] {\tiny \color{black} $b_2$};
	\draw[cont] 
	(-2,-0.5) node[below] {\tiny \color{black} $\alpha$} -- (-1,0.5) -- (0, 0.5) -- (0,1.5) node[above] {\tiny \color{black} $q^\Delta\alpha$};
	\node[above right] at (0,-0.5) {\tiny{ $w^{\s*}_{x; s}$}};
	\node[above right] at (0,0.6) {\tiny{ $\widetilde{W}^\s_{t, s}$}};
	\node[right] at (-1.5,0) {\tiny\ {$\tilde w^{\s*}_{y;t}$}};
}
\end{equation}
where $\alpha=q^N$ for sufficiently large $N\in \mathbb Z_{\geq 0}$, and on both sides we are summing over all configurations of internal edges such that the vertex weights are nonzero.\footnote{There is still a finite number of such configurations, even with the analytically continued edges.}  Note that both sides of \eqref{contCauchyYB} are rational functions in $\alpha$, equal at infinitely many points, hence \eqref{contCauchyYB} holds for any $\alpha$.

The proof is finished by \eqref{contSpec} -- multiplying both sides of \eqref{contCauchyYB} by $\frac{1-xs}{1-\eta^2xs}=\frac{1-yt}{1-\eta^2yt}$ and substituting $\alpha=\eta^{2(1-a_1)}q^{a_2}$ readily implies the claim.
\end{proof}

\begin{rem}
Proposition \ref{defWYBprop} can also be obtained via analytic continuation, in a manner similar to Proposition \ref{defCauchyYBprop}. In order to do that one needs to analytically continue the weights $W^\s_{t,s}$ with respect to another pair of labels
\be
\tikz{0.95}{
	\draw[fused] (-1,0) -- (0,0) -- (0,1);
	\draw[cont] (0,-1) -- (0,0) -- (1,0);
	\node[left] at (-1,0) {\tiny $j$};\node[right] at (1,0) {\tiny $\alpha$};
	\node[below] at (0,-1) {\tiny $q^{\Delta}\alpha$};\node[above] at (0,1) {\tiny $k$};
	\node[above right] at (0,0.1) {\tiny $\widehat{W}^\s_{t,s}$};
}:=\1_{\Delta=k-j}\1_{\Delta\geq 0}\ \alpha^{2\log_q(s/t)}\frac{(q\alpha;q)_\Delta(s^2/t^2;q)_{\Delta}}{(q;q)_{\Delta}}\frac{(t^2;q)_\infty}{(\alpha t^2;q)_\infty}\frac{(q^{\Delta}\alpha s^2;q)_\infty}{(s^2;q)_\infty},
\ee
and then use the following continuation of the Yang-Baxter equation \eqref{WYB}:
\be
\tikzbase{1.2}{-3}{
	\draw[fused] 
	(-1,1) node[left,scale=0.6] {\color{black} $a_1$} -- (0,1) -- (0,2) node[above,scale=0.6] {\color{black} $b_3$}; 
	\draw[fused] 
	(-1,0) node[left,scale=0.6] {\color{black} $a_2$} -- (0,0) -- (0,1) --(1,1) --(1.5,0.5) -- (2,1) node[above,scale=0.6] {\color{black} $b_2$};
	\draw[cont] 
	(0,-1) node[below,scale=0.6] {\color{black} $q^{\Delta}\alpha$} -- (0, 0) -- (1,0) -- (1.5,0.5) -- (2,0) node[below,scale=0.6] {\color{black} $\alpha$};
	\node[above right] at (0,0) {\tiny{ $\widehat{W}^\s_{t_2,t_3}$}};
	\node[above right] at (0,1) {\tiny{ $W^\s_{t_1,t_3}$}};
	\node[right] at (1.5,0.5) {\ \tiny{$\widehat{W}^\s_{t_1,t_2}$}};
}\qquad
=
\qquad
\tikzbase{1.2}{-3}{
	\draw[fused]
	(-2,0.5) node[above,scale=0.6] {\color{black} $a_1$} -- (-1.5,0) --  (-1,0.5) -- (0,0.5) -- (0,1.5) node[above,scale=0.6] {\color{black} $b_3$};
	\draw[fused] 
	(-2,-0.5) node[below,scale=0.6] {\color{black} $a_2$} -- (-1.5,0) -- (-1,-0.5) -- (0,-0.5) -- (0,0.5) -- (1,0.5) node[right,scale=0.6] {\color{black} $b_2$};
	\draw[cont] 
	(0,-1.5) node[below,scale=0.6] {\color{black} $q^\Delta\alpha$} -- (0,-0.5)  -- (1,-0.5) node[right,scale=0.6] {\color{black} $\alpha$}; 
	\node[above right] at (0,-0.4) {\tiny{ $\widehat{W}^\s_{t_1,t_3}$}};
	\node[above right] at (0,0.6) {\tiny{ $W^\s_{t_2,t_3}$}};
	\node[right] at (-1.5,0) {\ \tiny{$W^\s_{t_1,t_2}$}};
}
\ee
\end{rem}

\subsection*{Colored deformed Yang-Baxter equations.} As mentioned earlier, the deformed Yang-Baxter equation from Proposition \ref{defCauchyYBprop} plays an essential role for the results of this work. However, initially we have discovered it in a different context, which is briefly described in the remainder of this section. The argument below is not relevant to the rest of this work, so we will omit details, limiting ourselves to a rough outline of the main ideas.

For a fixed $n\in\Z_{\geq 1}$, \emph{colored} vertex models are defined in the same way as the vertex models from Section \ref{prelim}, but with integer labels replaced by $n$-tuples of nonnegative integers called \emph{compositions}. The entries in the compositions are interpreted as numbers of paths of a specific color passing through the edge, hence the term ``colored". The first stochastic version of such a model was introduced in \cite{KMMO16}, and, in the current context, it was studied in \cite{BW18} and \cite{BW20}.

It turns out that one can repeat all the constructions from Section \ref{prelim} getting the colored analogues of the weights $w^\s_{u;s}, W^\s_{t,s}$ and $\W^{(J)}_{u;s}$. For example, the colored version of the weights $w^{\s}_{u;s}$ is explicitly described by the following table:
\begin{align}
\label{Lweights}
\begin{tabular}{|c|c|c|}
\hline
\quad
\tikz{1}{
	\draw[unfused] (-1,0) -- (1,0);
	\draw[fused] (0,-1) -- (0,1);
	\node[left] at (-1,0) {\tiny $\mathbf{0}$};\node[right] at (1,0) {\tiny $\mathbf{0}$};
	\node[below] at (0,-1) {\tiny ${\bm I}$};\node[above] at (0,1) {\tiny ${\bm I}$};
	\node[above right] at (0,0) {\tiny $w^{\mathbf{col}}_{u;s}$};
}
\quad
&
\quad
\tikz{1}{
	\draw[unfused] (-1,0) -- (1,0);
	\draw[fused] (0,-1) -- (0,1);
	\node[left] at (-1,0) {\tiny ${\mathbf e}^i$};\node[right] at (1,0) {\tiny ${\mathbf e}^i$};
	\node[below] at (0,-1) {\tiny ${\bm I}$};\node[above] at (0,1) {\tiny ${\bm I}$};
	\node[above right] at (0,0) {\tiny $w^{\mathbf{col}}_{u;s}$};
}
\quad
&
\quad
\tikz{1}{
	\draw[unfused] (-1,0) -- (1,0);
	\draw[fused] (0,-1) -- (0,1);
	\node[left] at (-1,0) {\tiny $\mathbf{0}$};\node[right] at (1,0) {\tiny ${\mathbf e}^i$};
	\node[below] at (0,-1) {\tiny ${\bm I}$};\node[above] at (0,1) {\tiny ${\bm I}-{\mathbf e}^i$};
	\node[above right] at (0,0) {\tiny $w^{\mathbf{col}}_{u;s}$};
}
\quad
\\[1.8cm]
\quad
$\dfrac{1-s x q^{I_{[1;n]}}}{1-sx}$
\quad
& 
\quad
$\dfrac{(s^2q^{I_i}-sx) q^{I_{[i+1;n]}}}{1-sx}$
\quad
& 
\quad
$\dfrac{sx(q^{I_i}-1) q^{I_{[i+1;n]}}}{1-sx}$
\quad
\\[0.7cm]
\hline
\quad
\tikz{1}{
	\draw[unfused] (-1,0) -- (1,0);
	\draw[fused] (0,-1) -- (0,1);
	\node[left] at (-1,0) {\tiny ${\mathbf e}^i$};\node[right] at (1,0) {\tiny $\mathbf{0}$};
	\node[below] at (0,-1) {\tiny ${\bm I}$};\node[above] at (0,1) {\tiny ${\bm I}+{\mathbf e}^i$};
	\node[above right] at (0,0) {\tiny $w^{\mathbf{col}}_{u;s}$};
}
\quad
&
\quad
\tikz{1}{
	\draw[unfused] (-1,0) -- (1,0);
	\draw[fused] (0,-1) -- (0,1);
	\node[left] at (-1,0) {\tiny ${\mathbf e}^i$};\node[right] at (1,0) {\tiny ${\mathbf e}^j$};
	\node[below] at (0,-1) {\tiny ${\bm I}$};\node[above] at (0,1) {\tiny ${\bm I}+{\mathbf e}^i-{\mathbf e}^j$};
	\node[above right] at (0,0) {\tiny $w^{\mathbf{col}}_{u;s}$};
}
\quad
&
\quad
\tikz{1}{
	\draw[unfused] (-1,0) -- (1,0);
	\draw[fused] (0,-1) -- (0,1);
	\node[left] at (-1,0) {\tiny ${\mathbf e}^j$};\node[right] at (1,0) {\tiny ${\mathbf e}^i$};
	\node[below] at (0,-1) {\tiny ${\bm I}$};\node[above] at (0,1) {\tiny ${\bm I}+{\mathbf e}^j-{\mathbf e}^i$};
	\node[above right] at (0,0) {\tiny $w^{\mathbf{col}}_{u;s}$};
}
\quad
\\[1.8cm] 
\quad
$\dfrac{1-s^2 q^{I_{[1;n]}}}{1-sx}$
\quad
& 
\quad
$\dfrac{sx(q^{I_j}-1) q^{I_{[j+1;n]}}}{1-sx}$
\quad
&
\quad
$\dfrac{s^2(q^{I_i}-1)q^{I_{[i+1;n]}}}{1-sx}$
\quad
\\[0.7cm]
\hline
\end{tabular} 
\end{align}
Here we assume that $i<j$ are integers between $1$ and $n$, ${\bm I}=(I_1,\dots, I_n)$ denotes an $n$-tuple of nonnegative integers, $\mathbf{e}^i$ denotes an $n$-tuple with $1$ at the $i$th position and $0$ elsewhere, and $\mathbf{0}=\mathbf{e}^0:=(0,\dots,0)$. We also use the notation $I_{[i,n]}:=I_i+I_{i+1}+\dots+I_n$ and $|{\bm I}|:=I_1+\dots+I_n$. Note that these weights depend rationally on $q^{I_c}$ for each $c\in\{1,\dots, n\}$, hence we can again perform analytic continuation of labels of the thick edges. The resulting weights satisfy the following property:
\be
\tikz{1}{
	\draw[unfused] (-1,0) -- (1,0);
	\draw[cont] (0,-1) -- (0,1);
	\node[left] at (-1,0) {\tiny $\mathbf{e}^j$};\node[right] at (1,0) {\tiny $\mathbf{e}^l$};
	\node[below] at (0,-1) {\tiny $\restr{\(T_{\eta^2,j}{\bm A}\)}{\alpha_c=q^{I_c}}$};\node[above] at (0,1) {\tiny $\restr{\(T_{\eta^2, j}{\bm B}\)}{\beta_c=q^{K_c}}$};
	\node[above right] at (0,0) {\tiny $\widetilde{w}^{\mathbf{col}}_{u;s}$};
}=
\tikz{1}{
	\draw[unfused] (-1,0) -- (1,0);
	\draw[fused] (0,-1) -- (0,1);
	\node[left] at (-1,0) {\tiny $\mathbf{e}^{j}$};\node[right] at (1,0) {\tiny $\mathbf{e}^l$};
	\node[below] at (0,-1) {\tiny ${\bm I}$};\node[above] at (0,1) {\tiny ${\bm K}$};
	\node[above right] at (0,0) {\tiny $w^{\mathbf{col}}_{u/\eta; \eta s}$};
},
\ee
where ${\bm A}=(\alpha_1, \dots, \alpha_n)$ and ${\bm B}=(\beta_1, \dots, \beta_n)$ are tuples of analytically continued labels,  and $T_{\eta^2, j}$ is the operator that multiplies the $j$th coordinate by $\eta^2$ and leaves the other coordinates unchanged. Using \eqref{directionChange} one can readily see that for $n=1$ the equality above is equivalent to \eqref{contSpec}.

Following the same notation, the colored analogue of the $q$-Hahn weights $W^{\s}_{t,s}$ is given by, \emph{cf.} \cite[(6.17)]{BGW19},
\be
\tikz{1}{
	\draw[fused] (-1,0) -- (1,0);
	\draw[fused] (0,-1) -- (0,1);
	\node[left] at (-1,0) {\tiny ${\bm J}$};\node[right] at (1,0) {\tiny ${\bm L}$};
	\node[below] at (0,-1) {\tiny ${\bm I}$};\node[above] at (0,1) {\tiny ${\bm K}$};
	\node[above right] at (0,0.1) {\tiny $W^{\mathbf{col}}_{t,s}$};
}=\1_{\bm I+\bm J=\bm K+\bm L}\(s^2/t^2\)^{|\bm L|}\frac{(s^2/t^2;q)_{|\bm I|-|\bm L|}(t^2;q)_{|\bm L|}}{(s^2;q)_{|\bm I|}}\prod_{c=1}^n\1_{I_c\geq L_c}\frac{q^{L_{[1,c-1]}(I_c-L_c)}(q;q)_{I_c} }{(q;q)_{L_c}(q;q)_{I_c-L_c}}.
\ee
Similarly to the $n=1$ case, the colored versions of vertex weights above satisfy various forms of the Yang-Baxter equation. It turns out that we can deform some of these equations for the colored model using the same trick of analytic continuation of labels. For instance, the colored analogue of \eqref{CauchyYBChanged} holds, and repeating the same argument as in Proposition \ref{defCauchyYBprop} one can get the following deformed Yang-Baxter equation for $x/s=y/t$:
\be
\tikzbase{1.2}{-3}{
	\draw[unfused] 
	(-1,1) node[left,scale=0.6] {\color{black} $\mathbf{e}^a$} -- (1,1)  -- (2,0) node[below,scale=0.6] {\color{black} $\mathbf{e}^b$};
	\draw[fused] 
	(-1,0) node[left,scale=0.6] {\color{black} $\bm A_1$} -- (1,0) -- (2,1) node[above,scale=0.6] {\color{black} $\bm B_1$};
	\draw[fused] 
	(0,-1) node[below,scale=0.6] {\color{black} $\bm A_2$} -- (0,0.5) -- (0,2) node[above,scale=0.6] {\color{black} $\bm B_2$};
	\node[above right] at (0,0) {\tiny{ $W^{\mathbf{col}}_{t,s}$}};
	\node[above right] at (0,1) {\tiny{ $w^{\mathbf{col}}_{x/\eta;\eta s}$}};
	\node[right] at (1.5,0.5) {\tiny\ {$w^{\mathbf{col}}_{y; t}$}};
}\quad
=
\quad
\tikzbase{1.2}{-3}{
	\draw[unfused]
	(-2,0.5) node[above,scale=0.6] {\color{black} $\mathbf{e}^a$} -- (-1,-0.5) -- (1,-0.5) node[right,scale=0.6] {\color{black} $\mathbf{e}^b$};
	\draw[fused] 
	(-2,-0.5) node[below,scale=0.6] {\color{black} $\bm A_1$} -- (-1,0.5) -- (1,0.5) node[right,scale=0.6] {\color{black} $\bm B_1$};
	\draw[fused] 
	(0,-1.5) node[below,scale=0.6] {\color{black} $\bm A_2$} -- (0,0) -- (0,1.5) node[above,scale=0.6] {\color{black} $\bm B_2$};
	\node[above right] at (0,-0.5) {\tiny{ $w^{\mathbf{col}}_{x; s}$}};
	\node[above right] at (0,0.5) {\tiny{ $W^{\mathbf{col}}_{t, s}$}};
	\node[right] at (-1.5,0) {\tiny\ {$w^{\mathbf{col}}_{y/\eta;\eta t}$}};
}
\ee
For $n=1$ this deformed Yang-Baxter equation is equivalent to Proposition \ref{defCauchyYBprop}.

The deformed equation above can be used in the following way: Rotating the vertex with the weight $W^{\mathbf{col}}_{t,s}$ by $90^\circ$, rescaling and considering the limiting regime $s,x\to 0$, $x/s=y/t=const$, one gets
\begin{equation}
\label{OYB}
\tikzbase{1.2}{-3}{
	\draw[fused]
	(1,-0.5) node[right,scale=0.6] {\color{black} $\bm A_2$} -- (-1,-0.5) -- (-2,0.5) node[above,scale=0.6] {\color{black} $\bm B_2$};
	\draw[unfused] 
	(-2,-0.5) node[below,scale=0.6] {\color{black} $\mathbf{e}^a$} -- (-1,0.5) -- (1,0.5) node[right,scale=0.6] {\color{black} $\mathbf{e}^b$};
	\draw[fused] 
	(0,-1.5) node[below,scale=0.6] {\color{black} $\bm A_1$} -- (0,0) -- (0,1.5) node[above,scale=0.6] {\color{black} $\bm B_1$};
	\node[above right] at (0,-0.5) {\tiny{ $W^{\mathcal{O}}_{y;t}$}};
	\node[above right] at (0,0.5) {\tiny{ $w^{\mathbf{col}}_{y; t}$}};
	\node[right] at (-1.5,0) {\tiny\ {$w^{\mathcal O}_{\eta^2t/y}$}};
}\quad
=
\quad
\tikzbase{1.2}{-3}{
	\draw[fused] 
	(2,0) node[below,scale=0.6] {\color{black} $\bm A_2$} -- (1,1)  -- (-1,1) node[left,scale=0.6] {\color{black} $\bm B_2$};
	\draw[unfused] 
	(-1,0) node[left,scale=0.6] {\color{black} $\mathbf{e}^a$} -- (1,0) -- (2,1) node[above,scale=0.6] {\color{black} $\mathbf{e}^b$};
	\draw[fused] 
	(0,-1) node[below,scale=0.6] {\color{black} $\bm A_1$} -- (0,0.5) -- (0,2) node[above,scale=0.6] {\color{black} $\bm B_1$};
	\node[above right] at (0,0) {\tiny{ $w^{\mathbf{col}}_{y/\eta,t\eta}$}};
	\node[above right] at (0,1) {\tiny{ $W^{\mathcal{O}}_{y;t}$}};
	\node[right] at (1.5,0.5) {\tiny\ {$w^{\mathcal O}_{t/y}$}};
}\end{equation}
where the weights $W^{\mathcal O}_{y;t}$ are given by
\be
\tikz{1}{
	\draw[fused] (1,0) -- (-1,0);
	\draw[fused] (0,-1) -- (0,1);
	\node[left] at (-1,0) {\tiny ${\bm K}$};\node[right] at (1,0) {\tiny ${\bm I}$};
	\node[below] at (0,-1) {\tiny ${\bm J}$};\node[above] at (0,1) {\tiny ${\bm L}$};
	\node[above right] at (0,0.1) {\tiny $W^{\mathcal O}_{y;t}$};
}=\1_{\bm I+\bm J=\bm K+\bm L}\frac{(t^2;q)_{|\bm L|}}{(yt)^{|\bm L|}}\prod_{c=1}^n\frac{q^{L_{c}(I_{[c+1,n]}-L_{[c+1,n]})}(q^{I_c-L_c+1};q)_{L_c} }{(q;q)_{L_c}},
\ee
while the weights $w^{\mathcal O}_z$ are certain degenerations of the weights $w^{\mathbf{col}}_{u;s}$.

Now we can use a standard strategy: given a Yang-Baxter equation like \eqref{OYB}, there is a way to produce a Cauchy-type summation identity by repeatedly applying the Yang-Baxter equation to vertices in row partition functions. Various examples of this procedure are described in \cite{Bor14}, \cite{WZJ15}, \cite{BP16b}, \cite{BW17}, \cite{BW18}, and in the case of varying parameters of the tilted crosses this is performed in Propositions \ref{exchange} and \ref{dualCauchyTheo} below. It turns out that this procedure applied to \eqref{OYB} results in a Cauchy type summation identity of the form
\be
\mathbb{E}_{\mathbf{u}\mid\Xi,\S}(\mathcal O)=f(\mathbf{u}\mid\Xi,\S),
\ee
where the expectation is taken with respect to a certain probability measure depending on parameters $\mathbf{u}, \Xi,\S$ and described in terms of the stochastic (inhomogeneous) colored higher spin six-vertex model, the observable $\mathcal O$ is closely related to $q$-moments of the colored height function of the colored vertex model, and the functions $f$ are partition functions of the colored vertex model wth specific boundary conditions. This identity, in less generality, was first obtained in \cite{BW20} in the case of the \emph{homogeneous} colored vertex model, where the same strategy was used to produce a Cauchy-like summation identity starting from a non-deformed Yang-Baxter equation, \emph{cf.} \cite[Theorem 4.7]{BW20}. Our deformation allows to extend that result to the inhomogeneous case, answering \cite[Remark 6.5]{BW20}. 

However, there now exists a completely different approach used in \cite{BK20} to achieve more general results about $q$-moments of the height functions; thus, we are not pursuing this topic here any further.

\section{Row operators}\label{rowOperatorsSection}
In this section we introduce \emph{row operators}, which form an algebraic framework for our results. The operators depend on one variable, $\kappa$ or $u$ depending on the operator, as well as two infinite sequences of complex inhomogeneity parameters $\Xi=(\xi_0, \xi_1, \xi_2,\dots)$ and $\S=(s_0, s_1,s_2,\dots)$, which we assume to be fixed throughout the section. 

In what follows, all non-tilted vertices in our diagrams will have one of the four types of weights: $w^\s_{u;s}, W^\s_{t,s}, w^{\s*}_{u;s} , W^{\s*}_{t,s}$; the pictorial notation for these vertices will always coincide with the pictures from \eqref{hsVertexDef}, \eqref{Wdef}, \eqref{dualWeights} and \eqref{dualWWeights} respectively. To simplify the figures below, we will interchangeably use the following equivalent notations
\be
\tikz{1}{
	\draw[unfused] (-1,0) -- (1,0);
	\draw[fused] (0,-1) -- (0,1);
	\node[left] at (-1,0) {\tiny $j$};\node[right] at (1,0) {\tiny $l$};
	\node[below] at (0,-1) {\tiny $i$};\node[above] at (0,1) {\tiny $k$};
	\node[above right] at (0,0) {\tiny $w^{\s}_{u;s}$};
}\quad=\quad w^\s\(
\tikz{1}{
	\draw[unfused] (-1,0) -- (1,0);
	\draw[fused] (0,-1) -- (0,1);
	\node[left] at (-1,0) {\tiny $j$};\node[right] at (1,0) {\tiny $l$};
	\node[below] at (0,-1) {\tiny $i$};\node[above] at (0,1) {\tiny $k$};
	\node[above right] at (0,0) {\tiny $(u;s)$};
}\)\quad=\quad
\tikz{1}{
	\draw[unfused] (-1,0) -- (1,0);
	\draw[fused] (0,-1) -- (0,1);
	\node[left] at (-1,0) {\tiny $j$};\node[right] at (1,0) {\tiny $l$};
	\node[below] at (0,-1) {\tiny $i$};\node[above] at (0,1) {\tiny $k$};
	\node[above right] at (0,0) {\tiny $(u;s)$};
}
\ee
for the vertices with weights $w^\s_{u,s}$, and similarly for the other three types of weights. As the thickness and the direction of edges uniquely determine the weight family, this should not cause confusion.

\subsection{Mixed shift.}\label{mixedShiftSection} Before discussing the row operators, we would like to define a certain transformation, which is regularly applied to the sequences $\Xi$ and $\S$ throughout the further text. For a pair of sequences $\mathcal A=(a_0, a_1,\dots)$, $\mathcal B=(b_0, b_1,\dots)$  and an integer $k\in\Z_{\geq 0}$, introduce a \emph{mixed shift} operator $\tau_{\mathcal B}$ by
\begin{equation}
\label{mixedShiftDef}
\tau^k_{\mathcal B}\mathcal A=(a_0^{(k)}, a_1^{(k)},\dots), \quad  a_i^{(k)}=\sqrt{a_{i+k}b_{i+k}a_i/b_i}.
\end{equation}
Note that for any $k,l\geq 0$ we have
\be
\tau^k_{\mathcal B'}\tau^l_{\mathcal B}\mathcal A=\tau^{k+l}_{\mathcal B}\mathcal A,\quad \text{where}\ \mathcal B':=\tau^l_{\mathcal A}\mathcal B.
\ee
We will regularly use this construction with sequences $\mathcal A$ and $\mathcal B$ being equal to $\Xi$ or $\S$. Set
\be
\xi_i^{(k)}:=\sqrt{\xi_{i+k}s_{i+k}\xi_i/s_i},\qquad s_i^{(k)}:=\sqrt{s_{i+k}\xi_{i+k}s_i/\xi_i},
\ee
\be
\tau^k_{\S}\Xi=(\xi_0^{(k)}, \xi_1^{(k)}, \xi_2^{(k)},\dots),\qquad \tau^k_{\Xi}\S=(s_0^{(k)}, s_1^{(k)}, s_2^{(k)},\dots).
\ee
The square roots here are treated formally: for each parameter $\xi_i$ we fix a symbol $\sqrt{\xi_i}$ subject to relation  $\sqrt{\xi_i}^2=\xi_i$, and similarly for $s_i$. One can alternatively treat the parameters $\Xi$ and $\S$ as  complex variables, and fix a branch of the square root. However, for the majority of expressions we do not need square roots at all due to the relations
\be
s_i^{(k)}/\xi_i^{(k)}=s_i/\xi_i,\qquad s_i^{(k)}\xi_i^{(k)}=s_{i+k}\xi_{i+k}.
\ee

\subsection{Row operators $T_i$ and $\T_a$.}\label{rowSect} Define an infinite dimensional vector space $\V:=\Span\{|\lambda\rangle\}_{\lambda\in\mathbb Y}$ with a basis enumerated by partitions (we only allow finite linear combinations of basis vectors). We think of this space as the space of possible configurations of a semi-infinite collection of vertical edges of a vertex model, where $|\lambda\rangle$ corresponds to a configuration such that for any $j\geq 1$ the $j$th edge has  label $m_j(\lambda')=\lambda_j-\lambda_{j+1}$. We also define a space $\V^*$ with the basis $\{\langle\lambda|\}_{\lambda\in\mathbb Y}$ dual to the basis $\{|\lambda\rangle\}_{\lambda\in\mathbb Y}$: $\langle \lambda|\mu\rangle=\1_{\lambda=\mu}$.

Let $\kappa$ be an indeterminate, and for each $a\geq 0$ introduce operators $\T_a(\kappa\mid\Xi,\S)$ in $\End(\V)$ by
 \begin{equation}
 \label{defT}
 \T_a(\kappa\,|\,\Xi,\S):|\mu\rangle\mapsto \sum_{\lambda}W^\s\(
 \tikzbase{0.9}{-2}{
	\draw[fused*] (-1,0) -- (3.9,0);
	\foreach\x in {0,...,1}{
		\draw[fused] (2*\x,-1) -- (2*\x,1);
	}
	\node[left] at (-1,0) {\tiny $a$};
	\node[right] at (4,0) {\dots};
	\draw[fused] (4.8,0) -- (6,0);
	\node[right] at (6,0) {\tiny $0$};
	\node[below] at (0,-1) {\tiny $m_1(\mu')$};\node[above] at (0,1) {\tiny $m_1(\lambda')$};
	\node[below] at (2,-1) {\tiny $m_2(\mu')$};\node[above] at (2,1) {\tiny $m_2(\lambda')$};
	\node[above right] at (0,0) {\tiny $\(\sqrt{\frac{s_1\xi_1}{\kappa}}, s_1 \)$};
	\node[above right] at (2,0) {\tiny $\(\sqrt{\frac{s_2\xi_2}{\kappa}}, s_2 \)$};
    }
\)|\lambda\rangle.
\end{equation}

Here the coefficients are given in terms of partition functions of a semi-infinite row, with the vertex in the $i$th column having the weight $W^\s_{t,s}$ with parameters $(t,s)=\(\sqrt{\frac{s_i\xi_i}{\kappa}}, s_i\)$. Since such vertex weights are extensively used throughout the text, for convenience we rewrite \eqref{Wdef} as
 \begin{equation}
 \label{Wredef}
 W^\s\(\tikz{0.7}{
	\draw[fused] (-1,0) -- (2,0);
	\draw[fused] (0,-1) -- (0,1);
	\node[left] at (-1,0) {\tiny $j$};\node[right] at (2,0) {\tiny $l$};
	\node[below] at (0,-1) {\tiny $i$};\node[above] at (0,1) {\tiny $k$};
	\node[above right] at (0,0) {\tiny ${\(\sqrt{\frac{s\xi}{\kappa}},s\)}$};
}\)=\1_{i+j=k+l}\1_{i\geq l}\ (\kappa s/\xi)^l\frac{(\kappa s/\xi;q)_{i-l}(\kappa^{-1}s\xi;q)_l}{(s^2;q)_i}\frac{(q;q)_i}{(q;q)_{i-l}(q;q)_l}.
\end{equation}

We also need to clarify what we mean by a partition function of the semi-infinite row. Let $\lambda,\mu$ be a pair of partitions with $l(\lambda),l(\mu)<L$ for some $L\in\Z_{\geq 0}$, and consider the partition function of the first $L$ vertices of the row above, with the right outgoing edge having label $0$. Then the $L$th vertex of this finite row has labels $0$ around it and the weight $W^\s$ of this vertex is equal to $1$, hence we can reduce $L$ by $1$ without changing the partition function. So, the partition function of the finite row consisting of the first $L$ vertices does not depend on $L$ as long as  $l(\lambda),l(\mu)\leq L$, so we can drop $L$ from the notation and work with the semi-infinite row, with all but finitely many vertices having weight $1$. Note that the label of the horizontal edge between the $(L-1)$st and the $L$th columns is also uniquely determined by the conservation law in columns $1$, $2,\dots,L-1$, and the fact that it equals $0$ is equivalent to the fact that $\lambda_1=\mu_1+a$. 

The discussion above implies that all coefficients in \eqref{defT} are finite products of weights $W^\s$. By \eqref{Wredef}, all these weights are polynomials in $\kappa$, thus $\langle \lambda|\T_a(\kappa\,|\,\Xi,\S)|\mu\rangle$ is a polynomial in $\kappa$. Moreover, note that by the conservation law in columns $1,2,\dots, r-1$, the label of the horizontal edge immediately to the left of the $r$th column is uniquely determined and is equal to $\lambda_{r}-\mu_r$ (recall that $\lambda_1-\mu_1=a$). Hence, using vanishing of the weights $W^\s_{t,s}(i,j;k,l)$ for $l>i$, we have
 \be
 0\leq \lambda_1-\mu_1,\qquad 0\leq \lambda_r-\mu_r\leq \mu_{r-1}-\mu_r \quad \text{for}\ r\geq 2.
 \ee
 The condition above is equivalent to the \emph{interlacing} $\lambda\succ \mu$ of partitions, that is, $\lambda_1\geq\mu_1\geq\lambda_2\geq\mu_2\geq\dots$. So, we can sum up the discussion above as follows:
 \begin{equation}
 \label{vanishing}
 \langle \lambda | \T_a(\kappa\mid\Xi,\S)|\mu\rangle=0 \qquad \text{unless}\ \mu\prec\lambda\ \text{and}\ \lambda_1=\mu_1+a.
 \end{equation}
 
 Similarly, we define a dual row operator using the vertex weights $W^{\s*}_{s,t}$:
 \begin{equation}
 \label{defDualT}
 \T^*_a(\kappa\mid\Xi,\S):|\lambda\rangle\mapsto \(\frac{\kappa}{s_0\xi_0}\)^{\lambda_1}\sum_{\mu}W^{\s*}\(
 \tikzbase{0.9}{-2}{
	\draw[fused] (3.9,0) -- (-1,0);
	\foreach\x in {0,...,1}{
		\draw[fused] (2*\x,-1) -- (2*\x,1);
	}
	\node[left] at (-1,0) {\tiny $a$};
	\node[right] at (4,0) {\dots};
	\draw[fused*] (4.8,0) -- (6,0);
	\node[right] at (6,0) {\tiny $0$};
	\node[below] at (0,-1) {\tiny $m_1(\lambda')$};\node[above] at (0,1) {\tiny $m_1(\mu')$};
	\node[below] at (2,-1) {\tiny $m_2(\lambda')$};\node[above] at (2,1) {\tiny $m_2(\mu')$};
	\node[above right] at (0,0) {\tiny $\(\sqrt{\frac{s_1\xi_1}{\kappa}}, s_1 \)$};
	\node[above right] at (2,0) {\tiny $\(\sqrt{\frac{s_2\xi_2}{\kappa}}, s_2 \)$};
}
\)|\mu\rangle.
 \end{equation}
 Here the coefficients are given in terms of the partition function of the semi-infinite row of vertices with the weights $W^{\s*}$, where the parameters of $i$th vertex are again $\(\sqrt{\frac{s_i\xi_i}{\kappa}}, s_i\)$. Note that we also multiply the sum by a factor that depends only on $\lambda_1$. It will be needed in Proposition \ref{adjProp} below.
 
 Additionally, we need row operators constructed from vertices with $w^{\s*}$-weights. For $a\in\{0,1\}$ set
 \begin{equation}
 \label{defUnfusedDualT}
 T^*_a(u\mid\Xi,\S):|\lambda\rangle\mapsto \sum_{\mu}w^{\s*}\(
 \tikzbase{0.9}{-2}{
	\draw[unfused] (3.9,0) -- (-1,0);
	\foreach\x in {0,...,1}{
		\draw[fused] (2*\x,-1) -- (2*\x,1);
	}
	\node[left] at (-1,0) {\tiny $a$};
	\node[right] at (4,0) {\dots};
	\draw[unfused*] (4.8,0) -- (6,0);
	\node[right] at (6,0) {\tiny $0$};
	\node[below] at (0,-1) {\tiny $m_1(\lambda')$};\node[above] at (0,1) {\tiny $m_1(\mu')$};
	\node[below] at (2,-1) {\tiny $m_2(\lambda')$};\node[above] at (2,1) {\tiny $m_2(\mu')$};
	\node[above right] at (0,0) {\tiny $\(u\xi_1; s_1 \)$};
	\node[above right] at (2,0) {\tiny $\(u\xi_2; s_2 \)$};
}
\)|\mu\rangle,
 \end{equation}
 where in the partition function we are using a semi-infinite row of vertices with $w^{\s*}$-weights, setting the parameters of the $i$th vertex to $(u\xi_i; s_i)$. As for $\T_a$ and $\T^*_a$, due to the conservation law we can compute each coefficient $\langle\mu| T^*_a(u\mid\Xi,\S)|\lambda\rangle$ using only finite number of vertices.

\subsection{Linear combinations $\C,\B^*$ and $\tB$}\label{linearCombSect} It turns out that instead of working directly with the row operators $\T_a, \T_a^*$ and $T_a$, it is more convenient to define certain linear combinations of them. Here we follow the notation used in \cite{BW17}, and we denote these linear combinations by $\C(\kappa\mid\Xi,\S), \B^*(\kappa\mid\Xi,\S)$ and $\tB^*(u\mid\Xi,\S)$, setting:
\begin{equation}
\label{defC}
\C(\kappa\mid\Xi,\S):=\sum_{a\geq 0}(\kappa s_0/\xi_0)^{a}\frac{(\kappa^{-1}s_0\xi_0;q)_a}{(q;q)_a}\,\T_a(\kappa\mid\Xi,\S),
\end{equation}
\begin{equation}
\B^*(\kappa\mid\Xi,\S)=\sum_{a\geq 0}\T^*_a(\kappa\mid\Xi,\S),
\end{equation}
\begin{equation}
\label{deftB}
\tB^*(u\mid\Xi,\S)=\sum_{a=0}^1(-u\xi_0/s_0)^a\,T^*_a(u\mid\Xi,\S).
\end{equation}

One can readily see that both combinations $\B^*(\kappa\mid\Xi,\S)$ and $\tB^*(u\mid\Xi,\S)$ are actually well-defined linear operators in $\End(\V)$, though the linear combination $\C(\kappa\mid\Xi,\S)$ is not. In particular, compositions like $\tB^*(u\mid\Xi,\S)\C(\kappa\mid\Xi,\S)$ are not \emph{a priori} well-defined.  To partially remedy this issue, we treat $\C(\kappa\mid\Xi,\S)$ as an infinite matrix in $\V$ with respect to the basis $\{|\lambda\rangle\}_\lambda$\footnote{so we will often call $\C(\kappa\mid\Xi,\S)$ an operator}: in view of \eqref{vanishing} we have
\begin{equation}
\label{coefficientC}
\langle \lambda|\C(\kappa\mid\Xi,\S)|\mu\rangle=(\kappa s_0/\xi_0)^{\lambda_1-\mu_1}\frac{(\kappa^{-1} s_0\xi_0;q)_{\lambda_1-\mu_1}}{(q;q)_{\lambda_1-\mu_1}}\langle \lambda|\T_{\lambda_1-\mu_1}(\kappa\mid\Xi,\S)|\mu\rangle.
\end{equation}
Moreover, $\langle \lambda|\C(\kappa\mid\Xi,\S)|\mu\rangle=0$ unless  $\mu\prec\lambda$.

To justify compositions of such operators, consider a space $\widetilde{\End}(\V)$ consisting of infinite matrices $M$ with respect to the basis $\{|\lambda\rangle\}_\lambda$, satisfying
\be
\langle \lambda|M|\mu\rangle=0 \quad\text{unless}\ l(\lambda)\leq l(\mu)+A \ \text{and}\ \lambda_1\geq \mu_1-B,
\ee
for certain integers $A,B\in\Z_{\geq 0}$ depending on $M$. Then the usual matrix multiplication is well defined on the space $\widetilde{\End}(\V)$\footnote{Any matrix element of the product can be computed as a sum over partitions whose Young diagrams fit inside a finite rectangle.}, and $\C(\kappa\mid\Xi,\S)$ can be considered as an element of $\widetilde{\End}(\V)$ with $(A,B)=(1,0)$. Moreover, using the conservation law in a way similar to the derivation of \eqref{vanishing}, we have 
\begin{equation}
\label{Bvanishing}
\langle\lambda|\tB^*(u\mid\Xi,\S)|\mu\rangle=0\quad\text{unless}\ \lambda\prec\mu \ \text{and}\ \mu_1-\lambda_1\in\{0,1\}.
\end{equation}
Thus, $\tB^*(u\mid\Xi,\S)$ is also an element of $\widetilde{\End}(\V)$, with $(A,B)=(0,1)$, so arbitrary compositions of operators  $\C(\kappa\mid\Xi,\S)$ and $\tB^*(u\mid\Xi,\S)$ are well defined as elements of $\widetilde{\End}(\V)$.

The space $\widetilde{\End}(\V)$ provides only a partial resolution to the problem of interpreting $\C(\kappa\mid\Xi,\S)$ as an operator, since in general the operator $\B^*(\kappa\mid\Xi,\S)$ is not an element of $\widetilde{\End}(\V)$. Moreover, it turns out that the computation of the product $\B^*(\chi\mid\Xi,\S)\C(\kappa\mid\Xi,\S)$ in general requires infinite summations, \emph{cf.} Remark \ref{qGauss} below. Since we are not using the operators $\B^*(\kappa\mid\Xi,\S)$ for in a large part of the text, we postpone this issue until Theorem \ref{Cauchy}, where we will treat it using convergence of formal power sums.

The linear combinations $\C(\kappa\mid\Xi,\S)$ and $\B^*(\kappa\mid\Xi,\S)$ are closely related to each other:
\begin{prop}\label{adjProp}
For any pair of partitions $\lambda,\mu$, we have
\begin{equation}
\label{adj}
\langle\mu|\B^*(\kappa\mid\Xi,\S)|\lambda\rangle=\frac{(\tau_\Xi\S)^{2\mu}}{(\S)^{2\lambda}}\frac{\cc_{\tau_\Xi\S}(\mu)}{\cc_\S(\lambda)}\langle\lambda|\C(\kappa\mid\tau_\S\Xi,\tau_\Xi\S)|\mu\rangle,
\end{equation}
where $\tau_\S$ and $\tau_\Xi$ are mixed shift operators, and 
\be
\cc_\S(\lambda):=\prod_{i\geq1}\frac{(s_i^2;q)_{\lambda_i-\lambda_{i+1}}}{(q;q)_{\lambda_i-\lambda_{i+1}}},\qquad (\S)^{2\lambda}:=\prod_{i\geq 1}(s_{i-1})^{2\lambda_i}.
\ee
\end{prop} 
\begin{proof}
For the duration of the proof, set 
\be
\tau_\S\Xi=(\xi'_0,\xi'_1,\xi'_2,\dots),\qquad \tau_\Xi\S=(s'_0,s'_1,s'_2,\dots).
\ee
In particular, we have $\xi'_is'_i=\xi_{i+1}s_{i+1}$ and $\xi'_i/s'_i=\xi_{i}/s_{i}$.

Recall that applying the conservation law to the row partition function $\langle \lambda|\T_a(\kappa\mid\Xi,\S)|\mu\rangle$ we can uniquely determine the labels of horizontal edges: the horizontal edge between columns $c-1$ and $c$ has label $\lambda_c-\mu_c$. Thus, using \eqref{coefficientC}, we can write an explicit expression for $\C(\kappa\mid\tau_\S\Xi,\tau_\Xi\S)$ in terms of the vertex weights $W^\s$:
\begin{equation} \label{CviaWexpression}
\langle\lambda|\C(\kappa\mid\tau_\S\Xi,\tau_\Xi\S)|\mu\rangle=(\kappa s_0/\xi_0)^{\lambda_1-\mu_1}\frac{(\kappa^{-1} s'_0\xi'_0;q)_{\lambda_1-\mu_1}}{(q;q)_{\lambda_1-\mu_1}}\prod_{c\geq 1}W^\s_{\sqrt{s'_{c}\xi'_{c}/\kappa},\,s'_c}\(m_c(\mu'), \lambda_c-\mu_c;m_c(\lambda'), \lambda_{c+1}-\mu_{c+1}\).
\end{equation}
Similarly, applying the same argument for the dual weights, we get 
\begin{equation} \label{BviaWexpression}
\langle\mu|\B^*(\kappa\mid\Xi,\S)|\lambda\rangle=\(\frac{\kappa}{\s_0\xi_0}\)^{\lambda_1}\prod_{c\geq 1}W^{\s*}_{\sqrt{s_{c}\xi_{c}/\kappa},\,s_c}\(m_c(\lambda'), \lambda_{c+1}-\mu_{c+1};m_c(\mu'), \lambda_{c}-\mu_{c}\).
\end{equation}
Now a direct computation using the explicit expressions from \eqref{Wdef} and \eqref{dualWWeights} shows that for any parameters $t_1,t_2$ and $\eta$ we have
\be
W^{\s*}_{t_1,t_2}(i,l;k,j)=t_2^{-2l}\frac{(q;q)_i}{(t_2^2;q)_i}\frac{(\eta^2t_2^2;q)_k}{(q;q)_k}\cdot t_1^{2k}\frac{(t_1^2;q)_j}{(q;q)_j}\frac{(q;q)_l}{(\eta^2t_1^2;q)_l} W^\s_{\eta t_1,\eta t_2}(k,j;i,l).
\ee
Substituting $t_1=\sqrt{s_c\xi_c/\kappa}$, $t_2=s_c$ and $\eta=\sqrt{\frac{s_{c+1}\xi_{c+1}}{s_c\xi_c}}$, we arrive at
\be
W^{\s*}_{\sqrt{s_c\xi_c/\kappa},\ s_c}(i,l;k,j)=s_c^{-2l}(\kappa^{-1}s_c\xi_c)^k\frac{(q;q)_i}{(s_c^2;q)_i}\frac{{(s_c'}^2;q)_k}{(q;q)_k}\frac{(\kappa^{-1}s_c\xi_c;q)_j}{(q;q)_j}\frac{(q;q)_l}{(\kappa^{-1}s'_{c}\xi'_{c};q)_l} W^\s_{\sqrt{s'_{c}\xi'_{c}/\kappa},\ s'_c}(k,j;i,l).
\ee
Applying the equation above to \eqref{CviaWexpression}, we see that all the Pochhammer symbols depending on $j,l$ cancel out due to the relation $s_c'\xi_c'=s_{c+1}\xi_{c+1}$, so we get
\be
\langle\lambda|\C(\kappa\mid\tau_\S\Xi,\tau_\Xi\S)|\mu\rangle=\(\frac{\kappa}{ s_0\xi_0}\)^{\lambda_1}\frac{\(\S\)^{2\lambda}}{\(\tau_\Xi\S\)^{2\mu}}\frac{\cc_{\S}(\lambda)}{\cc_{\tau_\Xi\S}(\mu)}\prod_{c\geq 1}W^{\s*}_{\sqrt{s_{c}\xi_{c}/\kappa}, s_c}\(m_c(\lambda'), \lambda_{c+1}-\mu_{c+1};m_c(\mu'), \lambda_{c}-\mu_{c}\),
\ee
where we have used the relation
\be
\frac{\(\S\)^{2\lambda}}{\(\tau_\Xi\S\)^{2\mu}}=(s_0\xi_0)^{\mu_1} s_0^{2\lambda_1-2\mu_1}\prod_{c=1}^{\infty} s_{c}^{2(\lambda_{c+1}-\mu_{c+1})} \prod_{c=1}^{\infty}\(s_c\xi_c\)^{\mu_{c+1}-\mu_{c}}.
\ee
The proof is finished by comparing with \eqref{BviaWexpression}.
\end{proof}

\begin{rem}
Up to a multiplication by a factor consisting of $(\S)^{2\lambda}$ and $\cc_\S(\lambda)$, our operators $\C(\kappa\mid\Xi,\S)$, $\B^*(\kappa\mid\Xi,\S)$ and $\tB(u\mid\Xi,\S)$ degenerate to the corresponding operators from \cite{BW17} when $\S=(s,s,s,\dots)$ and $\Xi=(1,1,1,\dots)$. We also want to clarify that the usage of the letter $T$ for row operators is not universally standard: for example, \cite{BP16b} and \cite{BW17} use letters $A,B,C,D$ for analogues of such operators. Partially to avoid confusion with this alternative notation we are using the tilde in the notation $\tB^*(u\mid\Xi,\S)$. Another reason is the connection with stable spin Hall-Littlewood functions described in Section \ref{HLsect} below.
\end{rem}

\subsection{Fusion of the operators $\tB^*$} Here we prove that in some cases the operator $\B^*(\kappa\mid\Xi,S)$ can be obtained using operators $\tB^*(u\mid\Xi,\S)$. This will only be used later as a step in the proof of Theorem \ref{Cauchy}.

\begin{prop} \label{Bfusion}For any $J\in\Z_{\geq 1}$ we have
\be
\tB^*(1\mid\S,\S)\tB^*(q\mid\S,\S)\dots \tB^*(q^{J-1}\mid\S,\S)=\B^*(q^{J}\mid\bar\S,\S),
\ee
where 
\be
\bar\S:=(s_0^{-1}, s_1^{-1},\dots).
\ee
\end{prop}
\begin{proof}
This is a reformulation of the fusion procedure described in Section \ref{fusionSection}. We start with an observation that the relation \eqref{dualWeights} between the weights $w^\s_{u;s}$ and $w^{\s*}_{u;s}$ can be stacked vertically to obtain
\begin{equation}
\label{stackedDualWeights}
\tikz{0.8}{
	\foreach\y in {2,...,3}{
		\draw[unfused] (3,\y) -- (1,\y);
	}
	\draw[unfused] (3,4.5) -- (1,4.5);
	\draw[fused*] (2,1) -- (2, 3.5) node[above, scale = 0.6] {\color{black} $\vdots$};
	\draw[fused] (2,4) -- (2,5.5);
	\node[above] at (2,5.5) {$k$};
	\node[left] at (1,2) {$a_1$};
	\node[left] at (1,3) {$a_2$};
	\node[left] at (1,4.5) {$a_{J}$};
	\node[below] at (2,1) {$i$};
	\node[right] at (3,2) {$b_1$};
	\node[right] at (3,3) {$b_2$};
	\node[right] at (3,4.5) {$b_{J}$};
	\node[above right, scale = 0.8] at (2,2) {$w^{\s*}_{u_1;s}$};
	\node[above right, scale = 0.8] at (2,3) {$w^{\s*}_{u_2;s}$};	
	\node[above right, scale = 0.8] at (2,4.5) {$w^{\s*}_{u_J;s}$};	
}
=
s^{-2\sum_{r=1}^Jb_r}\frac{(q;q)_i}{(s^2;q)_i}\frac{(s^2;q)_k}{(q;q)_k}
\(\tikz{0.8}{
	\draw[unfused] (1,2) -- (3,2);
	\draw[unfused] (1,3.5) -- (3,3.5);
	\draw[unfused] (1,4.5) -- (3,4.5);
	\draw[fused*] (2,1) -- (2, 2.5) node[above, scale = 0.6] {\color{black} $\vdots$};
	\draw[fused] (2,3) -- (2,5.5);
	\node[above] at (2,5.5) {$i$};
	\node[left] at (1,2) {$a_J$};
	\node[left] at (1,3.5) {$a_2$};
	\node[left] at (1,4.5) {$a_{1}$};
	\node[below] at (2,1) {$k$};
	\node[right] at (3,2) {$b_J$};
	\node[right] at (3,3.5) {$b_2$};
	\node[right] at (3,4.5) {$b_{1}$};
	\node[above right, scale = 0.8] at (2,2) {$w^\s_{u_J;s}$};
	\node[above right, scale = 0.8] at (2,3.5) {$w^\s_{u_2;s}$};	
	\node[above right, scale = 0.8] at (2,4.5) {$w^\s_{u_1;s}$};	
}\)
\end{equation}
At the same time, using the defining relation \eqref{dualWWeights} for the weights $W^{\s*}_{t,s}$, we have
\be
q^{kJ}\frac{(q^{-J};q)_l}{(q;q)_l}\frac{(q;q)_j}{(q^{-J};q)_j}
\tikz{1}{
	\draw[fused] (1,0) node[right] {\color{black} \tiny $l$} -- (-1, 0) node[left] {\color{black} \tiny $j$};
	\draw[fused] (0,-1) node[below] {\color{black} \tiny $i$} -- (0, 1) node[above] {\color{black} \tiny $k$};
	\node[above right] at (0,0) {\tiny $W^{\s*}_{q^{-J/2},s}$};
}=s^{-2l}\frac{(q;q)_i}{(s^2;q)_i}\frac{(s^2;q)_k}{(q;q)_k}
\tikz{1}{
	\draw[fused] (-1,0) node[left] {\color{black} \tiny $j$} -- (1, 0) node[right] {\color{black} \tiny $l$};
	\draw[fused] (0,-1) node[below] {\color{black} \tiny $k$} -- (0, 1) node[above] {\color{black} \tiny $i$};
	\node[above right] at (0,0.05) {\tiny $W^{\s}_{q^{-J/2},s}$};
}.
\ee
Applying the identities above to \eqref{fusion} and using the simplifying specialization $u=s$ from Section \ref{simplify}, we get
\begin{equation}
\label{dualFusion}(-1)^{l-j}q^{iJ}
\tikz{0.9}{
	\draw[fused] (1,0) node[right] {\color{black} \tiny $l$} -- (-1, 0) node[left] {\color{black} \tiny $j$};
	\draw[fused] (0,-1) node[below] {\color{black} \tiny $i$} -- (0, 1) node[above] {\color{black} \tiny $k$};
	\node[above right] at (0,0) {\tiny $W^{\s*}_{q^{-J/2},s}$};
}
=q^{-\sum_{r=1}^J(r-1)b_r}\sum_{a_1+\dots+a_J=j}q^{\sum_{r=1}^J(r-1)a_r}\(
\tikz{0.8}{
	\draw[unfused] (3.5,2) -- (1,2);
	\draw[unfused] (3.5,3.5) -- (1,3.5);
	\draw[unfused] (3.5,4.5) -- (1,4.5);
	\draw[fused*] (2,1) -- (2, 2.5) node[above, scale = 0.6] {\color{black} $\vdots$};
	\draw[fused] (2,3) -- (2,5.5);
	\node[above] at (2,5.5) {$k$};
	\node[left] at (1,2) {$a_J$};
	\node[left] at (1,3.5) {$a_2$};
	\node[left] at (1,4.5) {$a_{1}$};
	\node[below] at (2,1) {$i$};
	\node[right] at (3.5,2) {$b_J$};
	\node[right] at (3.5,3.5) {$b_2$};
	\node[right] at (3.5,4.5) {$b_{1}$};
	\node[above right, scale = 0.8] at (2,2) {$w^{\s*}_{sq^{J-1},s}$};
	\node[above right, scale = 0.8] at (2,3.5) {$w^{\s*}_{sq,s}$};	
	\node[above right, scale = 0.8] at (2,4.5) {$w^{\s*}_{s,s}$};	
}\),
\end{equation}
where $b_1,\dots, b_J\in\{0,1\}$ satisfy $\sum_{r=1}^Jb_r=l$, and we have used the relation
\be
Z_j(J)=(-1)^jq^{Jj}\frac{(q^{-J};q)_{j}}{(q;q)_{j}}.
\ee

Let $\lambda,\mu$ be partitions and take $L>l(\lambda), l(\mu)$. Similarly to \eqref{multiFusion} we can fuse $L$ dual columns by stacking \eqref{dualFusion} horizontally $L$ times with the parameters $s$ being equal to $s_1,s_2,\dots, s_L$. Setting the bottom (resp. top) boundary condition to a configuration corresponding to $|\lambda\rangle$ (resp. $\langle\mu|$), and assuming that all boundary edges on the right have labels equal to $0$, we see that the parameter $L$ becomes irrelevant and we can take $L\to\infty$ to get, using notations \eqref{defDualT} and  \eqref{defUnfusedDualT},
\begin{equation}
\label{dualFusionOp}
(-1)^j\langle\mu|\T^*_j(q^J\mid\bar\S,\S)|\lambda\rangle=\sum_{a_1+\dots+a_J=j}q^{\sum_{r=1}^J(r-1)a_r}\langle\mu|T_{a_1}^*(1\mid\S,\S) T_{a_2}^*(q\mid\S,\S)\dots T_{a_J}^*(q^{J-1}\mid\S,\S)|\lambda\rangle.
\end{equation}
Recall that
\be
\B^*(q^J\mid\bar\S,\S)=\sum_{j\geq 0}\T_j^*(q^J\mid\bar\S,\S),
\ee
\be
\tB^*(q^{r-1}\mid\S,\S)=\sum_{a_r\in\{0,1\}}(-1)^{a_r}q^{(r-1)a_r}\, T_{a_r}^*(q^{r-1}\mid\S,\S),\quad r=1,\dots,J,
\ee
so the claim follows after multiplying \eqref{dualFusionOp} by $(-1)^j$ and summing over $j$.
\end{proof}

\subsection{Infinite 0th column.} There is a useful interpretation of the linear combinations $\C(\kappa\mid\Xi,\S)$ and $\tB^*(u\mid\Xi,\S)$ in terms of the operators $\T_a$ and $T^*_a$ with an additional 0th column with ``infinite" vertical labels. For a partition $\lambda$ and an integer $N\geq \lambda_1$, let $\widehat{\lambda}^{(N)}$ denote the partition $(N, \lambda_1, \lambda_2,\dots)$, and set $\widehat{\Xi}=(1,\xi_0,\xi_1,\dots)$, $\widehat{\S}=(1,s_0,s_1,\dots)$.

\begin{prop} \label{infiniteColumn}For any partitions $\lambda,\mu$ and any $a\in\Z_{\geq 0}$, we have
\begin{equation*}
\langle \lambda| \C(\kappa\mid\Xi,\S)|\mu\rangle=\frac{(s_0^2;q)_\infty}{(\kappa s_0\xi_0^{-1};q)_{\infty}}\lim_{N\to\infty}\langle \widehat{\lambda}^{(N+a)}| \T_a(\kappa\mid\widehat{\Xi},\widehat{\S})|\widehat{\mu}^{(N)}\rangle,
\end{equation*}
where the convergence holds in the space of formal power series in $q$ or numerically if $|q|<1$. Similarly, for any $i\in\{0,1\}$ we have
\begin{equation*}
\langle\mu| \tB^*(u\mid\Xi,\S)|\lambda\rangle=(1-us_0\xi_0)\lim_{N\to\infty}\langle \widehat{\mu}^{(N-i)}| T^*_i(u\mid\widehat{\Xi},\widehat{\S})|\widehat{\lambda}^{(N)}\rangle,
\end{equation*}
where the convergence holds in the space of formal power series in $q$ or numerically if $|q|<1$.
\end{prop}
\begin{proof}
We start with the first equality. Recall that from \eqref{coefficientC} we have
\be
\langle \lambda| \C(\kappa\mid\Xi,\S)|\mu\rangle=(\kappa s_0/\xi_0)^{\Delta}\frac{(\kappa^{-1} s_0\xi_0;q)_{\Delta}}{(q;q)_{\Delta}}  \langle \lambda| \T_{\Delta}(\kappa\mid\Xi,\S)|\mu\rangle,
\ee
where $\Delta=\lambda_1-\mu_1$. On the other hand, separating the first vertex of the row partition function used to define the operator $\T_a$, we have
\be
\langle \widehat{\lambda}^{(N+a)}| \T_a(\kappa\mid\widehat{\Xi},\widehat{\S})|\widehat{\mu}^{(N)}\rangle=W^\s\(
\tikzbase{0.8}{-4}{
	\draw[fused] (-1,0) -- (2,0);
	\draw[fused] (0,-1) -- (0,1);
	\node[left] at (-1,0) {\tiny $a$};\node[right] at (2,0) {\tiny $\Delta$};
	\node[below] at (0,-1) {\tiny $N-\mu_1$};\node[above] at (0,1) {\tiny $N+a-\lambda_1$};
	\node[above right] at (0,0) {\tiny $\(\sqrt{\frac{\xi_0s_0}{\kappa}},s_0\)$};
}\)\langle\lambda| \T_\Delta(\kappa\mid\Xi,\S)|\mu\rangle.
\ee
Hence, the claim follows from the limiting relation (\emph{cf.} \eqref{Wredef})
\be
\lim_{N\to\infty}W^\s_{\sqrt{\xi_0s_0/\kappa},s_0}(N-\mu_1,i;N-\lambda_1+i,\Delta)=(\kappa s_0/\xi_0)^{\Delta}\frac{(\kappa s_0/\xi_0;q)_{\infty}}{(s_0^2;q)_\infty}\frac{(\kappa^{-1} s_0\xi_0;q)_{\Delta}}{(q;q)_{\Delta}}.
\ee

The other equality follows similarly, using the limiting relation (\emph{cf.} \eqref{dualWeightsTable})
\be
\lim_{N\to\infty}w^{\s*}_{u\xi_0;s_0}(N-\lambda_1,i;N-\mu_1-i,\Delta)=\frac{1}{1-u\xi_0s_0}(-u\xi_0/s_0)^\Delta
\ee
for $i,\Delta\in\{0,1\}$.
\end{proof}

\subsection{Exchange relations.} The vertex weights we are using are distinguished by the Yang-Baxter equation, which makes the models with these weights integrable. In terms of row operators, the Yang-Baxter equation turns into so-called \emph{exchange relations}. In this section we cover the exchange relations between operators $\C(\kappa\mid\Xi,\S)$ and $\tB^*(u\mid\Xi,\S)$, which will result in various algebraic properties of the symmetric functions constructed in the next section. We start with relations resembling commutativity of operators.
\begin{prop}
\label{commutation}
The following relations hold:
\be
\C(\kappa_1\mid\Xi,\S)\C(\kappa_2\mid\tau_\S\Xi,\tau_\Xi\S)=\C(\kappa_2\mid\Xi,\S)\C(\kappa_1\mid\tau_\S\Xi,\tau_\Xi\S),
\ee
\be
\tB^*(u_1\mid\Xi,\S)\tB^*(u_2\mid\Xi,\S)=\tB^*(u_2\mid\Xi,\S)\tB^*(u_1\mid\Xi,\S).
\ee
\end{prop}
\begin{proof}
We start with the second relation. It suffices to prove that for any partitions $\mu,\lambda$ the following relation holds:
\be
\langle\mu|\tB^*(u_1\mid\Xi,\S)\tB^*(u_2\mid\Xi,\S)|\lambda\rangle=\langle\mu|\tB^*(u_2\mid\Xi,\S)\tB^*(u_1\mid\Xi,\S)|\lambda\rangle.
\ee 
Note that the tilted cross weights $\R^*$ in the Yang-Baxter equation \eqref{higherSpinDualYangBaxter} depend only on the ratio $x/y$. Hence, after setting $s=s_r$, $x=u_1\xi_r$, $y=u_2\xi_r$, the Yang-Baxter equation \eqref{higherSpinDualYangBaxter} gives
\be
\tikzbase{1}{-3}{
	\draw[unfused] 
	(2.5,0) node[below,scale=0.6] {\color{black} $b_3$} -- (1.5,1)  -- (-1,1) node[left,scale=0.6] {\color{black} $a_3$};
	\draw[unfused] 
	(2.5,1) node[above,scale=0.6] {\color{black} $b_2$} -- (1.5,0)  -- (-1,0)  node[left,scale=0.6] {\color{black} $a_2$};
	\draw[fused] 
	(0,-1) node[below,scale=0.6] {\color{black} $a_1$} -- (0,0.5)  -- (0,2) node[above,scale=0.6] {\color{black} $b_1$};
	\node[above right] at (0,0) {\tiny{ $(u_1\xi_r,s_r)$}};
	\node[above right] at (0,1) {\tiny{ $(u_2\xi_r,s_r)$}};
	\node[right] at (2,0.5) {\tiny{$\ \R^*_{u_1/u_2}$}};
}\qquad
=
\qquad
\tikzbase{1}{-3}{
	\draw[unfused]
	(1.5,-0.5) node[right,scale=0.6] {\color{black} $b_3$} -- (-1,-0.5)  -- (-2,0.5) node[above,scale=0.6] {\color{black} $a_3$};
	\draw[unfused] 
	(1.5,0.5) node[right,scale=0.6] {\color{black} $b_2$} -- (-1,0.5)  -- (-2,-0.5) node[below,scale=0.6] {\color{black} $a_2$};
	\draw[fused] 
	(0,-1.5) node[below,scale=0.6] {\color{black} $a_1$} -- (0,0)  -- (0,1.5) node[above,scale=0.6] {\color{black} $b_1$};
	\node[above right] at (0,-0.5) {\tiny{ $(u_2\xi_r,s_r)$}};
	\node[above right] at (0,0.5) {\tiny{ $(u_1\xi_r,s_r)$}};
	\node[right] at (-1.5,0) {\tiny{$\ \R^*_{u_1/u_2}$}};
}
\ee
where all non-tilted vertices have $w^{\s*}$-weights. Stacking $L$ columns horizontally and iterating the Yang-Baxter equation above for $r=1,\dots, L$, we can move the tilted cross $\R^*_{u_1/u_2}$ through $L$ columns to get
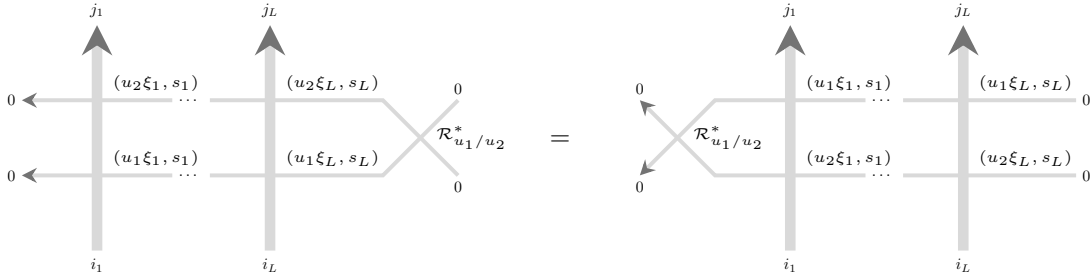
\be
\tikzbase{1}{-3}{
	\draw[unfused] 
         (-1.3,1) node[right,scale=0.6] {\color{black} $\dots$}  --  (-3.3,1) node[left,scale=0.6] {\color{black} $0$};
	\draw[unfused] 
         (-1.3,0) node[right,scale=0.6] {\color{black} $\dots$} -- (-3.3,0) node[left,scale=0.6] {\color{black} $0$}; 
	\draw[unfused*] 
	(2.5,0) node[below,scale=0.6] {\color{black} $0$} -- (1.5,1)  -- (-0.8,1);
	\draw[unfused*] 
	(2.5,1) node[above,scale=0.6] {\color{black} $0$} -- (1.5,0)  -- (-0.8,0);
	\draw[fused] 
	(0,-1) node[below,scale=0.6] {\color{black} $i_L$}   -- (0,2) node[above,scale=0.6] {\color{black} $j_L$};
	\draw[fused] 
	(-2.3,-1) node[below,scale=0.6] {\color{black} $i_1$}  -- (-2.3,2) node[above,scale=0.6] {\color{black} $j_1$};
	\node[above right] at (0,0) {\tiny{ $(u_1\xi_L,s_L)$}};
	\node[above right] at (0,1) {\tiny{ $(u_2\xi_L,s_L)$}};
	\node[above right] at (-2.3,0) {\tiny{ $(u_1\xi_1,s_1)$}};
	\node[above right] at (-2.3,1) {\tiny{ $(u_2\xi_1,s_1)$}};
	\node[right] at (2,0.5) {\tiny{$\ \R^*_{u_1/u_2}$}};
}\quad
=
\qquad
\tikzbase{1}{-3}{
	\draw[unfused] 
         (-1.3,1) node[right,scale=0.6] {\color{black} $\dots$}  --  (-3.3,1) -- (-4.3,0) node[below,scale=0.6] {\color{black} $0$};
	\draw[unfused] 
         (-1.3,0) node[right,scale=0.6] {\color{black} $\dots$} -- (-3.3,0)  -- (-4.3, 1) node[above,scale=0.6] {\color{black} $0$}; 
	\draw[unfused*] 
	(1.5,1) node[right,scale=0.6] {\color{black} $0$}  -- (-0.8,1);
	\draw[unfused*] 
	(1.5,0) node[right,scale=0.6] {\color{black} $0$}  -- (-0.8,0);
	\draw[fused] 
	(0,-1) node[below,scale=0.6] {\color{black} $i_L$}   -- (0,2) node[above,scale=0.6] {\color{black} $j_L$};
	\draw[fused] 
	(-2.3,-1) node[below,scale=0.6] {\color{black} $i_1$}  -- (-2.3,2) node[above,scale=0.6] {\color{black} $j_1$};
	\node[above right] at (0,0) {\tiny{ $(u_2\xi_L,s_L)$}};
	\node[above right] at (0,1) {\tiny{ $(u_1\xi_L,s_L)$}};
	\node[above right] at (-2.3,0) {\tiny{ $(u_2\xi_1,s_1)$}};
	\node[above right] at (-2.3,1) {\tiny{ $(u_1\xi_1,s_1)$}};
	\node[right] at (-3.8,0.5) {\tiny{$\ \R^*_{u_1/u_2}$}};
}\ee
Note that in the identity above we have also specified the left and right boundary conditions to be $0$.  In this case the conservation law forces the tilted crosses on both sides to be trivial of the form 
\tikzbase{0.24}{-2}{
	\draw[unfused] (1,1) node[right] {\color{black} \tiny $0$} -- (-1,-1) node[left] {\color{black} \tiny $0$};
	\draw[unfused] (1,-1) node[right] {\color{black} \tiny $0$} -- (-1,1) node[left] {\color{black} \tiny $0$};
}
 and have weight $1$, whence for any partitions $\lambda,\mu$ such that $l(\lambda), l(\mu)<L$ we get
\be
\sum_\nu \langle \mu|T^*_0(u_2\mid\Xi,\S)|\nu\rangle\langle\nu|T^*_0(u_1\mid\Xi,\S)|\lambda\rangle=\sum_\nu \langle \mu|T^*_0(u_1\mid\Xi,\S)|\nu\rangle\langle\nu|T^*_0(u_2\mid\Xi,\S)|\lambda\rangle,
\ee
with sums on both sides being finite. Since $L$ was arbitrary, we can remove the restrictions on $l(\lambda)$ and $l(\mu)$. One finishes the proof by replacing $\Xi,\S,\lambda,\mu,\nu$ by $\widehat{\Xi},\widehat{\S},\widehat{\lambda}^{(N)},\widehat{\mu}^{(N)},\widehat{\nu}^{(N)}$ and applying Proposition \ref{infiniteColumn}.

To prove the first equation we will use a similar argument, but this time the tilted cross will change while we move it through the columns. Again, it suffices to prove
\be
\langle\lambda|\C(\kappa_1\mid\Xi,\S)\C(\kappa_2\mid\tau_\S\Xi,\tau_\Xi\S)|\mu\rangle=\langle\lambda|\C(\kappa_2\mid\Xi,\S)\C(\kappa_1\mid\tau_\S\Xi,\tau_\Xi\S)|\mu\rangle.
\ee 

This time we use the deformed Yang-Baxter equation from Proposition \ref{defWYBprop}. Setting $t_1=\sqrt{\frac{s_r\xi_r}{\kappa_1}}$,  $t_2=\sqrt{\frac{s_r\xi_r}{\kappa_2}}$, $t_3=s_r$, $\eta=\sqrt{s_{r+1}\xi_{r+1}}/\sqrt{s_r\xi_r}$ in \eqref{defWYB}, we obtain
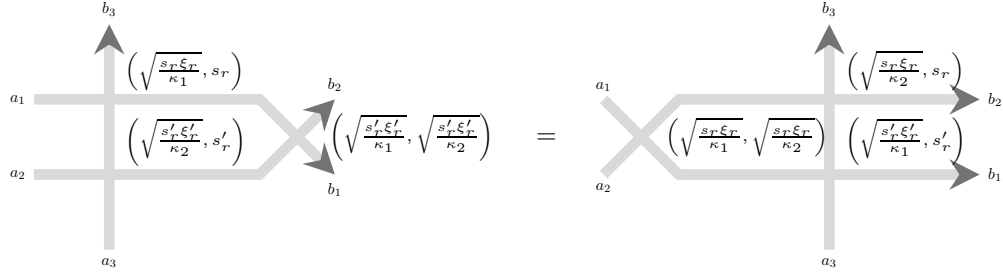
\be
\tikzbase{1}{-3}{
	\draw[fused] 
	(-1,1) node[left,scale=0.6] {\color{black} $a_1$} -- (2,1)  -- (3,0) node[below,scale=0.6] {\color{black} $b_1$};
	\draw[fused] 
	(-1,0)  node[left,scale=0.6] {\color{black} $a_2$} -- (2,0)  -- (3,1) node[above,scale=0.6] {\color{black} $b_2$};
	\draw[fused] 
	(0,-1) node[below,scale=0.6] {\color{black} $a_3$} -- (0,0.5)  -- (0,2) node[above,scale=0.6] {\color{black} $b_3$};
	\node[above right] at (0,1) {\tiny{ $\(\sqrt{\frac{s_r\xi_r}{\kappa_1}},s_r\)$}};
	\node[above right] at (0,0) {\tiny{ $\(\sqrt{\frac{s'_r\xi'_r}{\kappa_2}},s'_r\)$}};
	\node[right] at (2.8,0.5) {\tiny{$\(\sqrt{\frac{s'_r\xi'_r}{\kappa_1}}, \sqrt{\frac{s'_r\xi'_r}{\kappa_2}}\)$}};
}\quad
=\quad
\tikzbase{1}{-3}{
	\draw[fused]
	(-2,0.5) node[above,scale=0.6] {\color{black} $a_1$} -- (-1,-0.5)  -- (3,-0.5) node[right,scale=0.6] {\color{black} $b_1$};
	\draw[fused] 
	(-2,-0.5) node[below,scale=0.6] {\color{black} $a_2$} -- (-1,0.5)  -- (3,0.5) node[right,scale=0.6] {\color{black} $b_2$};
	\draw[fused] 
	(1,-1.5) node[below,scale=0.6] {\color{black} $a_3$} -- (1,0)  -- (1,1.5) node[above,scale=0.6] {\color{black} $b_3$};
	\node[above right] at (1,-0.5) {\tiny{ $\(\sqrt{\frac{s'_r\xi'_r}{\kappa_1}},s'_r\)$}};
	\node[above right] at (1,0.5) {\tiny{ $\(\sqrt{\frac{s_r\xi_r}{\kappa_2}},s_r\)$}};
	\node[right] at (-1.3,0) {\tiny{$\(\sqrt{\frac{s_r\xi_r}{\kappa_1}}, \sqrt{\frac{s_r\xi_r}{\kappa_2}}\)$}};
}
\ee
where all the vertices use $W^\s$-weights with indicated parameters, and we set
\be
\tau_\S\Xi=(\xi_0',\xi_1',\dots),\qquad \tau_\Xi\S=(s_0',s_1',\dots).
\ee
Note that due to the equality $\xi'_rs'_r=\xi_{r+1}s_{r+1}$ the weights of the tilted cross on the left-hand side for $r=i$ coincide with the weights of the tilted cross on the right-hand side for $r=i+1$. This suggests that we can stack columns with $r=1,\dots, L$ horizontally, place a tilted cross with parameters $\(\sqrt{\frac{s_{L+1}\xi_{L+1}}{\kappa_1}}, \sqrt{\frac{s_{L+1}\xi_{L+1}}{\kappa_2}}\)$ on the right, and iterate the equation above, moving the tilted cross and changing the parameters of it to $\(\sqrt{\frac{s_{1}\xi_{1}}{\kappa_1}}, \sqrt{\frac{s_{1}\xi_{1}}{\kappa_2}}\)$. The resulting equation with $0$ boundary conditions on the left and on the right is given below:
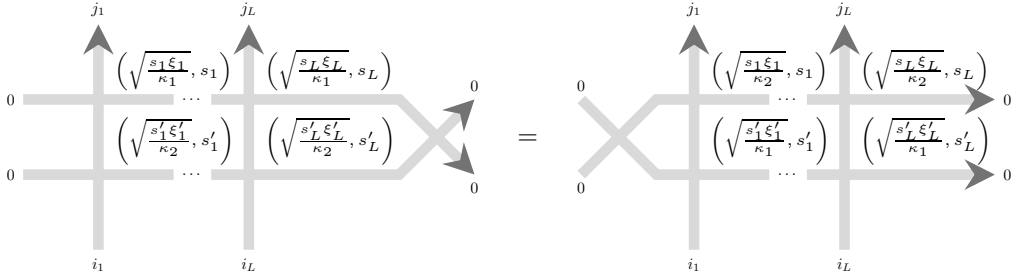
\be
\tikzbase{1}{-3}{
	\draw[fused*] 
	(-3,1) node[left,scale=0.6] {\color{black} $0$} -- (-1,1) node[right,scale=0.6] {\color{black} $\dots$};
	\draw[fused*] 
	(-3,0)  node[left,scale=0.6] {\color{black} $0$} -- (-1,0) node[right,scale=0.6] {\color{black} $\dots$};
	\draw[fused] 
	(-0.5,1) -- (2,1)  -- (3,0) node[below,scale=0.6] {\color{black} $0$};
	\draw[fused] 
	(-0.5,0)  -- (2,0)  -- (3,1) node[above,scale=0.6] {\color{black} $0$};
	\draw[fused] 
	(0,-1) node[below,scale=0.6] {\color{black} $i_L$} -- (0,2) node[above,scale=0.6] {\color{black} $j_L$};
	\draw[fused] 
	(-2,-1) node[below,scale=0.6] {\color{black} $i_1$} -- (-2,2) node[above,scale=0.6] {\color{black} $j_1$};
	\node[above right] at (0,1) {\tiny{ $\(\sqrt{\frac{s_L\xi_L}{\kappa_1}},s_L\)$}};
	\node[above right] at (0,0) {\tiny{ $\(\sqrt{\frac{s'_L\xi'_L}{\kappa_2}},s'_L\)$}};
	\node[above right] at (-2,1) {\tiny{ $\(\sqrt{\frac{s_1\xi_1}{\kappa_1}},s_1\)$}};
	\node[above right] at (-2,0) {\tiny{ $\(\sqrt{\frac{s'_1\xi'_1}{\kappa_2}},s'_1\)$}};
}\quad
=\quad
\tikzbase{1}{-3}{
	\draw[fused*]
	(-0.5,0.5) node[above,scale=0.6] {\color{black} $0$} -- (0.5,-0.5)  -- (2,-0.5) node[right,scale=0.6] {\color{black} $\dots$};
	\draw[fused*] 
	(-0.5,-0.5) node[below,scale=0.6] {\color{black} $0$} -- (0.5,0.5)  -- (2,0.5) node[right,scale=0.6] {\color{black} $\dots$};
	\draw[fused]
	(2.5,-0.5)  -- (5,-0.5) node[right,scale=0.6] {\color{black} $0$};
	\draw[fused] 
	(2.5,0.5)  -- (5,0.5) node[right,scale=0.6] {\color{black} $0$};
	\draw[fused] 
	(1,-1.5) node[below,scale=0.6] {\color{black} $i_1$} -- (1,1.5) node[above,scale=0.6] {\color{black} $j_1$};
	\draw[fused] 
	(3,-1.5) node[below,scale=0.6] {\color{black} $i_L$} -- (3,1.5) node[above,scale=0.6] {\color{black} $j_L$};
	\node[above right] at (1,-0.5) {\tiny{ $\(\sqrt{\frac{s'_1\xi'_1}{\kappa_1}},s'_1\)$}};
	\node[above right] at (1,0.5) {\tiny{ $\(\sqrt{\frac{s_1\xi_1}{\kappa_2}},s_1\)$}};
	\node[above right] at (3,-0.5) {\tiny{ $\(\sqrt{\frac{s'_L\xi'_L}{\kappa_1}},s'_L\)$}};
	\node[above right] at (3,0.5) {\tiny{ $\(\sqrt{\frac{s_L\xi_L}{\kappa_2}},s_L\)$}};
}
\ee
Note that on both sides of the equation above the tilted cross must again be trivial with weight $1$, hence for any partitions $\lambda,\mu$ such that $l(\lambda), l(\mu)<L$ we obtain
\be
\langle \mu|\T_0(\kappa_1\mid\Xi,\S)\T_0(\kappa_2\mid\tau_\S\Xi,\tau_\Xi\S)|\lambda\rangle=\langle \mu|\T_0(\kappa_2\mid\Xi,\S)\T_0(\kappa_1\mid\tau_\S\Xi,\tau_\Xi\S)|\lambda\rangle.
\ee 
Since $L$ was arbitrary, we can remove the restrictions on $l(\lambda),l(\mu)$. The proof is then finished by replacing $\Xi,\S,\lambda,\mu$ with $\widehat{\Xi},\widehat{\S},\widehat{\lambda}^{(N)},\widehat{\mu}^{(N)}$, multiplying both sides  by
\be
\frac{(s_0^2;q)_\infty(s_0s_1\xi_1/\xi_0;q)_\infty}{(\kappa_1 s_0/\xi_0;q)_{\infty}(\kappa_2 s_0/\xi_0;q)_{\infty}}\,,
\ee and applying Proposition \ref{infiniteColumn}.
\end{proof}

The next exchange relation involves operators $\C(\kappa\mid\Xi,\S)$ and $\tB^*(u\mid\Xi,\S)$.

\begin{prop}\label{exchange} The following exchange relation holds:
\be
\tB^*(u\mid\Xi,\S)\C(\kappa\mid\tau_\S\Xi,\tau_\Xi\S)=\frac{1-u\kappa}{1-u\xi_1s_1}\C(\kappa\mid\tau_\S\Xi,\tau_\Xi\S)\tB^*(u\mid\tau_\S\Xi,\tau_\Xi\S).
\ee
\end{prop}
\begin{proof}
The proof is similar to the proof of Proposition \ref{commutation}: we will take a suitable Yang-Baxter equation, iterate it for several columns, choose certain specific boundary conditions, and use Proposition \ref{infiniteColumn}. For the duration of the proof we again set
\be
\tau_\S\Xi=(\xi_0',\xi_1',\dots),\qquad \tau_\Xi\S=(s_0',s_1',\dots).
\ee

We start with the deformed Yang-Baxter equation from Proposition \ref{defCauchyYBprop}, where we set 
\be
s=s'_r=\sqrt{s_rs_{r+1}\xi_{r+1}/\xi_r}, \qquad t=\sqrt{\frac{s_{r+1}\xi_{r+1}}{\kappa}},
\ee 
\be
x=u\xi_r'=u\sqrt{\xi_r\xi_{r+1}s_{r+1}/s_r},\quad  y=u\sqrt{\kappa\xi_{r+1}s_{r+1}},\quad \eta=\sqrt{\frac{\xi_{r}s_{r}}{\xi_{r+1}s_{r+1}}}
\ee
to get
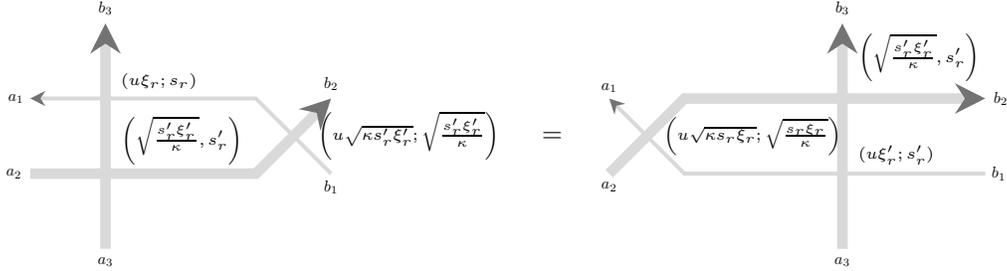
\begin{equation*}
\tikzbase{1}{-3}{
	\draw[unfused] 
	(3,0) node[below,scale=0.6] {\color{black} $b_1$} -- (2,1)  -- (-1,1) node[left,scale=0.6] {\color{black} $a_1$};
	\draw[fused] 
	(-1,0) node[left,scale=0.6] {\color{black} $a_2$} -- (2,0) -- (3,1) node[above,scale=0.6] {\color{black} $b_2$};
	\draw[fused] 
	(0,-1) node[below,scale=0.6] {\color{black} $a_3$} -- (0,2) node[above,scale=0.6] {\color{black} $b_3$};
	\node[above right] at (0,0) {\tiny{ $\(\sqrt{\frac{s_r'\xi_r'}{\kappa}},s_r'\)$}};
	\node[above right] at (0,1) {\tiny{ $(u\xi_r; s_r)$}};
	\node[right] at (2.7,0.5) {\tiny {$\(u\sqrt{\kappa s_r'\xi_r'}; \sqrt{\frac{s_r'\xi_r'}{\kappa}}\)$}};
}\quad
=\quad
\tikzbase{1}{-3}{
	\draw[unfused]
	(2.9,-0.5) node[right,scale=0.6] {\color{black} $b_1$} -- (-1.1,-0.5) -- (-2.1,0.5) node[above,scale=0.6] {\color{black} $a_1$};
	\draw[fused] 
	(-2.1,-0.5) node[below,scale=0.6] {\color{black} $a_2$} -- (-1.1,0.5) -- (2.9,0.5) node[right,scale=0.6] {\color{black} $b_2$};
	\draw[fused] 
	(1,-1.5) node[below,scale=0.6] {\color{black} $a_3$} -- (1,1.5) node[above,scale=0.6] {\color{black} $b_3$};
	\node[above right] at (1,-0.5) {\tiny{ $(u\xi'_r; s_r')$}};
	\node[above right] at (1,0.6) {\tiny{ $\(\sqrt{\frac{s_r'\xi_r'}{\kappa}},s_r'\)$}};
	\node[right] at (-1.5,0) {\tiny{$\(u\sqrt{\kappa s_r\xi_r}; \sqrt{\frac{s_r\xi_r}{\kappa}}\)$}};
}
\end{equation*}
where 
\tikzbase{0.3}{-2}{
	\draw[unfused] (-1,0) -- (1,0);
	\draw[unfused] (0,-1) -- (0,1);
}
 vertices have $W^\s$-weights, and the remaining vertices have $w^{\s*}$-weights. 

Note that the tilted cross in the left-hand side for $r=i$ coincides with the tilted cross in the right-hand side for $r=i+1$, so we can stack several columns horizontally and iterate the Yang-Baxter equation to get
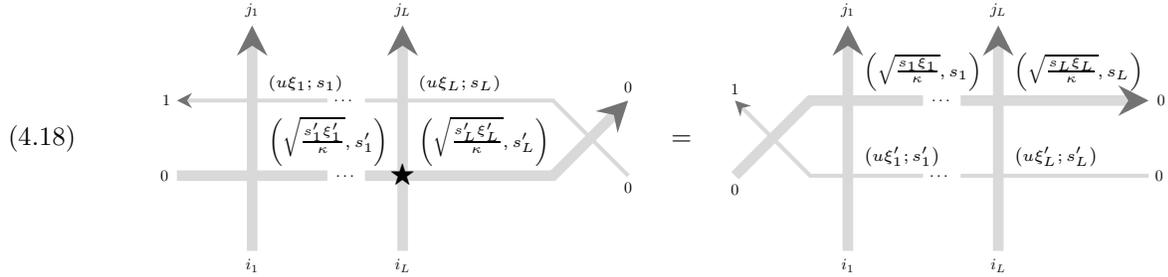
\begin{equation}
\label{BCZipper}
\tikzbase{1}{-3}{
	\draw[unfused] 
	(-1,1) node[right,scale=0.6] {\color{black} $\dots$} -- (-3,1) node[left,scale=0.6] {\color{black} $1$};
	\draw[fused*] 
	(-3,0)  node[left,scale=0.6] {\color{black} $0$} -- (-1,0) node[right,scale=0.6] {\color{black} $\dots$};
	\draw[unfused*] 
	(3,0) node[below,scale=0.6] {\color{black} $0$} -- (2,1)  -- (-0.5,1);
	\draw[fused] 
	(-0.5,0)  -- (2,0)  -- (3,1) node[above,scale=0.6] {\color{black} $0$};
	\draw[fused] 
	(0,-1) node[below,scale=0.6] {\color{black} $i_L$} -- (0,2) node[above,scale=0.6] {\color{black} $j_L$};
	\draw[fused] 
	(-2,-1) node[below,scale=0.6] {\color{black} $i_1$} -- (-2,2) node[above,scale=0.6] {\color{black} $j_1$};
	\node[above right] at (0,1) {\tiny{ $\(u\xi_L; s_L\)$}};
	\node[above right] at (0,0) {\tiny{ $\(\sqrt{\frac{s'_L\xi'_L}{\kappa}},s'_L\)$}};
	\node[above right] at (-2,1) {\tiny{ $\(u\xi_1;s_1\)$}};
	\node[above right] at (-2,0) {\tiny{ $\(\sqrt{\frac{s'_1\xi'_1}{\kappa}},s'_1\)$}};
	\node[] at (0,0) {$\bigstar$};
}\quad
=\quad
\tikzbase{1}{-3}{
	\draw[unfused]
	(2,-0.5) node[right,scale=0.6] {\color{black} $\dots$} -- (0.5,-0.5)  -- (-0.5,0.5) node[above,scale=0.6] {\color{black} $1$};
	\draw[fused*] 
	(-0.5,-0.5) node[below,scale=0.6] {\color{black} $0$} -- (0.5,0.5)  -- (2,0.5) node[right,scale=0.6] {\color{black} $\dots$};
	\draw[unfused*]
	(5,-0.5) node[right,scale=0.6] {\color{black} $0$} -- (2.5,-0.5);
	\draw[fused] 
	(2.5,0.5)  -- (5,0.5) node[right,scale=0.6] {\color{black} $0$};
	\draw[fused] 
	(1,-1.5) node[below,scale=0.6] {\color{black} $i_1$} -- (1,1.5) node[above,scale=0.6] {\color{black} $j_1$};
	\draw[fused] 
	(3,-1.5) node[below,scale=0.6] {\color{black} $i_L$} -- (3,1.5) node[above,scale=0.6] {\color{black} $j_L$};
	\node[above right] at (1,-0.5) {\tiny{ $\(u\xi'_1;s'_1\)$}};
	\node[above right] at (1,0.5) {\tiny{ $\(\sqrt{\frac{s_1\xi_1}{\kappa}},s_1\)$}};
	\node[above right] at (3,-0.5) {\tiny{ $\(u\xi'_L;s'_L\)$}};
	\node[above right] at (3,0.5) {\tiny{ $\(\sqrt{\frac{s_L\xi_L}{\kappa}},s_L\)$}};
}
\end{equation}
Here the parameters of the tilted cross on the left-hand side are $\(u\sqrt{\kappa s_{L+1}\xi_{L+1}}, \sqrt{\frac{s_{L+1}\xi_{L+1}}{\kappa}}\)$, and on the right-hand side they are $\(u\sqrt{\kappa s_1\xi_1}, \sqrt{\frac{s_1\xi_1}{\kappa}}\)$.

Let $\lambda,\mu$ be partitions with $l(\lambda),l(\mu)<L$, and set $i_r=m_r(\mu'), j_r=m_r(\lambda')$. In particular, $i_L=0$, which by \eqref{Wredef} forces the $\star$-vertex on the left-hand side of \eqref{BCZipper} to have configuration \tikzbase{0.24}{-2}{
	\draw[unfused] (-1,0) node[left] {\color{black} \tiny $g$} -- (1,0) node[right] {\color{black} \tiny $0$} -- (1.2,0);
	\draw[unfused] (0,-1) node[below] {\color{black} \tiny $0$} -- (0,1) node[above] {\color{black} \tiny $g$} -- (0, 1.2);
}. Hence, the tilted cross on the left-hand side must be trivial with the weight $1$. On the other hand, the only possible configuration for the tilted cross on the right-hand side of \eqref{BCZipper} is 
\tikzbase{0.24}{-2}{
	\draw[unfused] (-1,-1) -- (-0.7,-0.7) node[below left] {\color{black} \tiny $0$} -- (0.7,0.7) node[above right] {\color{black} \tiny $0$} -- (1,1);
	\draw[unfused] (1,-1) -- (0.7,-0.7) node[below right] {\color{black} \tiny $1$} -- (-0.7,0.7) node[above left] {\color{black} \tiny $1$} -- (-1,1);
} with weight $\dfrac{1-u\kappa}{1-u\xi_1s_1}$. Hence, we get
\be
\langle \lambda| T_1^*(u\mid\Xi,\S)\T_0(\kappa\mid\tau_\S\Xi,\tau_\Xi\S)|\mu\rangle=\frac{1-u\kappa}{1-u\xi_1s_1}\langle \lambda| \T_0(\kappa\mid\tau_\S\Xi,\tau_\Xi\S)T_1^*(u\mid\tau_\S\Xi,\tau_\Xi\S)|\mu\rangle.
\ee
We complete the proof by the following application of Proposition \ref{infiniteColumn}:
\begin{multline*}
\langle \lambda| \tB^*(u\mid\Xi,\S)\C(\kappa\mid\tau_\S\Xi,\tau_\Xi\S)|\mu\rangle\\
=(1-u\xi_0s_0)\frac{({s'_0}^2;q)_\infty}{(\kappa s'_0/\xi'_0;q)_\infty}\lim_{N\to\infty}\langle \widehat{\lambda}^{(N-1)}| T_1^*(u\mid\widehat{\Xi},\widehat{\S})\T_0(\kappa\mid\tau_{\widehat{\S}}\widehat{\Xi},\tau_{\widehat{\Xi}}\widehat{\S})|\widehat{\mu}^{(N)}\rangle\\
=(1-u\xi_0s_0)\frac{1-u\kappa}{1-u{\xi}_0{s}_0}\frac{({s'_0}^2;q)_\infty}{(\kappa s'_0/\xi'_0;q)_\infty}\lim_{N\to\infty}\langle \widehat{\lambda}^{(N-1)}|\T_0(\kappa\mid\tau_{\widehat{\S}}\widehat{\Xi},\tau_{\widehat{\Xi}}\widehat{\S})T_1^*(u\mid\tau_{\widehat {\S}}\widehat{\Xi},\tau_{\widehat{\Xi}}\widehat{\S})|\widehat{\mu}^{(N)}\rangle\\
=\frac{1-u\kappa}{1-u\xi_1s_1}\langle\lambda|\C(\kappa\mid\tau_\S\Xi,\tau_\Xi\S)\tB^*(u\mid\tau_\S\Xi,\tau_\Xi\S)|\mu\rangle,
\end{multline*}
where we have used the relation $\xi'_0s'_0=\xi_1s_1$.
\end{proof}

\section{Inhomogeneous spin $q$-Whittaker functions}\label{sqWhitSect}

In this section for any pair of partitions $\lambda,\mu$ we define a function ${\F_{\lambda/\mu}(\kappa_1,\dots, \kappa_n\mid\Xi,\S)}$ in variables $\kappa_1,\dots, \kappa_n$ as a partition function of a certain vertex model with vertices having $W$-weights.  These functions are symmetric polynomials in $\kappa_1,\dots, \kappa_n$, and they depend on inhomogeneity parameters $\Xi=(\xi_0, \xi_1, \xi_2,\dots)$ and $\S=(s_0, s_1,s_2,\dots)$. We also prove several algebraic properties of these functions using the exchange relations from the previous section.

\subsection{Vertex model construction.}\label{FconstSect} For a pair of partitions $\lambda,\mu$ with $l(\lambda),l(\mu)<L$ consider a grid of vertices consisting of $n$ rows and $L$ columns. For convenience, we number the rows from top to bottom, while the columns are numbered from the left to the right, starting from $1$ in both cases. To a vertex at the intersection of the $i$th row and the $j$th column we assign the weight $W^\s_{t,s}$ with parameters $t=\sqrt{\frac{\xi_{i+j}s_{i+j}}{\kappa_i}}$ and $s=\sqrt{s_{i+j}\xi_{i+j}s_j/\xi_j}$. Define boundary conditions as follows:
\begin{itemize}
\item the labels of all right edges are equal to $0$;
\item the left incoming edge in the $i$th row has label $a_i$;
\item the bottom edge in the $j$th column has label $m_j(\mu')=\mu_j-\mu_{j+1}$;
\item the top edge in the $j$th column has label $m_j(\lambda')=\lambda_j-\lambda_{j+1}$.
\end{itemize} 
The resulting partition function $\ZW^{(a_1,\dots, a_n)}_{\lambda/\mu}(\kappa_1,\dots, \kappa_n\mid\Xi,\S)$ is depicted in Figure \ref{partitionFunction}, where we use the notation $s^{(i)}_j:=\sqrt{s_{i+j}\xi_{i+j}s_j/\xi_j}$.

Similarly to the row partition functions from the previous section, the conservation law forces all labels to the right of the $l(\lambda)$th column in the partition function $\ZW^{(a_1,\dots, a_n)}_{\lambda/\mu}(\kappa_1,\dots, \kappa_n\mid\Xi,\S)$ to be $0$, with all vertices to the right of the $l(\lambda)$th column having weight $1$. Hence, the partition function does not actually depend on $L$ and, taking $L\to\infty$, we can instead consider a semi-infinite grid, with $n$ rows and an infinite number of columns. 

Using the partition function $\ZW^{(a_1,\dots, a_n)}_{\lambda/\mu}(\kappa_1,\dots, \kappa_n\mid\Xi,\S)$ we set 
\begin{equation}
\label{defPartitionF}
\F^{\s}_{\lambda/\mu}(\kappa_1,\dots, \kappa_n\mid\Xi,\S):=\sum_{a_1,\dots, a_n}\(\prod_{i=1}^n \({\kappa_is_0/\xi_0}\)^{a_i}\frac{(\kappa_i^{-1} s_i \xi_i;q)_{a_i}}{(q;q)_{a_i}}\)\ZW_{\lambda/\mu}^{(a_1,\dots, a_n)}(\kappa_1,\dots,\kappa_n\mid\Xi,\S),
\end{equation}
where the sum is over $(a_1,\dots, a_n)\in\Z_{\geq 0}^n$ such that $\sum_{r=1}^na_r=\lambda_1-\mu_1$. Renormalizing, we define \emph{inhomogeneous spin $q$-Whittaker} functions by
\begin{equation}
\label{Frenorm}
\F_{\lambda/\mu}(\kappa_1,\dots, \kappa_n\mid\Xi,\S):=\frac{(-\tau_\Xi^n\S)^{\mu}}{(-\S)^{\lambda}}\frac{\cc_{\tau^n_{\Xi}\S}(\mu)}{\cc_\S(\lambda)}\F^{\s}_{\lambda/\mu}(\kappa_1,\dots, \kappa_n\mid\Xi,\S),
\end{equation}
where
\be
\cc_\S(\lambda):=\prod_{i\geq1}\frac{(s_i^2;q)_{\lambda_i-\lambda_{i+1}}}{(q;q)_{\lambda_i-\lambda_{i+1}}},\qquad (-\S)^{\lambda}:=\prod_{i\geq 1}(-s_{i-1})^{\lambda_i}.
\ee
In the particular case $\mu=\varnothing$ we also use the following notation\footnote{Throughout the text we freely use $\lambda$ instead of $\lambda/\varnothing$ in similar contexts; this should not cause any confusion.}
\be
\F_{\lambda}(\kappa_1,\dots, \kappa_n\mid\Xi,\S):=\F_{\lambda/\varnothing}(\kappa_1,\dots, \kappa_n\mid\Xi,\S).
\ee

\begin{figure}
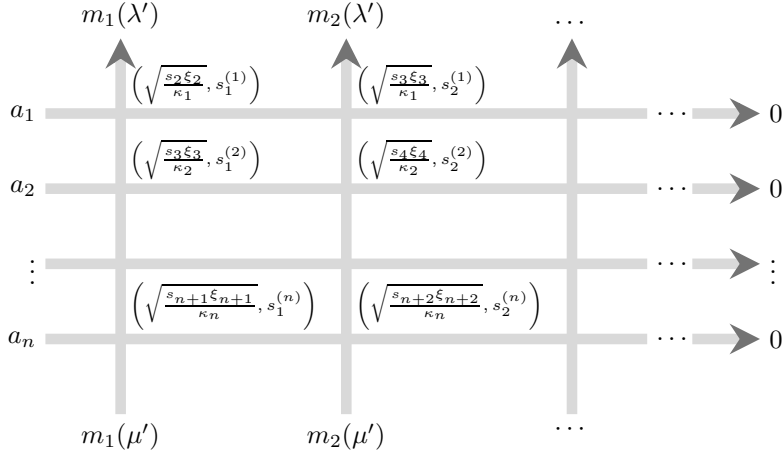

\tikz{1}{
	\foreach\y in {2,...,5}{
		\draw[fused*] (1,\y) -- (9,\y);
		\node[right] at (9,\y) {$\dots$};
		\draw[fused] (9.6,\y) -- (10.5,\y);
	}
	\foreach\x in {1,...,3}{
		\draw[fused] (3*\x-1,1) -- (3*\x-1,6);
	}
	\node[above right] at (2,5) {\tiny $\(\sqrt{\frac{s_2\xi_2}{\kappa_1}}, s^{(1)}_1\)$};
	\node[above right] at (5,5) {\tiny $\(\sqrt{\frac{s_3\xi_3}{\kappa_1}}, s^{(1)}_2\)$};
	
	\node[above right] at (2,4) {\tiny $\(\sqrt{\frac{s_3\xi_3}{\kappa_2}}, s^{(2)}_1\)$};
	\node[above right] at (5,4) {\tiny $\(\sqrt{\frac{s_4\xi_4}{\kappa_2}}, s^{(2)}_2\)$};
	
	\node[above right] at (2,2) {\tiny $\(\sqrt{\frac{s_{n+1}\xi_{n+1}}{\kappa_n}}, s^{(n)}_1\)$};
	\node[above right] at (5,2) {\tiny $\(\sqrt{\frac{s_{n+2}\xi_{n+2}}{\kappa_n}}, s^{(n)}_2\)$};
	\node[above] at (8,6) {$\cdots$};
	\node[above] at (5,6) {$m_2(\lambda')$};
	\node[above] at (2,6) {$m_1(\lambda')$};
	\node[left] at (1,2) {$a_n$};
	\node[left] at (1,3) {$\vdots$};
	\node[left] at (1,4) {$a_2$};
	\node[left] at (1,5) {$a_1$};
	\node[below] at (8,1) {$\cdots$};
	\node[below] at (5,1) {$m_2(\mu')$};
	\node[below] at (2,1) {$m_1(\mu')$};
	\node[right] at (10.5,2) {$0$};
	\node[right] at (10.5,3) {$\vdots$};
	\node[right] at (10.5,4) {$0$};
	\node[right] at (10.5,5) {$0$};
}
\caption{\label{partitionFunction} Partition function $\ZW^{(a_1,\dots, a_n)}_{\lambda/\mu}(\kappa_1,\dots, \kappa_n\mid\Xi,\S)$.}
\end{figure}

\begin{rem} 
When $\Xi=\xi^\infty:=(\xi,\xi,\dots)$ and $\S=s^\infty:=(s,s,\dots)$, the functions $\F_{\lambda/\mu}(\kappa_1,\dots,\kappa_n\mid\Xi, \S)$ turn into the spin $q$-Whittaker functions introduced in \cite{BW17}. More precisely, if we let $\F^{BW}_{\lambda/\mu}(x_1,\dots, x_n)$ denote the functions from \cite{BW17}, then
\be
\F^{BW}_{\lambda/\mu}(x_1,\dots, x_n)=\F_{\lambda/\mu}(-\xi x_1,\dots, -\xi x_n\mid\xi^\infty, s^\infty).
\ee
\end{rem}

 \begin{rem} The letter $\s$ in the notation $\F^\s_{\lambda/\mu}(\kappa_1,\dots, \kappa_n\mid\Xi,\S)$ stands for ``stochastic"; it is explained by the following identity:
\be
\prod_{i=1}^n\frac{(\kappa_is_0\xi_0^{-1};q)_{\infty}}{(s_i\xi_is_0/\xi_0;q)_\infty}\sum_{\lambda}\F^{\s}_\lambda(\kappa_1,\dots, \kappa_n\mid\Xi,\S)=1,
\ee
This easily follows from Remark \ref{stochRem} and the $q$-binomial theorem:
\be
\sum_{a\geq 0}x^a\frac{(y;q)_a}{(q;q)_a}=\frac{(xy;q)_\infty}{(x;q)_{\infty}}.
\ee
\end{rem}

\subsection{Inhomogeneous stable spin Hall-Littlewood functions.}\label{HLsect} It is also useful to consider another family of functions that straightforwardly generalizes both inhomogeneous spin Hall-Littlewood functions from \cite{BP16b} and stable Hall-Littlewood functions from \cite{GdGW16} and \cite{BW17}.

Let $u_1,\dots,u_n$ be a collection of variables. Similarly to the construction of the functions $\F_{\lambda/\mu}$ above, consider a grid of vertices consisting of $n$ rows and $L$ columns, with the vertex at the intersection of the $i$th row and the $j$th column having the $w^\s$-weight with parameters $(u_i\xi_j; s_j)$. For a pair of partitions $\lambda,\mu$ with $l(\lambda),l(\mu)<L$ define boundary conditions in the same way as before:
\begin{itemize}
\item the labels of all right edges are equal to $0$;
\item the left incoming edge in the $i$th row has label $a_i$ for some  $a_i\in\{0,1\}$;
\item the bottom edge in the $j$th column has label $m_j(\mu')=\mu_j-\mu_{j+1}$;
\item the top edge in the $j$th column has label $m_j(\lambda')=\lambda_j-\lambda_{j+1}$.
\end{itemize} 
The resulting partition function is denoted by $\Zw^{(a_1,\dots, a_n)}_{\lambda/\mu}(u_1,\dots, u_n\mid\Xi,\S)$, and its depiction is given in Figure \ref{partitionFunctionw}.

\begin{figure}
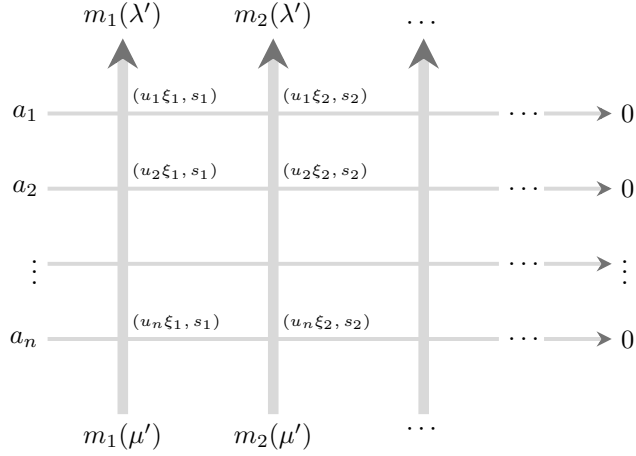

\tikz{1}{
	\foreach\y in {2,...,5}{
		\draw[unfused*] (1,\y) -- (7,\y);
		\node[right] at (7,\y) {$\dots$};
		\draw[unfused] (7.6,\y) -- (8.5,\y);
	}
	\foreach\x in {1,...,3}{
		\draw[fused] (2*\x,1) -- (2*\x,6);
	}
	\node[above right] at (2,5) {\tiny $\(u_1\xi_1, s_1\)$};
	\node[above right] at (4,5) {\tiny $\(u_1\xi_2, s_2\)$};
	
	\node[above right] at (2,4) {\tiny $\(u_2\xi_1, s_1\)$};
	\node[above right] at (4,4) {\tiny $\(u_2\xi_2, s_2\)$};
	
	\node[above right] at (2,2) {\tiny $\(u_n\xi_1, s_1\)$};
	\node[above right] at (4,2) {\tiny $\(u_n\xi_2, s_2\)$};
	\node[above] at (6,6) {$\cdots$};
	\node[above] at (4,6) {$m_2(\lambda')$};
	\node[above] at (2,6) {$m_1(\lambda')$};
	\node[left] at (1,2) {$a_n$};
	\node[left] at (1,3) {$\vdots$};
	\node[left] at (1,4) {$a_2$};
	\node[left] at (1,5) {$a_1$};
	\node[below] at (6,1) {$\cdots$};
	\node[below] at (4,1) {$m_2(\mu')$};
	\node[below] at (2,1) {$m_1(\mu')$};
	\node[right] at (8.5,2) {$0$};
	\node[right] at (8.5,3) {$\vdots$};
	\node[right] at (8.5,4) {$0$};
	\node[right] at (8.5,5) {$0$};
}
\caption{\label{partitionFunctionw} Partition function $\Zw^{(a_1,\dots, a_n)}_{\lambda/\mu}(\kappa_1,\dots, \kappa_n\mid\Xi,\S)$.}
\end{figure}

The \emph{inhomogeneous stable spin Hall-Littlewood functions} $\FF_{\lambda/\mu}(u_1,\dots, u_n\mid\Xi,\S)$ are defined by
\begin{equation}
\label{defPartitionFF}
\FF^{\s}_{\lambda'/\mu'}(u_1,\dots, u_n\mid\Xi,\S):=\sum_{a_1,\dots, a_n\in\{0,1\}}\prod_{i=1}^n(-u\xi_0s_0)^{a_i}\Zw_{\lambda/\mu}^{(a_1,\dots, a_n)}(u_1,\dots,u_n\mid\Xi,\S),
\end{equation}
\be
\FF_{\lambda'/\mu'}(u_1,\dots, u_n\mid\Xi,\S):=\frac{(-\S)^{\mu}}{(-\S)^\lambda}\frac{\cc_\S(\mu)}{\cc_\S(\lambda)}\,\FF^{\s}_{\lambda'/\mu'}(u_1,\dots, u_n\mid\Xi,\S).
\ee
Note that we use conjugated partitions $\lambda'$ and $\mu'$: this is needed to obtain usual Hall-Littlewood polynomials as a degeneration of $\FF_{\lambda/\mu}$, see \cite{BW17}. Also note that the parameter $s_0$ is actually artificial: both $\FF^\s_{\lambda'/\mu'}(u_1,\dots, u_n\mid\Xi,\S)$ and $(-\S)^{\lambda}/(-\S)^\mu$ have factor $s_0^{\lambda_1-\mu_1}$ with no other dependence on $s_0$, thus $\FF_{\lambda'/\mu'}(u_1,\dots, u_n\mid\Xi,\S)$ does not depend on $s_0$ at all.

For $\Xi=1^\infty$ and $\S=s^{\infty}$ the functions $\FF_{\lambda/\mu}(u_1,\dots, u_n\mid\Xi,\S)$ degenerate to the stable spin Hall-Littlewood functions used  in \cite{BW17} and originating from \cite{GdGW16}. Actually, the stability property holds for the inhomogeneous analogue as well: we have
\be
\FF_{\lambda/\mu}(u_1,\dots, u_{n-1}, 0\mid\Xi,\S)=\FF_{\lambda/\mu}(u_1,\dots, u_{n-1}\mid\Xi,\S),
\ee
which can be readily verified using the expression for the partition function $\Zw^{(a_1,\dots, a_n)}_{\lambda/\mu}(\kappa_1,\dots, \kappa_n\mid\Xi,\S)$ and the explicit values of weights $w^{\s}_{u;s}$ for $u=0$. 

At the same time, the  functions $\FF_{\lambda/\mu}(u_1,\dots, u_n\mid\Xi,\S)$ are closely related to the inhomogeneous spin Hall-Littlewood functions ${\sf F}$ defined in \cite[Definition 4.4]{BP16b}. Actually, we already have those functions in our setting:  one has
\be
\frac{(-{\S})^{\mu}}{(-{\S})^\lambda}\frac{\cc_{\widehat{\S}}(\widehat{\mu}^{(N)})}{\cc_{\widehat{\S}}(\widehat{\lambda}^{(N+n)})}\ \Zw^{(1,1,\dots,1)}_{\widehat{\lambda}^{(N+n)}/\widehat{\mu}^{(N)}}(u_1,\dots, u_n\mid\widehat{\Xi},\widehat{\S})={\sf F}_{\lambda'/\mu'}(u_1,\dots, u_n\mid\Xi,\S),
\ee
since both sides are defined using the same partition function, up to a renormalization of the weights mentioned in Remark \ref{renorm} and resulting in the prefactor on the left-hand side. Here $N\in\Z_{\geq 1}$ is a sufficiently large integer, which implicitly enters the right-hand side via the index $\lambda'/\mu'$ of the function ${\sf F}$: In \cite{BP16b} partitions are replaced by signatures, that is, nonincreasing integer sequences, and in the right-hand side we treat  $\lambda'$ and $\mu'$ as signatures of lengths $N+n$ and $N$, respectively, by adding an appropriate number of $0$'s at the end. We also use the notation from Proposition \ref{infiniteColumn} to shift the sequences of parameters $\Xi$ and $\S$ to the right, so that the partition function from Figure \ref{partitionFunctionw} starts with the column having parameters $(u_i\xi_0, s_0)$. 

Moreover, it turns out that one can interpret the functions $\FF_{\lambda/\mu}(u_1,\dots, u_n\mid\Xi,\S)$ as a special case of the inhomogeneous spin Hall-Littlewood functions ${\sf F}_{\lambda/\mu}$. To see it, note that
\be
\restr{\((s_0)^{-l}\ \tikz{0.7}{
	\draw[unfused] (-1,0) -- (1.6,0);
	\draw[fused] (0,-1) -- (0,1);
	\node[left] at (-1,0) {\tiny $1$};\node[right] at (1.6,0) {\tiny $l$};
	\node[below] at (0,-1) {\tiny $i$};\node[above] at (0,1) {\tiny $k$};
	\node[above right] at (0,0) {\tiny ${(u\xi_0,s_0)}$};
}\)}{s_0=0}=\restr{(-u\xi_0s_0)^l}{s_0=1}.
\ee
Applying this to the 0th column of the partition function $\Zw^{(1,1,\dots,1)}_{\widehat{\lambda}^{(N+n)}/\widehat{\mu}^{(N)}}(u_1,\dots, u_n\mid\widehat{\Xi},\widehat{\S})$, one can rewrite \eqref{defPartitionFF} as
\be
\restr{\(s_0^{\mu_1-\lambda_1}\Zw^{(1,1,\dots,1)}_{\widehat{\lambda}(N+n)/\widehat{\mu}(N)}(u_1,\dots, u_n\mid\widehat{\Xi},\widehat{\S})\)}{s_0=0}=\restr{\FF^{\s}_{\lambda'/\mu'}(u_1,\dots, u_n\mid\Xi,\S)}{s_0=1}
\ee
or, using the relation
\be
\frac{(q;q)_{N-\mu_1}}{(q;q)_{N+n-\lambda_1}}\restr{\(s_0^{\lambda_1-\mu_1}\frac{(-{\S})^{\mu}}{(-{\S})^\lambda}\frac{\cc_{\widehat{\S}}(\widehat{\mu}^{(N)})}{\cc_{\widehat{\S}}(\widehat{\lambda}^{(N+n)})}\)}{s_0=0}=\restr{\(\frac{(-\S)^{\mu}}{(-\S)^\lambda}\frac{\cc_\S(\mu)}{\cc_\S(\lambda)}\)}{s_0=1},
\ee
one can renormalize everything to get
\be
\frac{(q;q)_{N-\mu_1}}{(q;q)_{N+n-\lambda_1}}\,\restr{{\sf F}_{\lambda'/\mu'}(u_1,\dots, u_n\mid\Xi,\S)}{s_0=0}=\FF_{\lambda'/\mu'}(u_1,\dots, u_n\mid\Xi,\S),
\ee
where in the left-hand side we again implicitly treat $\lambda'$ and $\mu'$ as signatures of lengths $N+n$ and $N$ respectively. 

The observation above allows us to immediately translate certain results from \cite{BP16b} to our setting. For example, we immediately see that the functions $\FF_{\lambda}(u_1,\dots, u_n\mid\Xi,\S)$ are symmetric in $u_1,\dots,u_n$ since the same result holds for the functions ${\sf F}_{\lambda/\mu}(u_1,\dots, u_n\mid\Xi,\S)$, see \cite[Proposition 4.5]{BP16b}. We also have an explicit formula in the case $\mu=\varnothing$, which is similar to \cite[Theorem 4.14]{BP16b} and \cite[Remark 2]{GdGW16}:
\be
\FF_{\lambda}(u_1,\dots, u_n\mid\Xi,\S)=\frac{1}{(q;q)_{n-l(\lambda)}}\sum_{\sigma\in S_n}\sigma\(\prod_{1\leq\alpha<\beta\leq n}\frac{u_\alpha-qu_\beta}{u_\alpha-u_\beta}\prod_{i=1}^n\widetilde{\varphi}_{\lambda_i}(u_i\mid\Xi,\S)\),
\ee
where the sum is over permutations of $n$ elements that act on functions in $u_1,\dots, u_n$ by permuting variables, and for $k\geq 1$ we set
\be
\widetilde{\varphi}_0(u\mid\Xi,\S)=1-q, \qquad \widetilde{\varphi}_k(u\mid\Xi,\S)=\frac{\xi_0u(1-q)}{1-s_k\xi_ku}\prod_{j=1}^{k-1}\frac{\xi_ju-s_j}{1-s_j\xi_ju}.
\ee

\subsection{Row operator representations}\label{rowoprep} To prove algebraic properties of the functions $\F_{\lambda/\mu}$, instead of using the whole defining partition functions, it is more convenient to construct them row by row using the row operators introduced earlier. 

\begin{prop} For any partitions $\lambda,\mu$ we have
\label{Foperator}
\begin{multline*}
\F^\s_{\lambda/\mu}(\kappa_1,\dots,\kappa_n\mid\Xi,\S)=\langle\lambda|\C(\kappa_1\mid\tau_{\S}\Xi,\tau_{\Xi}\S)\C(\kappa_2\mid\tau^2_{\S}\Xi,\tau^2_{\Xi}\S)\dots\C(\kappa_n\mid\tau^n_{\S}\Xi,\tau^n_{\Xi}\S)|\mu\rangle\\
=\frac{(\S)^{2\lambda}}{(\tau_\Xi\S)^{2\mu}}\frac{\cc_\S(\lambda)}{\cc_{\tau_\Xi\S}(\mu)}\langle\mu|\B^*(\kappa_n\mid\tau^{n-1}_{\S}\Xi,\tau^{n-1}_{\Xi}\S)\dots\B^*(\kappa_2\mid\tau_{\S}\Xi,\tau_{\Xi}\S)\B^*(\kappa_1\mid\Xi,\S)|\lambda\rangle,
\end{multline*}
where $\tau_\Xi$ and $\tau_\S$ are the mixed shift operators defined by \eqref{mixedShiftDef}.
\end{prop}
\begin{proof} The second equality follows at once from Proposition \ref{adjProp}, so we focus on the first equality.

Recall that
\be
\tau^i_{\S}\Xi=(\xi^{(i)}_0, \xi^{(i)}_1, \dots),\qquad \tau^i_{\Xi}\S=(s^{(i)}_0, s^{(i)}_1, \dots),
\ee
where $s^{(i)}_j=\sqrt{s_{i+j}\xi_{i+j}s_j/\xi_j}$, and $\xi^{(i)}_j$ satisfy $\xi^{(i)}_js^{(i)}_j=\xi_{i+j}s_{i+j}$.  Thus, the parameters of the $(i,j)$th vertex in the partition function $ZW_{\lambda/\mu}^{(a_1,\dots, a_n)}(\kappa_1,\dots,\kappa_n\mid\Xi,\S)$ coincide with the parameters of the $j$th vertex in the row of vertices used to define $\T_{a_i}(\kappa_i\mid\tau_{\S}^i\Xi, \tau_{\Xi}^i\S)$, or, equivalently, we have
\be
ZW_{\lambda/\mu}^{(a_1,\dots, a_n)}(\kappa_1,\dots,\kappa_n\mid\Xi,\S)=\langle\lambda|\T_{a_1}(\kappa_1\mid\tau_{\S}\Xi,\tau_{\Xi}\S)\dots\T_{a_n}(\kappa_n\mid\tau^n_{\S}\Xi,\tau^n_{\Xi}\S)|\mu\rangle.
\ee
On the other hand, using relations $s^{(i)}_0\xi^{(i)}_0=s_i\xi_i$ and $s^{(i)}_0/\xi^{(i)}_0=s_0/\xi_0$, we see that the coefficients in \eqref{defPartitionF} coincide with the coefficients in \eqref{defC}, concluding the proof.
\end{proof}

As a corollary we observe that the functions $\F_{\lambda/\mu}(\kappa_1,\dots, \kappa_n\mid\Xi,\S)$ are polynomials in $\kappa_1,\dots, \kappa_n$. Indeed, matrix coefficients of the operator $\C(\kappa\mid\Xi,\S)$ are polynomials in $\kappa$, so Proposition \ref{Foperator} implies polynomiality.  

\begin{prop} \label{FFoperatorProp} For any partitions $\lambda,\mu$ we have
\begin{equation}
\label{FFoperator}
\FF^\s_{\lambda'/\mu'}(u_1,\dots, u_n\mid\Xi,\S)=\frac{\(\S\)^{2\lambda}}{\(\S\)^{2\mu}}\frac{\cc_{\S}(\lambda)}{\cc_{\S}(\mu)}\langle\mu|\tB^*(u_n\mid\Xi,\S)\dots\tB^*(u_2\mid\Xi,\S)\tB^*(u_1\mid\Xi,\S)|\lambda\rangle.
\end{equation}
\end{prop}
\begin{proof}
Similarly to the proof of \eqref{Foperator}, we consider the partition function $\Zw^{(a_1,\dots, a_n)}_{\lambda/\mu}(u_1,\dots, u_n\mid\Xi,\S)$ row by row, interpreting it as a matrix coefficient of a composition of row operators. In addition, here we will change the weights $w^\s$ to the weights $w^{\s*}$ using \eqref{dualWeights}. We have already used this change of weights for a single column during the discussion of fusion of the row operators $\tB^*(u\mid\Xi,\S)$, see \eqref{stackedDualWeights}. 

Noting that for the partition function $\Zw^{(a_1,\dots, a_n)}_{\lambda/\mu}(u_1,\dots, u_n\mid\Xi,\S)$ the sum of all labels of edges between columns $L-1$ and $L$ is equal to $\lambda_{L}-\mu_{L}$, we can use  \eqref{stackedDualWeights} for each column in $\Zw^{(a_1,\dots, a_n)}_{\lambda/\mu}(u_1,\dots, u_n\mid\Xi,\S)$ to get
\be
\Zw^{(a_1,\dots, a_n)}_{\lambda/\mu}(u_1,\dots, u_n\mid\Xi,\S)=\frac{s_0^{-\lambda_1}\(\S\)^{2\lambda}}{s_0^{-\mu_1}\(\S\)^{2\mu}}\frac{\cc_{\S}(\lambda)}{\cc_{\S}(\mu)}\langle\mu|T_{a_n}^*(u_n\mid\Xi,\S)\dots T_{a_2}^*(u_2\mid\Xi,\S)T_{a_1}^*(u_1\mid\Xi,\S)|\lambda\rangle.
\ee
The claim follows by comparing coefficients in \eqref{deftB} and \eqref{defPartitionFF}.
\end{proof}

\subsection{Properties of the inhomogeneous spin $q$-Whittaker functions.} So far we have been establishing an algebraic framework to analyse the functions $\F_{\lambda/\mu}(\kappa_1,\dots, \kappa_n|\mid\Xi,\S)$. Here we finally use it to prove our main results about these functions.
 
 \begin{prop}[Branching rule]\label{branching} For any partitions $\lambda,\mu$ and $1\leq m<n$ we have
 \begin{equation}
 \label{branchingEq}
 \F_{\lambda/\mu}(\kappa_1,\dots, \kappa_n\mid\Xi,\S)=\sum_{\nu} \F_{\lambda/\nu}(\kappa_1,\dots, \kappa_m\mid\Xi,\S)\F_{\nu/\mu}(\kappa_{m+1},\dots, \kappa_n\mid\tau^{m}_\S\Xi,\tau^{m}_\Xi\S),
 \end{equation}
 where the sum on the right-hand side is finite.
 \end{prop}
 \begin{proof}
 A similar claim for the functions $\F^\s_{\lambda/\mu}$  follows at once from Proposition \ref{Foperator} by an insertion of the summation $\sum_{\nu}|\nu\rangle\langle\nu|=id$:
 \begin{multline*}
  \F^\s_{\lambda/\mu}(\kappa_1,\dots, \kappa_n\mid\Xi,\S)=\langle\lambda|\C(\kappa_1\mid\tau_\S\Xi\tau_\Xi\S)\dots\C(\kappa_n\mid\tau^n_\S\Xi,\tau^n_\Xi\S)|\mu\rangle\\
 =\sum_\nu\langle\lambda|\C(\kappa_1\mid\tau_\S\Xi,\tau_\Xi\S)\dots\C(\kappa_m\mid\tau^m_\S\Xi,\tau^m_\Xi\S)|\nu\rangle\langle\nu|\C(\kappa_{m+1}\mid\tau^{m+1}_\S\Xi,\tau^{m+1}_\Xi\S)\dots\C(\kappa_n\mid\tau^n_\S\Xi,\tau^n_\Xi\S)|\mu\rangle\\
 =\sum_{\nu} \F^\s_{\lambda/\nu}(\kappa_1,\dots, \kappa_m\mid\Xi,\S)\F^\s_{\nu/\mu}(\kappa_{m+1},\dots, \kappa_n\mid\tau^{m}_\S\Xi,\tau^{m}_\Xi\S).
 \end{multline*}
 To reach \eqref{branchingEq}, one multiplies the above equality by
  \be
 \frac{(-\tau_\Xi^n\S)^{\mu}}{(-\S)^{\lambda}}\frac{\cc_{\tau^n_{\Xi}\S}(\mu)}{\cc_\S(\lambda)}=\frac{(-\tau_\Xi^n\S)^{\mu}}{(-\tau_\Xi^m\S)^{\nu}}\frac{\cc_{\tau^n_{\Xi}\S}(\mu)}{\cc_{\tau^m_\Xi\S}(\nu)}\cdot\frac{(-\tau_\Xi^m\S)^{\nu}}{(-\S)^{\lambda}}\frac{\cc_{\tau^m_{\Xi}\S}(\nu)}{\cc_\S(\lambda)}.
 \ee
 \end{proof}
 
 \begin{prop}\label{rowValue}We have
 \be
 \F_{\lambda/\mu}(\kappa\mid\Xi,\S)=\begin{cases}
 (-\kappa)^{|\lambda|-|\mu|}{\displaystyle \prod_{i\geq 1}}\xi_{i-1}^{\mu_i-\lambda_i} \(\frac{s_i\xi_i}{s_{i-1}\xi_{i-1}}\)^{\mu_i/2}\dfrac{(\kappa^{-1}s_i\xi_i;q)_{\lambda_i-\mu_i}(\kappa s_i/\xi_i;q)_{\mu_i-\lambda_{i+1}}(q;q)_{\lambda_i-\lambda_{i+1}}}{(q;q)_{\lambda_i-\mu_i}(q;q)_{\mu_i-\lambda_{i+1}}(s_i^2;q)_{\lambda_i-\lambda_{i+1}}}\\
 &\mkern-48mu \text{if}\ \mu\prec\lambda,\\
 0\ \  &\mkern-48mu\text{otherwise}.
 \end{cases}
 \ee
 \end{prop}
\begin{proof}
This is a result of an explicit computation consisting of the following steps: we consider an expression of $\langle\lambda|\C(\kappa\mid\tau_\S\Xi,\tau_\Xi\S)|\mu\rangle$ as a product of the vertex weights $W^\s$, which is done in \eqref{CviaWexpression} by using the conservation law, then plug in the explicit values of the weights $W^\s$ from \eqref{Wredef}, and finally renormalize the resulting functions $\F^s_{\lambda/\mu}(\kappa\mid\Xi,\S)$ to obtain the functions $\F_{\lambda/\mu}(\kappa\mid\Xi,\S)$ using \eqref{Frenorm}.
\end{proof} 
Note that the square roots $ \(\frac{s_i\xi_i}{s_{i-1}\xi_{i-1}}\)^{\mu_i/2}$ in the expression from Proposition \ref{rowValue} are due to the renormalization \eqref{Frenorm}, which involves square roots.
 
 \begin{cor}
 For a pair of partitions $\lambda,\mu$ the function $\F_{\lambda/\mu}(\kappa_1,\dots, \kappa_n\mid\Xi,\S)$ vanishes unless for any $r\geq 1$ we have
 \be
\mu'_r\leq\lambda'_r\leq\mu'_r+n.
 \ee
In particular,
\be
\F_{\lambda/\mu}(\kappa_1,\dots, \kappa_n\mid\Xi,\S)=0\qquad \text{unless}\ \mu\subset\lambda\ \text{and}\ l(\lambda)\leq l(\mu)+n.
\ee
 \end{cor}
 \begin{proof}
 We use the following elementary observation: if partitions $\mu$ and $\lambda$ interlace:
 \be
 \lambda_1\geq \mu_1\geq\lambda_2\geq\mu_2\geq\lambda_3\geq \dots,
 \ee
 then for any $r\geq 1$ we have
 \be
 \lambda'_r=\#\{i:\lambda_i\geq r\}\leq \#\{i:\mu_i\geq r\}+1= \mu'_r+1.
 \ee
 Now the claim follows from Propositions \ref{branching} and \ref{rowValue}: the function $\F_{\lambda/\mu}(\kappa_1,\dots, \kappa_n\mid\Xi,\S)$ vanishes unless there exists a sequence of partitions
 \be
 \mu=\nu^{(0)}\prec\nu^{(1)}\prec\dots\prec\nu^{(n)}=\lambda.
 \ee
 \end{proof}
 
 \begin{cor} We have the following stability property
 \be
 \F_\lambda(\kappa_1,\dots,\kappa_{n-1}, s_n\xi_n\mid\Xi,\S)=\F_\lambda(\kappa_1,\dots,\kappa_{n-1}\mid\Xi,\S).
 \ee
 \end{cor}
 \begin{proof}
From Proposition \ref{rowValue} we have
 \be
 \F_{\lambda}(\kappa\mid\Xi,\S)=\begin{cases}
 (-\kappa)^{|\lambda|}\xi_{0}^{-\lambda_1}\dfrac{(s_1\xi_1/\kappa;q)_{\lambda_1}(q;q)_{\lambda_1}}{(q;q)_{\lambda_1}(s_1^2;q)_{\lambda_1}}\quad&\text{if}\ \lambda_2=\lambda_3=\dots=0,\\
 0\quad &\text{otherwise}.
 \end{cases}
 \ee
 After replacing $\Xi, \S$ by $\tau^{n-1}_\S\Xi, \tau^{n-1}_\Xi\S$ and setting $\kappa=s^{(n-1)}_1\xi^{(n-1)}_1=s_n\xi_n$, we get
 \be
 \F_{\lambda}(s_n\xi_n\mid\tau^{n-1}_\S\Xi,\tau^{n-1}_\Xi\S)=\1_{\lambda=\varnothing}.
 \ee
The claim now follows from the branching rule of Proposition \ref{branching} for $m=n-1$.
 \end{proof}
 
 \begin{theo}
 The functions $\F_{\lambda/\mu}(\kappa_1,\dots,\kappa_n\mid\Xi,\S)$ are symmetric polynomials in $\kappa_1,\dots, \kappa_n$.
 \end{theo}
 \begin{proof}
 The functions $\F_{\lambda/\mu}(\kappa_1,\dots,\kappa_n\mid\Xi,\S)$ were observed to be polynomials in Section \ref{rowoprep}. To prove that they are symmetric, we use Proposition \ref{Foperator} to write $\F^{\s}_{\lambda/\mu}(\kappa_1,\dots,\kappa_n\mid\Xi,\S)$ as a matrix coefficient of a composition of the row operators $\C(\kappa\mid\tau^i_\S\Xi,\tau^i_\Xi\S)$ with shifted inhomogeneities. Proposition \ref{commutation} shows that this composition is invariant under simple transpositions exchanging $\kappa_i$ and $\kappa_{i+1}$, which implies the claim.
 \end{proof}
  Note that, via Proposition \ref{FFoperatorProp} and the second relation from Proposition \ref{commutation}, the same argument can be applied to the functions $\FF_{\lambda/\mu}(u_1,\dots, u_n\mid\Xi,\S)$ to directly prove that they are symmetric.

For our next result we need dual functions defined by
\be
\FF^{*}_{\lambda'/\mu'}(u_1,\dots, u_n\mid\Xi,\S):=\frac{\cc_\S(\lambda)}{\cc_{\S}(\mu)}\FF_{\lambda'/\mu'}(u_1,\dots, u_n\mid\Xi,\S),
\ee
\begin{equation}
\label{dualFRenorm}
\FF^{\s*}_{\lambda'/\mu'}(u_1,\dots, u_n\mid\Xi,\S):=\frac{(\S)^{2\mu}}{(\S)^{2\lambda}}\frac{\cc_{\S}(\mu)}{\cc_\S(\lambda)}\FF^{\s}_{\lambda'/\mu'}(u_1,\dots, u_n\mid\Xi,\S)=\frac{(-\S)^{\mu}}{(-\S)^{\lambda}}\frac{\cc_{\S}(\mu)}{\cc_{\S}(\lambda)}\FF^{*}_{\lambda'/\mu'}(u_1,\dots, u_n\mid\Xi,\S).
\end{equation}
In view of Proposition \ref{FFoperatorProp}, we have
\begin{equation}
\label{FFDualOperator}
\FF^{\s*}_{\lambda'/\mu'}(u_1,\dots, u_n\mid\Xi,\S)=\langle\mu|\tB^*(u_n\mid\Xi,\S)\dots\tB^*(u_2\mid\Xi,\S)\tB^*(u_1\mid\Xi,\S)|\lambda\rangle.
\end{equation}

\begin{theo}
 \label{dualCauchyTheo}
 For any partitions $\mu,\nu$ the following identity holds
 \begin{multline*}
 \sum_\lambda \FF_{\lambda'/\nu'}^*(u_1,\dots, u_m\mid\Xi,\S)\F_{\lambda/\mu}(\kappa_1,\dots, \kappa_n\mid\Xi,\S)\\
 =\prod_{i=1}^n\prod_{j=1}^m\frac{1-u_j\kappa_i}{1-u_j\xi_is_i}\sum_{\lambda}\F_{\nu/\lambda}(\kappa_1,\dots, \kappa_n\mid\Xi,\S)\FF^*_{\mu'/\lambda'}(u_1,\dots, u_m\mid\tau^n_\S\Xi,\tau^n_\Xi\S),
 \end{multline*}
 with both sums above having finitely many nonzero terms.
 \end{theo}
 \begin{proof}
 Note that by \eqref{Frenorm} and \eqref{dualFRenorm}, Theorem \ref{dualCauchyTheo} is equivalent to the following identity: 
  \begin{multline*}
 \sum_\lambda \FF^{\s*}_{\lambda'/\nu'}(u_1,\dots, u_m\mid\Xi,\S)\F^\s_{\lambda/\mu}(\kappa_1,\dots, \kappa_n\mid\Xi,\S)\\
 =\prod_{i=1}^n\prod_{j=1}^m\frac{1-u_j\kappa_i}{1-u_j\xi_is_i}\sum_{\lambda}\F^\s_{\nu/\lambda}(\kappa_1,\dots, \kappa_n\mid\Xi,\S)\FF^{\s*}_{\mu'/\lambda'}(u_1,\dots, u_m\mid\tau^n_\S\Xi,\tau^n_\Xi\S).
 \end{multline*}
So we will actually prove this identity instead of the identity in the theorem. The proof follows a standard outline given in \cite[Theorem 7.3]{BW17}: We rewrite the matrix element
\begin{multline*}
\mathcal E_{\mu}^\nu(\kappa_1, \dots, \kappa_n; u_1, \dots, u_m\mid\Xi,\S):=\langle \nu|\tB^*(u_m\mid\Xi,\S)\dots\tB^*(u_1\mid\Xi,\S)\\
\times \C(\kappa_1\mid\tau_\S\Xi,\tau_\Xi\S)\C(\kappa_2\mid\tau^2_\S\Xi,\tau^2_\Xi\S)\dots\C(\kappa_n\mid\tau^{n}_\S\Xi,\tau^{n}_\Xi\S)|\mu\rangle
\end{multline*}
in two different ways, corresponding to the two sides of the claim.

For the left-hand side we insert $\sum_{\lambda}|\lambda\rangle\langle\lambda|={\mathrm{id}}$ between the operators $\C$ and $\tB^*$ to get
\be
\mathcal E_{\mu}^\nu(\kappa_1, \dots, \kappa_n; u_1, \dots, u_m\mid\Xi,\S)=\sum_\lambda \FF^{\s*}_{\lambda'/\nu'}(u_1,\dots, u_m\mid\Xi,\S)\F^\s_{\lambda/\mu}(\kappa_1,\dots, \kappa_n\mid\Xi,\S),
\ee
where we have used Proposition \ref{Foperator} and relation \eqref{FFDualOperator}.

On the other hand, set 
\be
\C^{(n)}(\kappa_1,\dots,\kappa_n\mid\Xi,\S):=\C(\kappa_1\mid\tau_\S\Xi,\tau_\Xi\S)\C(\kappa_2\mid\tau^2_\S\Xi,\tau^2_\Xi\S)\dots\C(\kappa_n\mid\tau^n_\S\Xi,\tau^n_\Xi\S)
\ee
and iterate the exchange relation from Proposition \ref{exchange} to get
\be
\tB^*(u\mid\Xi,\S)\C^{(n)}(\kappa_1,\dots, \kappa_n\mid\Xi,\S)=\prod_{i=1}^n\frac{1-u\kappa_i}{1-u\xi_{i}s_{i}}\, \C^{(n)}(\kappa_1,\dots,\kappa_n\mid\Xi,\S)\tB^*(u\mid\tau^n_\S\Xi,\tau^n_\Xi\S).
\ee
Hence,
\begin{multline*}
\mathcal E_{\mu}^\nu(\kappa_1, \dots, \kappa_n; u_1, \dots, u_m\mid\Xi,\S)=\prod_{i=1}^n\prod_{j=1}^m\frac{1-u_j\kappa_i}{1-u_j\xi_is_i}\\
\times\langle \nu|\C^{(n)}(\kappa_1,\dots, \kappa_n\mid\Xi,\S)\tB^*(u_m\mid\tau^n_\S\Xi,\tau^n_\Xi\S)\dots\tB^*(u_1\mid\tau^n_\S\Xi,\tau^n_\Xi\S)|\mu\rangle.
\end{multline*}
Now we can again insert $\sum_{\lambda}|\lambda\rangle\langle\lambda|=\mathrm{id}$ and obtain
\begin{multline*}
\mathcal E_{\mu}^\nu(\kappa_1, \dots, \kappa_n; u_1, \dots, u_m\mid\Xi,\S)\\
=\prod_{i=1}^n\prod_{j=1}^m\frac{1-u_j\kappa_i}{1-u_j\xi_is_i}\sum_{\lambda}\F^{\s}_{\nu/\lambda}(\kappa_1,\dots, \kappa_n\mid\Xi,\S)\FF^{\s*}_{\mu'/\lambda'}(u_1,\dots, u_m\mid\tau^n_\S\Xi,\tau^n_\Xi\S).
\end{multline*}
 \end{proof}

By setting $\mu=\varnothing$ or $(\mu,\nu)=(\varnothing, \varnothing)$ and noting that
\be
\F_{\varnothing/\lambda}(\kappa_1,\dots,\kappa_n\mid\Xi,\S)=\FF_{\varnothing/\lambda}(u_1,\dots,u_m\mid\Xi,\S)=\1_{\lambda=\varnothing},
\ee
Theorem \ref{dualCauchyTheo} immediately leads to the following results: 
\begin{cor}[Pieri rule] For any partition $\nu$ we have
\be
\prod_{i=1}^n\prod_{j=1}^m\frac{1-u_j\kappa_i}{1-u_j\xi_is_i}\, \F_{\nu}(\kappa_1,\dots, \kappa_n\mid\Xi,\S)=\sum_\lambda \FF_{\lambda'/\nu'}^*(u_1,\dots, u_m\mid\Xi,\S)\F_{\lambda}(\kappa_1,\dots, \kappa_n\mid\Xi,\S).
 \ee
\end{cor}
\begin{cor}[Dual Cauchy identity]\label{dualCauchyCor} We have
 \be
 \sum_\lambda \FF_{\lambda'}^*(u_1,\dots, u_m\mid\Xi,\S)\F_{\lambda}(\kappa_1,\dots, \kappa_n\mid\Xi,\S)=\prod_{i=1}^n\prod_{j=1}^m\frac{1-u_j\kappa_i}{1-u_j\xi_is_i}.
 \ee
\end{cor}

\subsection{Integral representation} We can use the dual Cauchy identity of Corollary \ref{dualCauchyCor} and previously known results about inhomogeneous (stable) spin Hall-Littlewood functions to obtain a multivariate integral representation for the inhomogeneous spin $q$-Whittaker functions.

To formulate the result, let $0<|q|<1$  and assume that the variables $\kappa_1,\dots,\kappa_n$ and the parameters $\Xi=(\xi_0,\xi_1, \dots), \S=(s_0, s_1,\dots)$ are complex numbers such that there exists a positively-oriented simple contour $\mathcal C$ on the complex plane satisfying the following conditions:
\begin{itemize}
\item all points $\{0\}$ and $\{\xi_i^{-1}s_i\}_{i=1}^\infty$ are inside the contour $\mathcal C$;
\item all points $\{\xi_i^{-1}s_i^{-1}\}_{i=1}^\infty$ are outside of the contour $\mathcal C$;
\item the image $q\mathcal C$ of the contour $\mathcal C$ under the multiplication by $q$ is inside $\mathcal C$.
\end{itemize}
An example of such a configuration is sketched in Figure \ref{Cpict}.

\begin{figure}
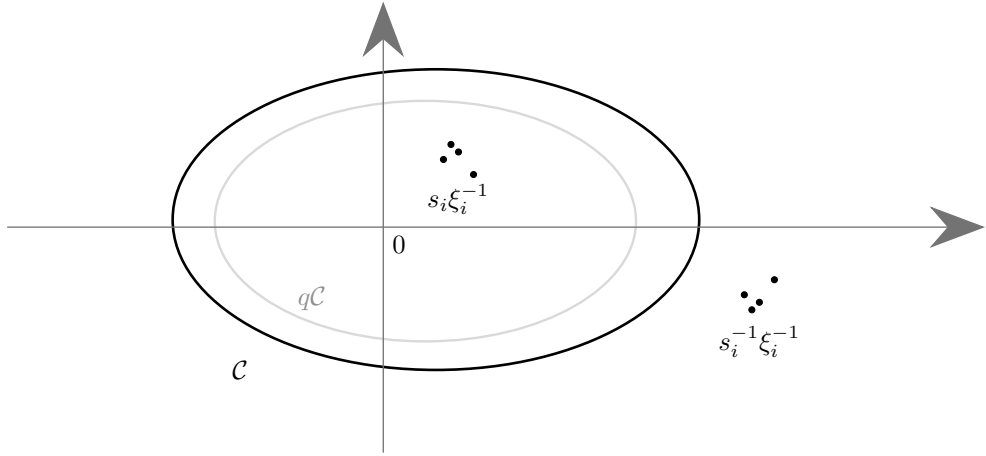

\tikz{1}{
\draw[line width=1pt, draw=black] (0.7,0.1) ellipse (3.5cm and 2cm);
\draw[line width=0.9pt, draw=lgray] (0.56,0.08) ellipse (2.8cm and 1.6cm);
\draw[line width=0.5pt, draw=dgray, arrows={->[scale=4]}] (-5,0) -- (8,0);
\draw[line width=0.5pt, draw=dgray, arrows={->[scale=4]}] (0,-3) -- (0,3);
\foreach \point in {(1,1), (0.9,1.1), (0.8,0.9), (1.2, 0.7)}
\draw[fill=black] \point circle (0.04);
\node[below] at (1,0.7) {$s_i\xi^{-1}_i$};
\foreach \point in {(5,-1), (4.9,-1.1), (4.8,-0.9), (5.2, -0.7)}
\draw[fill=black] \point circle (0.04);
\node[below] at (5,-1.2) {$s^{-1}_i\xi^{-1}_i$};
\node[] at (-1.9,-1.9) {$\mathcal C$};
\node[below right] at (0,0) {$0$};
\node[text=dgray!80] at (-0.95,-0.95) {$q\mathcal C$};
}
\caption{\label{Cpict}A possible of configuration of $\{\xi_i\s_i\}_{i=1}^\infty$, $\{\xi_i^{-1}\s_i\}_{i=1}^\infty$ and $\mathcal C$ for Theorem \ref{integral}.}
\end{figure}

\begin{theo}\label{integral}
Under the above assumptions, we have
\begin{multline*}
\F_{\mu}(\kappa_1,\dots, \kappa_n\mid\Xi,\S)\\
=\oint_{\mathcal C}\frac{dz_1}{2\pi\i z_1}\dots\oint_{\mathcal C}\frac{dz_k}{2\pi\i z_k}\prod_{\alpha<\beta}\frac{z_\alpha-z_\beta}{z_\alpha-qz_\beta}\prod_{\alpha=1}^k\(\frac{\xi_0^{-1}}{z_\alpha-\xi^{-1}_{\mu'_\alpha}s_{\mu'_\alpha}}\prod_{j=1}^{\mu'_\alpha-1}\frac{1-s_j\xi_jz_\alpha}{z_\alpha\xi_j-s_j}\prod_{i=1}^n\frac{1-z_\alpha\kappa_i}{1-z_\alpha\xi_is_i}\),
\end{multline*}
where $k=\mu_1$.
\end{theo}
\begin{proof} We will use the orthogonality of $\FF_{\lambda}$, \emph{cf.} \cite[Theorem 7.4]{BP16b}: for any pair of partitions $\lambda,\mu$ such that $l(\mu)\leq L$ we have
\begin{equation}
\label{FFortho}
\frac{(q;q)_{L-l(\mu)}\cc_{\S}(\lambda)}{(1-q)^L}\oint_{\mathcal C}\frac{dz_1}{2\pi\i}\dots\oint_{\mathcal C}\frac{dz_L}{2\pi\i}\prod_{\alpha<\beta}\frac{z_\alpha-z_\beta}{z_\alpha-qz_\beta}\FF_{\lambda}(z_1,\dots, z_L\mid\Xi,\S)\prod_{\alpha=1}^Lz_\alpha^{-1}\widetilde{\varphi}_{\mu_\alpha}(z_\alpha^{-1}\mid\overline{\Xi}, \S)=\1_{\lambda=\mu},
\end{equation}
where 
\be
\overline{\Xi}=(\xi_0^{-1},\xi_1^{-1},\xi_2^{-1},\dots)
\ee
and the functions $\widetilde{\varphi}_k(u\mid\Xi,\S)$ were defined at the end of Section \ref{HLsect}.

Fixing a partition $\mu$, conjugating partitions in \eqref{FFortho}, setting $L=k=\mu_1=l(\mu')$ and multiplying by $\F_{\lambda}(\kappa_1,\dots, \kappa_n\mid\Xi,\S)$, we get
\begin{multline*}
\frac{1}{(1-q)^k}\oint_{\mathcal C}\frac{dz_1}{2\pi\i}\dots\oint_{\mathcal C}\frac{dz_k}{2\pi\i}\prod_{\alpha<\beta}\frac{z_\alpha-z_\beta}{z_\alpha-qz_\beta}\FF^*_{\lambda'}(z_1,\dots, z_k\mid\Xi,\S)\F_{\lambda}(\kappa_1,\dots, \kappa_n\mid\Xi,\S)\prod_{\alpha=1}^kz_\alpha^{-1}\widetilde{\varphi}_{\mu'_\alpha}(z_\alpha^{-1}\mid\overline{\Xi}, \S)\\
=\1_{\lambda=\mu}\cdot\F_{\lambda}(\kappa_1,\dots, \kappa_n\mid\Xi,\S).
\end{multline*}
Summing over $\lambda$ and using Corollary \ref{dualCauchyCor}, we arrive at the integral representation for $\F_{\mu}(\kappa_1,\dots, \kappa_n\mid\Xi,\S)$.
\end{proof}

\section{A Cauchy identity}\label{cauchySection}
In this section we use fusion of $\tB^*$-operators to get a Cauchy-like identity between inhomogeneous spin $q$-Whittaker functions $\F_{\lambda/\mu}$ similar to \cite[Corollary 4.10]{BP16b} and \cite[Theorem 7.1]{BW17}. However, during the proof we will lose the degrees of freedom corresponding to the parameters $\Xi$, proving an identity involving only functions with $\Xi=\S$ and $\Xi=\overline{\S}$. Recall that for a sequence $\S=(s_0,s_1,\dots)$ we set
\be
\overline{\S}:=(s_0^{-1},s_1^{-1},s_2^{-1},\dots).
\ee

\subsection{$\Xi=\S$ and $\Xi=\bar\S$ specializations.} We start with the description of a simplification of functions $\F_{\lambda/\mu}(\kappa_1,\dots, \kappa_n|\Xi,\S)$ when $\Xi=\S$ and $\Xi=\overline\S$. In these two cases, the parameters of the vertex weights in the partition function $\ZW_{\lambda/\mu}^{(a_1, \dots, a_n)}(\kappa_1,\dots, \kappa_n\mid\Xi,\S)$ degenerate to $(\frac{s_{i+j}}{\sqrt{\kappa_i}}, s_{i+j})$ for $\Xi=\S$, and to $(\kappa_i^{-\frac{1}{2}},s_j)$ for $\Xi=\bar\S$, where $i$ and $j$ denote the row and the column of the corresponding vertex. These parameters are depicted in Figure \ref{partitionFunctionSpec}.

\begin{figure}
\tikz{1}{
	\foreach\y in {2,...,5}{
		\draw[fused*] (1,\y) -- (5.5,\y);
		\node[right] at (5.5,\y) {$\dots$};
		\draw[fused] (6.1,\y) -- (7,\y);
	}
	\foreach\x in {1,...,2}{
		\draw[fused] (2*\x,1) -- (2*\x,6);
	}
	\node[above right] at (2,5) {\tiny $\(\frac{s_2}{\sqrt{\kappa_1}}, s_2\)$};
	\node[above right] at (4,5) {\tiny $\(\frac{s_3}{\sqrt{\kappa_1}}, s_3\)$};
	
	\node[above right] at (2,4) {\tiny $\(\frac{s_3}{\sqrt{\kappa_2}}, s_3\)$};
	\node[above right] at (4,4) {\tiny $\(\frac{s_4}{\sqrt{\kappa_2}}, s_4\)$};
	
	\node[above right] at (2,2) {\tiny $\(\frac{s_{n+1}}{\sqrt{\kappa_n}}, s_{n+1}\)$};
	\node[above right] at (4,2) {\tiny $\(\frac{s_{n+2}}{\sqrt{\kappa_n}}, s_{n+2}\)$};
	\node[above] at (5.5,6) {$\cdots$};
	\node[above] at (4,6) {$m_2(\lambda')$};
	\node[above] at (2,6) {$m_1(\lambda')$};
	\node[left] at (1,2) {$a_n$};
	\node[left] at (1,3) {$\vdots$};
	\node[left] at (1,4) {$a_2$};
	\node[left] at (1,5) {$a_1$};
	\node[below] at (5.5,1) {$\cdots$};
	\node[below] at (4,1) {$m_2(\mu')$};
	\node[below] at (2,1) {$m_1(\mu')$};
	\node[right] at (7,2) {$0$};
	\node[right] at (7,3) {$\vdots$};
	\node[right] at (7,4) {$0$};
	\node[right] at (7,5) {$0$};
}
\hspace{1cm}
\tikz{1}{
	\foreach\y in {2,...,5}{
		\draw[fused*] (1,\y) -- (5.5,\y);
		\node[right] at (5.5,\y) {$\dots$};
		\draw[fused] (6.1,\y) -- (7,\y);
	}
	\foreach\x in {1,...,2}{
		\draw[fused] (2*\x,1) -- (2*\x,6);
	}
	\node[above right] at (2,5) {\tiny $\(\kappa_1^{-1/2}, s_1\)$};
	\node[above right] at (4,5) {\tiny $\(\kappa_1^{-1/2}, s_2\)$};
	
	\node[above right] at (2,4) {\tiny $\(\kappa_2^{-1/2}, s_1\)$};
	\node[above right] at (4,4) {\tiny $\(\kappa_2^{-1/2}, s_2\)$};
	
	\node[above right] at (2,2) {\tiny $\(\kappa_n^{-1/2}, s_1\)$};
	\node[above right] at (4,2) {\tiny $\(\kappa_n^{-1/2}, s_2\)$};
	\node[above] at (5.5,6) {$\cdots$};
	\node[above] at (4,6) {$m_2(\lambda')$};
	\node[above] at (2,6) {$m_1(\lambda')$};
	\node[left] at (1,2) {$a_n$};
	\node[left] at (1,3) {$\vdots$};
	\node[left] at (1,4) {$a_2$};
	\node[left] at (1,5) {$a_1$};
	\node[below] at (5.5,1) {$\cdots$};
	\node[below] at (4,1) {$m_2(\mu')$};
	\node[below] at (2,1) {$m_1(\mu')$};
	\node[right] at (7,2) {$0$};
	\node[right] at (7,3) {$\vdots$};
	\node[right] at (7,4) {$0$};
	\node[right] at (7,5) {$0$};
}
\caption{\label{partitionFunctionSpec} Partition functions $\ZW^{(a_1,\dots, a_n)}_{\lambda/\mu}(\kappa_1,\dots, \kappa_n\mid\S,\S)$ and $\ZW^{(a_1,\dots, a_n)}_{\lambda/\mu}(\kappa_1,\dots, \kappa_n\mid\bar\S,\S)$}
\end{figure}

Note that the action of the mixed shift operators $\tau_\Xi$ and $\tau_\S$ also simplifies:
\be
\tau_{\S}\S=\tau \S,\qquad \tau_{\S}\bar\S=\bar \S, \qquad \tau_{\,\overline{\S}}\S=\S,
\ee
where $\tau$ denotes the ordinary shift operator
\be
\tau^i\S:=(s_i, s_{i+1}, s_{i+2},\dots),\quad i\geq 0.
\ee
Then the functions $\F_{\lambda/\mu}(\kappa_1, \dots, \kappa_n\mid\Xi,\S)$ simplify to
\begin{multline*}
\F_{\lambda/\mu}(\kappa_1, \dots, \kappa_n\mid\S,\S)=\frac{(-\tau^n\S)^{\mu}}{(-\S)^{\lambda}}\frac{\cc_{\tau^n\S}(\mu)}{\cc_\S(\lambda)}\F^{\s}_{\lambda/\mu}(\kappa_1,\dots, \kappa_n\mid\S,\S)\\
=\frac{(-\tau^n\S)^{\mu}}{(-\S)^{\lambda}}\frac{\cc_{\tau^n\S}(\mu)}{\cc_\S(\lambda)}\sum_{a_1,\dots, a_n}\(\prod_{i=1}^n\kappa_i^{a_i}\frac{(s^2_i/\kappa_i;q)_{a_i}}{(q;q)_{a_i}}\)\ZW_{\lambda/\mu}^{(a_1,\dots, a_n)}(\kappa_1,\dots,\kappa_n\mid\S,\S).
\end{multline*}
\begin{multline*}
\F_{\lambda/\mu}(\kappa_1, \dots, \kappa_n\mid\bar\S,\S)=\frac{(-\S)^{\mu}}{(-\S)^{\lambda}}\frac{\cc_{\S}(\mu)}{\cc_\S(\lambda)}\F^{\s}_{\lambda/\mu}(\kappa_1,\dots, \kappa_n\mid\bar\S,\S)\\
=\frac{(-\S)^{\mu}}{(-\S)^{\lambda}}\frac{\cc_{\S}(\mu)}{\cc_\S(\lambda)}\sum_{a_1,\dots, a_n}\(\prod_{i=1}^n\(\kappa_is_0^2\)^{a_i}\frac{(\kappa_i^{-1};q)_{a_i}}{(q;q)_{a_i}}\)\ZW_{\lambda/\mu}^{(a_1,\dots, a_n)}(\kappa_1,\dots,\kappa_n\mid\bar\S,\S).
\end{multline*}

\subsection{Cauchy identity.} For any pair of partitions $\lambda,\mu$ define \emph{dual inhomogeneous spin $q$-Whittaker} polynomials by (\emph{cf.} \eqref{Frenorm})
\be
\F^{*}_{\lambda/\mu}(\kappa_1,\dots, \kappa_n\mid\Xi,\S):=\frac{\cc_\S(\lambda)}{\cc_{\tau^n_\Xi\S}(\mu)}\F_{\lambda/\mu}(\kappa_1,\dots, \kappa_n\mid\Xi,\S).
\ee
Additionally, define
\be
\F^{\s*}_{\lambda/\mu}(\kappa_1,\dots, \kappa_n\mid\Xi,\S):=\frac{(\tau^n_\Xi\S)^{2\mu}}{(\S)^{2\lambda}}\frac{\cc_{\tau^n_\Xi\S}(\mu)}{\cc_\S(\lambda)}\F^{\s}_{\lambda/\mu}(\kappa_1,\dots, \kappa_n\mid\Xi,\S)=\frac{(-\tau^n_\Xi\S)^{\mu}}{(-\S)^{\lambda}}\frac{\cc_{\tau^n_\Xi\S}(\mu)}{\cc_{\S}(\lambda)}\F^{*}_{\lambda/\mu}(\kappa_1,\dots, \kappa_n\mid\Xi,\S).
\ee
In view of Proposition \ref{Foperator}, we have
\be
\F^{\s*}_{\lambda/\mu}(\kappa_1, \dots, \kappa_n\mid\Xi,\S)=\langle\mu|\B^*(\kappa_n\mid\tau^{n-1}_{\S}\Xi,\tau^{n-1}_{\Xi}\S)\dots\B^*(\kappa_2\mid\tau_{\S}\Xi,\tau_{\Xi}\S)\B^*(\kappa_1\mid\Xi,\S)|\lambda\rangle.
\ee

\begin{theo} \label{Cauchy}The following identity holds in the space of formal power series in $\kappa_1,\dots, \kappa_n, \chi_1,\dots, \chi_m,q$ with rational in $\S$ coefficients:
\begin{multline}
\label{CauchyIdentity}
\sum_\lambda\F^*_{\lambda/\nu}(\chi_1,\dots, \chi_m\mid\bar\S,\S)\F_{\lambda/\mu}(\kappa_1,\dots, \kappa_n\mid\S,\S)\\
=\prod_{i=1}^n\prod_{j=1}^m\frac{(\kappa_i;q)_\infty(s_{i}^2\chi_j;q)_\infty}{(s_{i}^2;q)_\infty(\kappa_i\chi_j;q)_\infty}\sum_\lambda\F_{\nu/\lambda}(\kappa_1,\dots,\kappa_n\mid\S,\S)\F^*_{\mu/\lambda}(\chi_1,\dots,\chi_m\mid\tau^n\bar\S,\tau^n\S).
\end{multline}
\end{theo}
\begin{proof}
Similarly to the proof of Theorem \ref{dualCauchyTheo}, we will actually prove an equivalent identity 
\begin{multline}
\label{CauchyIdentityStoch}
\sum_\lambda\F^{\s*}_{\lambda/\nu}(\chi_1,\dots, \chi_m\mid\bar\S,\S)\F^\s_{\lambda/\mu}(\kappa_1,\dots, \kappa_n\mid\S,\S)\\
=\prod_{i=1}^n\prod_{j=1}^m\frac{(\kappa_i;q)_\infty(s_{i}^2\chi_j;q)_\infty}{(s_{i}^2;q)_\infty(\kappa_i\chi_j;q)_\infty}\sum_\lambda\F^{\s}_{\nu/\lambda}(\kappa_1,\dots,\kappa_n\mid\S,\S)\F^{\s*}_{\mu/\lambda}(\chi_1,\dots,\chi_m\mid\tau^n\bar\S,\tau^n\S),
\end{multline}
which is obtained by rescaling $\F$, $\F^*$ to $\F^{\s}$, $\F^{\s*}$. 

Following the proof of Theorem \ref{dualCauchyTheo}, let
\be
\C^{(n)}(\kappa_1,\dots,\kappa_n\mid\S,\S)=\C(\kappa_1\mid\tau \S,\tau \S)\dots\C(\kappa_n\mid\tau^n\S,\tau^n\S).
\ee
Then the exchange relation of Proposition \ref{exchange} together with the relation $\tau_\S\S=\tau \S$ imply 
\begin{multline*}
\tB^*(1\mid\S,\S)\dots\tB^*(q^{J-1}\mid\S,\S)\C^{(n)}(\kappa_1,\dots,\kappa_n\mid\S,\S)\\
=\prod_{i=1}^n\frac{(\kappa_i;q)_J}{(s_{i}^2;q)_J}\C^{(n)}(\kappa_1,\dots,\kappa_n\mid\S,\S)\tB^*(1\mid\tau^n\S,\tau^n\S)\dots\tB^*(q^{J-1}\mid\tau^n\S,\tau^n\S).
\end{multline*}
Now we can use Proposition \ref{Bfusion} to get
\begin{equation*}
\B^*(q^J\mid\bar\S,\S)\C^{(n)}(\kappa_1,\dots,\kappa_n\mid\S,\S)=
\prod_{i=1}^n\frac{(\kappa_i;q)_\infty(s_{i}^2q^J;q)_\infty}{(s_{i}^2;q)_\infty(\kappa_iq^J;q)_\infty}\C^{(n)}(\kappa_1,\dots,\kappa_n\mid\S,\S)\B^*(q^J\mid\tau^n\bar\S,\tau^n\S).
\end{equation*}
Iterating this equation $m$-times for $J=J_1,\dots, J_m$, we obtain
\begin{multline*}
\B^*(q^{J_m}\mid\bar\S,\S)\dots\B^*(q^{J_1}\mid\bar\S,\S)\C^{(n)}(\kappa_1,\dots,\kappa_n\mid\S,\S)\\
=\prod_{i=1}^n\prod_{j=1}^m\frac{(\kappa_i;q)_\infty(s_{i}^2q^{J_j};q)_\infty}{(s_{i}^2;q)_\infty(\kappa_iq^{J_j};q)_\infty}\C^{(n)}(\kappa_1,\dots,\kappa_n\mid\S,\S)\B^*(q^{J_m}\mid\tau^n\bar\S,\tau^n\S)\dots\B^*(q^{J_1}\mid\tau^n\bar\S,\tau^n\S),
\end{multline*}
which, after taking the matrix element $\langle\nu|\cdot|\mu\rangle$, inserting a summation $\sum_\lambda|\lambda\rangle\langle\lambda|=id$, and using both identities from Proposition \ref{Foperator}, gives the claim of the theorem for $\chi_r=q^{J_r}$.

Up until now we have operated with finite sums of rational functions: one can see\footnote{for example, using Proposition \ref{Bfusion} and the discussion from Section \ref{linearCombSect}} that for $(\chi_1,\dots,\chi_m)=(q^{J_1},\dots, q^{J_m})$ with $\{J_r\}_{r=1}^m\in\Z_{\geq 1}$ both sums in \eqref{CauchyIdentity} have finitely many nonzero terms. But for general parameters $\chi$ the sum in the left-hand side is infinite, so we need to be careful with the general statement. 

We say that a formal power series $f\in K[[x_1,\dots, x_r]]$ over a field $K$ is of \emph{order at least $N$} if coefficients of all monomials $x_1^{a_1}\dots x_r^{a_r}$ such that $\sum_{i=1}^ra_i<N$ are equal to $0$. We claim that $\F_{\lambda/\mu}(\kappa_1,\dots, \kappa_n\mid\Xi,\S)$ is a formal power series in $K[[\kappa_1,\dots, \kappa_n,q]]$ of order at least $\lambda_1-\mu_1-n$, where $K=\mathbb Q_{(\Xi,\S)}$ is the field of rational functions in variables $\Xi=(\xi_0,\xi_1,\dots)$ and  $\S=(s_0, s_1,\dots)$. To see that, note that the matrix coefficients of $\T_a$ are formal power series in $\kappa,q$, that is
\be
\langle \lambda|\T_a(\kappa\mid\Xi,\S)|\mu\rangle\in K[[\kappa,q]],
\ee
while the coefficient of $\T_a$ in the definition \eqref{defC} of the operator $\C(\kappa\mid\Xi,\S)$ has order at least $a-1$ with respect to $\kappa,q$:
\be
\kappa^as_0^a\xi_0^{-a}\frac{(\kappa^{-1} s_0\xi_0;q)_a}{(q;q)_a}=\frac{\prod_{i=1}^a(\kappa\xi^{-1}_0-s^2_0q^{i-1})}{(q;q)_a}.
\ee
Hence, the matrix element
\be
\langle \lambda|\C(\kappa\mid\Xi,\S)|\mu\rangle\in K[[\kappa,q]]
\ee
has order at least $\lambda_1-\mu_1-1$. Therefore, by Proposition \ref{Foperator}, the function $\F_{\lambda/\mu}(\kappa_1,\dots,\kappa_n\mid\Xi,\S)$ is a formal power series in $\kappa_1,\dots, \kappa_n,q$ of order at least $\lambda_1-\mu_1-n$.

Now we can finish the proof. The right-hand side of \eqref{CauchyIdentity} is clearly a formal power series in $\chi,\kappa,q$. For the sum on the left-hand side, the term corresponding to $\lambda$ has order at least $2\lambda_1-\mu_1-\nu_1-n-m$. Moreover, the summand corresponding to $\lambda$ vanishes if $l(\lambda)$ is greater than either $l(\mu)+n$ or $l(\nu)+m$. Overall this means that all but finitely many summands in the left-hand side have order at least $N$ for any $N\geq 1$. Hence both sides of \eqref{CauchyIdentity} are elements of $K[[\kappa_1, \dots, \kappa_n, \chi_1,\dots, \chi_m,q]]$, which are equal after substitution $(\chi_1,\dots,\chi_m)=(q^{J_1},\dots, q^{J_m})$ for any integer vector $\(J_r\)_{r=1}^m\in\Z^m_{\geq 1}$\footnote{Note that such a substitution is well-defined and it behaves well with respect to the convergence of formal power series, as well as the notion of order defined earlier. Indeed, the substitution cannot reduce the order of a formal power series, hence the substitution preserves convergence.}. An iterative application of the following elementary statement finishes the proof.
\begin{lem}[{\cite[Lemma 3.2]{Ste88}}]
 Let $F(z,q)$ and $G(z,q)$ be formal power series satisfying $F(q^J,q)=G(q^J,q)$ for infinitely many integers $J\geq 1$. Then $F(z,q)=G(z,q)$.
\end{lem}
\end{proof}

\begin{rem}
The convergence from Theorem \ref{Cauchy} also holds numerically if 
\be
|\kappa_1|, \dots, |\kappa_n|\leq\kappa,\qquad |\chi_1|,\dots, |\chi_m|\leq\chi,\qquad 0<|q|,\kappa\chi<1.
\ee
For the right-hand side of \eqref{CauchyIdentity} the convergence is immediate, while for the left-hand side this can be verified as follows: Using Proposition \ref{rowValue}, the one variable functions $\F_{\lambda/\mu}(\kappa_i\mid\Xi,\S)$ can be bounded by $C |\kappa_i|^{|\lambda|}\frac{\(\tau_\S\Xi\)^{\mu}}{\(\Xi\)^{\lambda}}$, where $C$ is a positive constant not depending on $\lambda,\mu$ (to reach this bound it is enough to bound a finite number of possibly appearing $q$-Pochhammer symbols $(x,q)_a$ by a uniform constant and rearrange the remaining terms). Then, by the branching rule of Proposition \ref{branching}, we obtain an upper bound 
\be
|\F_{\lambda/\mu}(\kappa_1, \dots, \kappa_n\mid\Xi,\S)|\leq C |\lambda|^d \kappa^{|\lambda|}/(\Xi)^\lambda,
\ee
for some $C,d>0$ not depending on $\lambda$. This bound is sufficient for proving the convergence of the left-hand side of \eqref{CauchyIdentity}. 
\end{rem}

\begin{rem} \label{conjecture}
One can note a similarity between Theorem \ref{Cauchy} and the $\Xi=\S$ case of the Cauchy identity for inhomogeneous spin Hall-Littlewood functions \cite[Corollary 4.10]{BP16b}. Thus, one might hope that it is possible to remove the restriction $\Xi=\S$, although we don't know how to do that.
\end{rem}

By setting $\mu$ and $\nu$ equal to $\varnothing$ one immediately obtains a Cauchy type summation identity:
\begin{cor}\label{CauchyCor}The following identity holds in the space of formal power series in $\kappa_1,\dots, \kappa_n, \chi_1,\dots, \chi_m,q$ with rational in $\S$ coefficients:
\begin{equation*}
\sum_\lambda\F^*_{\lambda}(\chi_1,\dots, \chi_m\mid\bar\S,\S)\F_{\lambda}(\kappa_1,\dots, \kappa_n\mid\S,\S)=\prod_{i=1}^n\prod_{j=1}^m\frac{(\kappa_i;q)_\infty(s_{i}^2\chi_j;q)_\infty}{(s_{i}^2;q)_\infty(\kappa_i\chi_j;q)_\infty}.
\end{equation*}
The same identity holds numerically for $0<|q|,|\kappa_1|,\dots,|\kappa_n|,|\chi_1|,\dots,|\chi_m|<1$.
\end{cor}

\begin{rem}\label{qGauss}
As noted in \cite[Section 7.2]{BW17}, in the simplest case when $n=m=1$ Corollary \ref{CauchyCor} reads
\be
\sum_{a\geq 0} (-\kappa/ s_0)^a\frac{(\kappa^{-1}s_1^2;q)_a}{(s_1^2;q)_a} \cdot (-\chi s_0)^a\frac{(\chi^{-1};q)_a}{(q;q)_a}=\frac{(\kappa;q)_\infty(s_{1}^2\chi;q)_\infty}{(s_{1}^2;q)_\infty(\kappa\chi;q)_\infty},
\ee
which is equivalent to the \emph{$q$-Gauss summation identity}. Actually, we have
\be
\sum_{a\geq 0} (-\kappa/ \xi_0)^a\frac{(\kappa^{-1}s_1\xi_1;q)_a}{(s_1^2;q)_a} \cdot (-\chi \xi_0)^a\frac{(\chi^{-1}s_1\xi_1^{-1};q)_a}{(q;q)_a}=\frac{(\kappa s_1/\xi_1;q)_\infty(\chi s_1\xi_1;q)_\infty}{(s_{1}^2;q)_\infty(\kappa\chi;q)_\infty},
\ee
which, along the lines of Remark \ref{conjecture}, may correspond to a hypothetical more general version of Theorem \ref{Cauchy} with arbitrary $\Xi$.
\end{rem}

\end{document}